\documentclass[11pt,psamsfonts, oneside, reqno]{amsart}
\newcommand{\hkra}{\hookrightarrow}

\newcommand{\dimp}{\Leftrightarrow}

\newcommand{\set}[1]{\left\{ #1 \right\}}

\newcommand{\cc}[1]{\overline{#1}}

\newcommand{\union}[2]{\bigcup\limits_{#1}{#2}}
\newcommand{\inter}[2]{\bigcap\limits_{#1}{#2}}
\newcommand{\sub}{\subset}

\newcommand{\ssuper}{\supsetneq}

\newcommand{\sm}{\ensuremath{\setminus}}

\newcommand{\s}[2]{\sum\limits_{#1}{#2} }
\newcommand{\p}[2]{\prod\limits_{#1}{#2} }
\newcommand{\g}{\circ}
\newcommand{\modulo}[2]{{\raisebox{.2em}{$#1$}\left/\raisebox{-.2em}{$#2$}\right.}}

\newcommand{\inv}{^{-1}}
\newcommand{\limes}[3]{\lim\limits_{#1 \rightarrow #2}{#3}}
\newcommand{\norm}[1]{\lVert#1\rVert}

\newcommand{\cl}{\colon}
\newcommand{\djun}[1]{\bigsqcup\limits_{#1}}

\newcommand{\lbr}[1]{\Bigl(#1\Bigr)}
\newcommand{\emst}{\emptyset}
\newcommand{\iI}{_{i \in I}}

\newcommand{\jN}{_{j\in \bN}}

\newcommand{\xra}{\xrightarrow}

\newcommand{\scA}{\mathscr{A}}

\newcommand{\scC}{\mathscr{C}}

\newcommand{\scF}{\mathscr{F}}
\newcommand{\scG}{\mathscr{G}}

\newcommand{\scJ}{\mathscr{J}}

\newcommand{\scP}{\mathscr{P}}

\newcommand{\scS}{\mathscr{S}}

\newcommand{\scV}{\mathscr{V}}

\newcommand{\normd}{\norm{\cdot}}

\newcommand{\conn}{\nabla}
\newcommand{\eva}{ev}
\newcommand{\ii}{\text{i}}

\newcommand{\lspan}[1]{\langle {#1}\rangle}
\newcommand{\pf}{\pitchfork}

\newcommand{\inn}{^\mathrm{o}}
\newcommand{\Ide}{Id}

\newcommand{\lc}{\underline}

\newcommand{\dul}{^\vee}
\newcommand{\vfc}[1]{[#1]^{\virt}}
\newcommand{\ide}{\normalfont\text{id}}
\newcommand{\pt}{\normalfont\text{pt}}
\newcommand{\im}{\normalfont\text{im}}
\newcommand{\reg}{\normalfont\text{reg}}

\newcommand{\GL}{\normalfont\text{GL}}
\newcommand{\PGL}{\normalfont\text{PGL}}

\newcommand{\Set}{\normalfont\text{Set}}
\newcommand{\coker}{\normalfont\text{coker}}
\newcommand{\pr}{\normalfont\text{pr}}

\newcommand{\Mbar}{\overline{\mathcal M}}

\newcommand{\vdim}{\normalfont\text{vdim}}
\newcommand{\delbar}{\bar\partial}
\newcommand{\del}{\partial}
\newcommand{\Hom}{\normalfont\text{Hom}}

\newcommand{\Aut}{\normalfont\text{Aut}}

\newcommand{\bB}{\mathbb{B}}
\newcommand{\bC}{\mathbb{C}}
\newcommand{\bD}{\mathbb{D}}
\newcommand{\bE}{\mathbb{E}}

\newcommand{\bK}{\mathbb{K}}

\newcommand{\bN}{\mathbb{N}}

\newcommand{\bP}{\mathbb{P}}
\newcommand{\bQ}{\mathbb{Q}}
\newcommand{\bR}{\mathbb{R}}

\newcommand{\bZ}{\mathbb{Z}}
\newcommand{\cA}{\mathcal{A}}
\newcommand{\cB}{\mathcal{B}}
\newcommand{\cC}{\mathcal{C}}

\newcommand{\cE}{\mathcal{E}}
\newcommand{\cF}{\mathcal{F}}
\newcommand{\cG}{\mathcal{G}}
\newcommand{\cH}{\mathcal{H}}
\newcommand{\cI}{\mathcal{I}}
\newcommand{\cJ}{\mathcal{J}}
\newcommand{\cK}{\mathcal{K}}
\newcommand{\cL}{\mathcal{L}}

\newcommand{\cN}{\mathcal{N}}
\newcommand{\cO}{\mathcal{O}}
\newcommand{\cP}{\mathcal{P}}

\newcommand{\cS}{\mathcal{S}}
\newcommand{\cT}{\mathcal{T}}
\newcommand{\cU}{\mathcal{U}}
\newcommand{\cV}{\mathcal{V}}
\newcommand{\cW}{\mathcal{W}}
\newcommand{\cX}{\mathcal{X}}
\newcommand{\cY}{\mathcal{Y}}
\newcommand{\cZ}{\mathcal{Z}}

\newcommand{\fF}{\mathfrak{F}}

\newcommand{\fJ}{\mathfrak{J}}

\newcommand{\fL}{\mathfrak{L}}
\newcommand{\fM}{\mathfrak{M}}

\newcommand{\fc}{\mathfrak{c}}

\newcommand{\ff}{\mathfrak{f}}
\newcommand{\fg}{\mathfrak{g}}

\newcommand{\fii}{\mathfrak{i}}
\newcommand{\fj}{\mathfrak{j}}

\newcommand{\fl}{\mathfrak{l}}

\newcommand{\fo}{\mathfrak{o}}
\newcommand{\fp}{\mathfrak{p}}
\newcommand{\fq}{\mathfrak{q}}

\newcommand{\fs}{\mathfrak{s}}

\newcommand{\virt}{\normalfont\text{vir}}
\newcommand{\rank}{\normalfont\text{rank}}

\newcommand{\wt}{\widetilde}

\newcommand{\wh}{\widehat}

\newcommand{\evai}{evi}
\newcommand{\evab}{evb}

\newcommand{\ind}{\normalfont\text{ind}}
\newcommand{\topw}{\Lambda^{\max}}
\newcommand{\orl}{\hspace{0.8pt}\fo}

\newcommand{\hodge}{\bE\dul}

\newcommand{\al}{\alpha}

\newcommand{\mfa}{\mathsf{a}}
\newcommand{\obs}{s}
\newcommand{\M}[1]{\Mbar_{#1}}

\newtheorem{intthm}{Theorem}

\newcommand{\stb}{\normalfont\text{st}}

\newcommand{\ov}[1]{\overline{#1}}
\newcommand{\cKc}{\cK^{\scale{\square}{0.6}}}

\newcommand{\scale}[2]{\scaleobj{#2}{#1}}
\newcommand{\cBc}{\cB^{\scale{\square}{0.6}}}
\newcommand{\cEc}{\cE^{\scale{\square}{0.6}}}
\newcommand{\cTc}{\cT^{\scale{\square}{0.6}}}

\usepackage[utf8]{inputenc}
\usepackage[english]{babel}
\usepackage{amsmath} 
\usepackage{amssymb}
\usepackage{amsfonts}
\usepackage{mathtools}
\usepackage{float}
\usepackage{amsthm}
\usepackage[shortlabels]{enumitem}
\usepackage[dvipsnames]{xcolor}
\usepackage{tikz-cd}
\usepackage{subfiles}
\usepackage[thinlines]{easytable}
\usepackage{scalerel}[2016/12/29]
\usetikzlibrary{tqft}
\usepackage[colorinlistoftodos]{todonotes}
\usepackage[mathscr]{euscript}
\usepackage{extarrows}
\usepackage[pagebackref=true]{hyperref}
\usepackage{comment}
\usepackage{abraces}
\usepackage{upgreek}

\usepackage{appendix}
\usepackage{bm}
\hypersetup{
    colorlinks=true,
    linkcolor=Blue,
    filecolor=red,      
    urlcolor=teal,
    citecolor=NavyBlue,}

\DeclareFontFamily{U}{rcjhbltx}{}
\DeclareFontShape{U}{rcjhbltx}{m}{n}{<->rcjhbltx}{}
\DeclareSymbolFont{hebrewletters}{U}{rcjhbltx}{m}{n}
\DeclareMathSymbol{\mem}{\mathord}{hebrewletters}{109}
\DeclareMathSymbol{\nun}{\mathord}{hebrewletters}{110}

\newsavebox{\pullback}
\sbox\pullback{%
	\begin{tikzpicture}%
		\draw (0,0) -- (1ex,0ex);%
		\draw (1ex,0ex) -- (1ex,1ex);%
\end{tikzpicture}}

\tikzset{
	immersion/.tip={Glyph[glyph math command=looparrowleft, swap]}
}

\newtheorem{theorem}{Theorem}[section]
\newtheorem{lemma}[theorem]{Lemma}
\newtheorem{cor}[theorem]{Corollary}
\newtheorem{proposition}[theorem]{Proposition}

\newtheorem{construction}[theorem]{Construction}
\newtheorem{assum}[theorem]{Assumption}

\theoremstyle{definition}
\newtheorem{definition}[theorem]{Definition}
\theoremstyle{remark}
\newtheorem{remark}[theorem]{Remark}

\newtheorem{example}[theorem]{Example}
\newtheorem{nota}[theorem]{Notation}
\newtheorem*{notation*}{Notation}

\numberwithin{equation}{section}

\makeatletter
\@addtoreset{equation}{section}
\@addtoreset{theorem}{section}
\makeatother

\DeclareMathOperator{\Pic}{Pic}

\newcommand{\Addresses}{{
  \bigskip
  \footnotesize
\textsc{A. Hirschi, Université Paris Cité, Sorbonne Université, CNRS, IMJ-PRG, F-75005 Paris, France}\par\nopagebreak
\text{ORCID}: \texttt{0000-0002-2392-7875}\\
  
  \textsc{K. Hugtenburg, School of Mathematics and Statistics, University of St Andrews}\par\nopagebreak
  \text{ORCID}: \texttt{0000-0002-7823-7126}
}}

\makeatletter 
\def\l@subsection{\@tocline{2}{0pt}{2pc}{6pc}{}} \makeatother

\usepackage{microtype} 
\begin{document}

\title[Open-closed Deligne--Mumford field theories I]{Open-closed Deligne--Mumford field theories: geometric foundations}
\author{Amanda Hirschi}\author{Kai Hugtenburg}

\begin{abstract} 
	We construct global Kuranishi charts for moduli spaces of pseudo-holomorphic maps of arbitrary genus with boundary on an embedded Lagrangian submanifold. We then build the geometric foundations required for obtaining compatible chain-level operations, which are employed in follow-up work to construct an open-closed Deligne--Mumford field theory.
\end{abstract}

\maketitle
\tableofcontents

\section{Introduction}

\subsection{Context}
The success of invariants based on curves with boundary has been clear since the construction of Lagrangian Floer theory, \cite{Flo88}, \cite{FOOO}, and the Fukaya category, \cite{Sei08}. However, while the construction of enumerative invariants based on closed curves in a closed symplectic manifold $(X,\omega)$ was established relatively early on, the same program in the open setting, that is, considering bordered curves with Lagrangian boundary conditions, was faced with two problems, the possible non-orientability of the relevant moduli spaces and the fact that they, respectively their compactifications have codimension $1$ strata.

There are several approaches to dealing with these boundary strata and the way they obstruct invariants. One way is to work in a situation where they do not occur. This was pursued by Liu in \cite{Liu20}, defining invariants from curves of any complexity in the presence of a circle action on the symplectic manifold that left the Lagrangian invariant. In a similar vein, \cite{Geo16} constructs invariants based on pseudo-holomorphic disks if the Lagrangian is the fixed point locus of an anti-symplectic involutions.
Another ansatz, restricted to dimension $4$ and $6$, is given in \cite{Wel13,Wel15}, following the seminal work \cite{Wel05}. The construction which can be described as `counting intelligently', that is, one defines an invariant explicitly in terms of certain counts and shows that the resulting numbers do not depend on the choices required for the count. While very geometric, this approach requires geometric assumptions on the Lagrangian, restricting the class of examples.

In \cite{Fu10,Fuk11}, Fukaya took a very different approach to construct genus $0$ open Gromov--Witten invariants. Concretely, he considers the algebraic structure arising from the boundary of the moduli space in its own right and extracts invariants from this structure for Lagrangians in Calabi-Yau $3$-folds. This viewpoint is further developed in \cite{FOOO,FOOO18,FOOOToric1,FOOO24}.
Building on their insights, Solomon--Tukachinsky established an algebraic framework to define genus $0$ open Gromov--Witten invariants in any dimension as long as the relevant moduli spaces were sufficiently regular, \cite{ST1,ST16,ST23}. Whereas the invariants defined by Solomon-Tukachinsky are real numbers, Haney, \cite{Han24}, constructs rational open Gromov--Witten invariants for monotone Lagrangians, based on a Morse model for the Fukaya algebra as discussed in \cite{CW17}. The main difficulty of this strategy is to construct compatible \emph{chain-level} operations based on moduli spaces of curves with boundary.

A fourth approach was outlined in \cite{Fuk06}. It is based on the observation that boundary degenerations of curves, when looking at how their boundary circles degenerate, correspond exactly to string topology operations such as the string bracket and cobracket. In the non-equivariant setting, this was proven by Irie in \cite{Iri18,Iri20}, showing that moduli spaces of disks define an Maurer-Cartan element for the BV algebra structure on the homology of the loop space of the Lagrangian. A similar approach is taken in \cite{ES24}, where they construct invariants in the skein module for Lagrangians in Calabi-Yau $3$-folds.

However, to work with moduli spaces of curves with boundary one has to deal with the lack of transversality that is ubiquitous in symplectic geometry outside of special cases. In particular, the boundary mentioned above should be understood in the derived sense - often the moduli space is highly singular. A very general framework, using the language of derived orbifolds, to describe these moduli spaces is developed in \cite{Par25}, while Kuranishi structures for moduli spaces of disks were developed in \cite{Fu10,Fuk11,FOOO16,FOOO18,FOOOToric1} in order to construct the chain-level operations described above, while \cite{Liu20} constructs local Kuranishi charts for moduli spaces of curves in higher genus. Other approaches are given by \cite{Jo19} using polyfolds and \cite{She15}, the methods of \cite{CM07} were adapted to disks. For closed holomorphic curves, the foundational work \cite{AMS21}, followed by \cite{AMS23}, \cite{HS22}, resolved these problems using the language of derived orbifolds/global Kuranishi charts. This paper extends these results to the case of holomorphic maps with a Lagrangian boundary condition.
 
\subsection{Main results}
This paper constructs global Kuranishi charts for moduli spaces of curves with boundary on a Lagrangian and describes how to get compatible systems of such charts.
The first result is a generalisation of the global Kuranishi chart construction of \cite{HS22} to open stable maps. To state it, let $(X,\omega)$ be a closed symplectic manifold and let $L \subset X$ be an embedded Lagrangian submanifold. Fix an almost complex structure $J$ tamed by $\omega$.
We denote by $\Mbar_{g,h;k,\ell}^{J,\beta}(X,L)$ the moduli space of stable $J$-holomorphic maps on bordered Riemann surfaces of genus $g$ equipped with $k$ interior marked points and $\ell$ boundary marked points as well as an ordering of the $h$ boundary components. See \cite{Rab24} for a similar construction in the case of discs.

\begin{intthm}\label{thm:open-gw-global-chart} For any $g,h \geq 0$ and $k,\ell_1,\dots,\ell_h\geq 0$ the following holds.
	\begin{enumerate}[label=\alph*),leftmargin=20pt,ref=\alph*]
		\item\label{gkc-existence} $\Mbar^{J,\beta}_{g,h;k,\ell}(X,L)$ admits a rel-$C^\infty$ global Kuranishi chart $\cK = (\cT, \cE, s)$ with boundary whose virtual dimension agrees with the expected dimension and which is unique up to equivalence of global Kuranishi charts.
		\item\label{gkc-orientation} The orientation sheaf of the global Kuranishi chart is canonically isomorphic to the orientation sheaf of the virtual tangent bundle of $\Mbar^{J,\beta}_{g,h;k,\ell}(X,L)$.
		\item\label{gkc-cobordant} If $J'$ is another $\omega$-tame almost complex structure, then we can choose auxiliary data for $\Mbar^J_{g,h;k,\ell}(X,L;\beta,\mu)$ and $\Mbar^{J'}_{g,h;k,\ell}(X,L;\beta,\mu)$ so that the resulting global Kuranishi charts are (oriented) cobordant.
	\end{enumerate}
\end{intthm}

\noindent
Both the forgetful (or stabilisation) map 
\[\stb \cl \Mbar^{\,J,\,\beta}_{g,h;k,\ell}(X,L)\to \Mbar_{g,h;k,\ell}\]
(when the latter is nonempty) and the evaluation map 
\[\eva \cl \Mbar^{\,J,\,\beta}_{g,h;k,\ell}(X,L)\to X^k\times L^{|\ell|}\]
can be lifted to smooth maps on the thickening $\cT$. Replacing the global Kuranishi chart by a stabilisation, the evaluation map can moreover be assumed to a be a submersion, cf. \cite{AMS21}. The smooth structure of $\cK$ is not canonical. We first construct a relative smooth structures as in \cite{HS22} and upgrade this to a smooth structure in \S\ref{subsec:smoothing-theory}, using a relative version of the smoothing theory of Lashof, \cite{Las79}. The advantage of this refinement is that we can achieve smoothness of the stabilisation map and the evaluation map.\\

In contrast to the case of moduli spaces of stable maps from closed surfaces, the global Kuranishi charts of Theorem~\ref{thm:open-gw-global-chart} have boundary and corners. On the level of moduli spaces, the boundary strata are given by fibre products of moduli spaces of lower complexity. We lift this statement to the level of global Kuranishi charts.

\begin{intthm}[Theorem~\ref{thm:boundary-strata-equivalent}]\label{int:boundary} The restriction of $\cK$ to a boundary stratum of $\Mbar_{g,h;k,\ell}^{\,J,\,\beta}(X,L)$ is oriented equivalent, up to an explicitly computed sign, to the respective fibre product of global Kuranishi charts.
\end{intthm}

\noindent
 Theorem~\ref{thm:boundary-strata-equivalent} also determines the orientation signs of the clutching maps with respect to the boundary orientation. This recursive structure of the boundary is the geometric basis underlying all algebraic structures arising from these moduli spaces. Having a sufficiently strong equivalence of the global Kuranishi charts over boundary strata will be crucial in the construction of compatible chain-level operations such as in the Fukaya algebra. This is considerably more subtle for curves with genus because the boundary strata are more complicated and different types of nodes appear.

\begin{remark}\label{}
	To our knowledge, this is the first time that these orientation signs, even for moduli spaces of open stable curves, have been worked out in higher genus. We use the orientation construction of \cite{CZ24} and their results on behaviour of orientation lines under degenerations; similar results have been shown in \cite{WW17,FOOO16}. We recall the constructions of \cite{CZ24} in \textsection\ref{sec:orientations}. 
\end{remark}

\noindent
While it is not always explicitly phrased that way, the key point of the obtaining algebraic structures from moduli spaces in \cite{Fu10} and \cite{ST23} is to consider moduli spaces as correspondences, which define push-pull operations. Concretely, given a moduli space $\Mbar$ whose marked points are partitioned into incoming and outgoing, a global Kuranishi chart $\cK = (\cT,\cE,s)$ as in Theorem~\ref{thm:open-gw-global-chart}, \emph{equipped with a Thom form of the obstruction bundle}, can be considered as `virtual' correspondence 
\begin{center}\begin{tikzcd}
		& \cK \arrow[dl,"\eva^-"]\arrow[dr,"\eva^+"]\\X^{k^-}\times L^{\ell^-} &&X^{k^+}\times L^{\ell^+}   \end{tikzcd} \end{center}
and thus yields a pull-push operation
\begin{align*}\label{}
	\Omega^*(X^{k^-}\times L^{\ell^-}) &\to \Omega^*(X^{k^+}\times L^{\ell^+})\\
	\alpha &\mapsto \eva^+_*({\eva^-}^*(\alpha)\wedge s^*\tau_{\cE}).
\end{align*}
This can be defined in any setting where (certain) morphisms define an exceptional pushforward (or pullback). For many applications it is necessary that the operation associated to a boundary stratum $\Mbar_1\times_L \Mbar_2$ of $\Mbar$ (respectively $\cK$) agrees with the the composition of the operations coming from $\Mbar_1$ and $\Mbar_2$. This is immediate for \cite{ST16} since they work directly with the moduli spaces themselves. It holds in closed Gromov-Witten theory (\cite{Hir23}) because these operations are cohomological and the virtual fundamental class only depends on the equivalence class of global Kuranishi charts.
 
In \cite{BX22} and \cite{Rez22}, global Kuranishi charts were constructed inductively to ensure compatibility in the setting of Hamiltonian Floer theory; cf. \cite{AB24}. 
In \cite{AGV24}, perturbation data were assigned to chains to obtain a strict framed $E_2$ action on a chain-model of symplectic cohomology. 
Either approach is in the operadic setting, respectively in genus zero. We instead opt for incorporating the differences between choices into the operation itself to obtain strict compatibility. 
To state the theorem, recall that boundary strata of moduli space are indexed by dual graphs $\Gamma$. Their precise definition is given in \S\ref{subsec:graphs}. 

\begin{intthm}[Theorem~\ref{thm:cubical-cobordisms-enhanced}]\label{} Given a choice of auxiliary data for $\Mbar_{g,h}^{J,\beta}(X,L)$ for each $g,h\ge 0$ and $\beta\in H_2(X,L;\bZ)$, respectively $H_2(X;\bZ)$ if $h = 0$, there exists a global Kuranishi chart $\cKc_\Gamma$ for $I^{E(\Gamma)}\times \del_\Gamma\Mbar^{J,\beta}_{g,h}(X,L)$ for each stable map graph $\Gamma$ so that 
	\begin{equation*}\label{} \cKc_\Gamma|_{(0,\dots, 0)} = \del_\Gamma\cK_{g,h}^\beta\qquad \quad\text{and}\qquad \quad \cKc_\Gamma|_{(1,\dots, 1)} \sub \p{v\in V(\Gamma)}{\cK_{g_v,h_v,k_v,\ell_v}^{\beta_v}}
	\end{equation*}
	 is the fibre product. Moreover, we can construct compatible Thom forms for their obstruction bundles. 
\end{intthm}

\noindent
The last assertion is Proposition~\ref{prop:thom systems exist}, which requires the extension result of \S\ref{sec:extend-Thom-forms}. Note that we can construct these system for \emph{all} moduli spaces of stable maps at once.
Taking the sum of the operations associated to the global Kuranishi charts $\cKc_\Gamma$, appropriately weighted, gives operations that are strictly compatible for almost all boundary strata. 

\subsection{Applications} The purpose of this paper is to build the geometric foundations required for chain-level operations obtained from moduli spaces of open stable maps. We describe several applications here.

\subsubsection{Open-closed Deligne--Mumford field theories}
The structure of an $A_\infty$ algebra on a vector space $\scA$ is most often described as the data of linear operations $\mu^k \cl \scA^{\otimes k}\to \scA$ that satisfy the $A_\infty$ relations, cf. \cite{Sei08}, or equivalently as an action of the $A_\infty$ operad. Another formulation, which is closer to the notion of a TCFT as in \cite{Cos07}, is to see an $A_\infty$ structure as a dg functor 
\begin{equation}
	\label{} \scF\cl C_*(\Mbar_{0,1;0,\bullet+1}) \to \text{Ch}_R
\end{equation} 
from the category of chains on moduli spaces of stable disks with boundary marked points. An open-closed Deligne--Mumford field theory (DMFT) is a generalisation of this to chains on moduli spaces of stable curves of arbitrary complexity, that is, we allow any genus and any number of boundary components. In practise, we cannot achieve strict compatibility with all the gluing of surface - but we can construct explicit homotopies. This leads to the following definition.

\begin{definition}\label{}
	An \emph{open-closed Deligne--Mumford field theory} is an $A_\infty$ functor $\scF\cl C_{*}(\Mbar_{oc})\to \text{Ch}_R$. $\scF$ is \emph{strict} if $\scF$ is a dg functor, \emph{semi-strict} if $\scF$ restricts to a dg functor on $C_*(\Mbar^{st}_{oc})$ and \emph{operadic} if $\scF$ restricts to a dg functor on $C_*(\Mbar^{opr}_{oc})$.
\end{definition}
\noindent
Here $\Mbar_{oc}$ denotes the category of possibly bordered Riemann surfaces with both interior and boundary marked points. We refer the reader to the follow-up paper \cite{HH1} for a precise definition. 

The target category $\text{Ch}_\bR$ is the category of chain complexes of $\bR$-vector spaces, enriched over itself. More generally, one can take any coefficient ring with a valuation. In a second paper, we show that moduli spaces of stable maps can be used to construct such an open-closed DMFT valued in chain complexes of modules over the Novikov ring. Concretely, we use the global Kuranishi charts constructed here to define an open-closed DMFT.

\begin{theorem}[\cite{HH1}]
	\label{thm:OCDMFT associated to L}
	Suppose $L \sub (X,\omega)$ is a relatively spin embedded Lagrangian. Then, there exists a curved open-closed DFMT $\scF_L$ whose closed part recovers the Gromov--Witten theory of $X$ and whose associated $A_\infty$ algebra is the Fukaya algebra of $L$. If $L$ is equipped with a weak bounding cochain, this can be deformed to an uncurved open-closed DMFT.
\end{theorem}

\subsubsection{Open Gromov-Witten invariants}
Taking the fundamental chain of the moduli spaces of stable discs, we extend the $\fq$-operations of Solomon--Tukachinsky, \cite{ST16}, which were defined under strong regularity assumption, to the general setting. Using the algebraic framework of \cite{ST1}, this allows us to define open Gromov-Witten invariants in genus $0$ for compact, weakly unobstructed Lagrangians. Here the strict compatibility of operations in the operadic setting, that is, where there is at most one outgoing marked point, makes the proofs significantly easier. The proof of the properties of these invariants, in particular their invariance under change of the almost complex structure or Hamiltonian isotopies is delegated to future work.

\begin{intthm}[{\cite{HH2}}]
	Associated to any weakly unobstructed closed Lagrangian $L$ in $(X,\omega)$ are genus $0$ open Gromov--Witten invariants that are invariant under changes of the almost complex structure and Hamiltonian isotopies and which satisfy the open Gromov-Witten axioms of \cite[Theorem~4]{ST1}.
\end{intthm}

\noindent
The invariance is obtained by studying the relations coming from (weakly unobstructed embedded) Lagrangian cobordisms. We expect that our geometric constructions are sufficient for defining higher genus open Gromov--Witten invariants, but more work is needed to algebraically extract invariants. We have been informed of work in-progess on numerical open Gromov--Witten invariants by Kedar--Solomon. A construction of higher-genus string-valued open Gromov-Witten invariants generalising to the work of Irie and Fukaya is given in \cite{DH}. It would be interesting to see whether the techniques of \cite{AB24b} could be used in conjunction with our results to define open Gromov--Witten invariants valued in complex bordism.

\addtocontents{toc}{\protect\setcounter{tocdepth}{1}}

\subsection*{Acknowledgements} We would like to thank Mohammed Abouzaid, Nick Sheridan, Jake Solomon, and Sara Tukachinsky for useful discussions. We thank Melissa Chiu-Chu Liu for her figures and her interest in our work and Or Kedar for pointing out a problem with using rel--$C^\infty$ structures. A.H. is grateful to Soham Chanda and Penka Georgieva for valuable conversations and thanks Mohan Swaminathan for generously sharing his knowledge on gluing. K.H thanks Jonny Evans for helpful conversations, leading in particular to \textsection \ref{sec:extend-Thom-forms}.

Part of this work was completed while both authors were in residence at the Simons Center for Geometry and Physics as well as the Mittag-Leffler Institute. We thank both institutes for their hospitality. A.H. is supported by ERC Grant No.864919 and K.H. was partially supported by EPSRC Grant EP/W015749/1.

\addtocontents{toc}{\protect\setcounter{tocdepth}{2}}
\section{Global Kuranishi charts for moduli spaces of open stable maps}\label{sec:construction}

\noindent
Let $(X^{2n},\omega)$ be a closed symplectic manifold and $L \sub X$ a closed Lagrangian submanifold. Let $J \in \cJ_\tau(X,\omega)$ be arbitrary. Given $g,h,k,\ell_1,\dots,\ell_h \geq 0$, as well as 
$$\beta = \beta_X\oplus (\beta_i)_{1\leq i \leq h}\in H_2(X,L;\bZ)\oplus H_1(L;\bZ)$$ \noindent
with $\del \beta_X = \sum_{i=1}^h \beta_i$, define $$\Mbar^{\,J,\,\beta}_{g,h;k,\ell}(X,L)$$
to be the space of $J$-holomorphic maps $u \cl (C,\del C)\to (X,L)$ such that 
\begin{itemize}[leftmargin=20pt]
	\item $C$ is a bordered nodal Riemann surface of arithmetic genus $g$ with $h' \leq h$ \emph{ordered} boundary components $S_1,\dots,S_{h'}$,
	\item $x_1,\dots,x_k \in C\sm \del C$ are $k$ distinct interior marked points
	\item $y_{i,1},\dots,y_{i,\ell_i}$ are $\ell_i$ distinct boundary marked points on $S_i$,
	\item $u_*[C] = \beta_X$ and $u_*[S_i] = \beta_i$ for $1 \leq i \leq h$,
	\item $\Aut(u,C,x_1,\dots,x_k,y_{1,1},\dots,y_{h,\ell_h})$ is finite,
\end{itemize}

\begin{remark}\label{rem:order-boundary-circles} The ordering of the boundary circles is crucial for orientability, see e.g. Lemma~\ref{lem:permutation-marked-points}\eqref{permute-boundary}. However, we will omit it from the notation throughout. \end{remark}

\noindent
By \cite{Ye94}, respectively \cite[Theorem 1.1]{Liu20}, the moduli space $\Mbar^{J,\,\beta}_{g,h;k,\ell}(X,L)$ endowed with the Gromov topology is compact and Hausdorff. The main result of this section is Theorem \ref{thm:open-gw-global-chart}, rephrased as Proposition \ref{prop:open-gw-global-chart-existence} and Proposition~ \ref{prop:marked-points-gkc-simple}. In order to proceed as linearly as possible, we relegate the description of the thickening, obstruction bundle and section to \textsection\ref{subsec:overview}.

\subsection{Base space and framings}\label{subsec:base-and-framings}
As in any other global Kuranishi chart construction, the key first step is the construction of a suitable base space for the thickening of the chart. This allows us to rigidify the stable maps, which is required to be able to perturb suitably. We first define the base space abstractly in \S\ref{subsec:base-space}, whose elements we call \emph{framings}. In \S\ref{subsec:framings}, we associate framings to open stable maps, and in \S\ref{subsec:reducing-group} we define the parametrised infinite-dimensional moduli space out of which we cut the thickening by a perturbed Cauchy--Riemann equation.

\subsubsection{Base space}\label{subsec:base-space}
Equip $\bC P^N$ with the standard complex structure. By \cite[Theorem 1.2]{Liu20} and Lemma \ref{lem:doubling-preserves-regularity}, the moduli space 
\begin{equation}\label{eq:base-space}\cB:= \Mbar_{g,h;0,0}^{*}(\bC P^N,\bR P^N;m)\sub \Mbar_{g,h;0,0}(\bC P^N,\bR P^N;m) \end{equation}
of maps whose complex double is a regular non-degenerate embedding into $\bC P^N$ is a topological manifold of real dimension
$$\dim(\cB) = \mu_{\bR P^N}(m) + (N-3)(2-2g-h).$$\noindent
We make the following observation.

\begin{lemma}\label{lem:doubling-preserves-regularity} Suppose $u \cl (C,\del C)\to (Y,N)$ is a $J_Y$-holomorphic map, where $Y$ admits an anti-holomorphic involution $\varphi$ and $L = \text{Fix}(\varphi)$. Let $u_\bC\cl C_\bC\to Y$ be its complex double. Then $u_\bC$ is regular if and only if $u$ is regular. Moreover, given any collection of smooth points $x_1,\dots,x_n\in C$, we have 
	$$\Aut(u,C,\del C,x_*) \leq \Aut(u_\bC,C_\bC,x^\bC_*).$$
\end{lemma}

\begin{proof}
	The first claim follows from the fact that $\varphi$ and the canonical symmetric structure $\sigma$ on $C_\bC$ induce an anti-holomorphic involution $\varphi_u$ on $H^1(C,u^*T_Y)$ with 
	$$H^1((C,\del C),u^*(T_X,T_L)) = H^1(C_\bC,u_\bC^*T_Y)^{\varphi_u}.$$\noindent
	If $\rho$ is an automorphism of $(u,C,x_1,\dots,x_n)$ preserving the boundary components set-wise, it extends to an automorphism of $u_\bC$.
\end{proof}
\noindent
We need the following stronger result.\footnote{Liu shows that the moduli space of all bordered maps to $\bC P^N$ admits a Kuranishi structure; however, their transition functions are not naturally smooth.} 

\begin{proposition}\label{prop:base-space} $\cB$ admits the structure of a smooth orientable manifold with corners, so that the doubling map $\cB\to \Mbar_{\tilde{g}}^*(\bC P^N,2m)$ is smooth.\end{proposition}

\begin{proof} Let $\fc \cl \bC P^n \to \bC P^n$ be the real structure given by complex conjugation and define $$\bR\cB := \bR\Mbar_{\tilde{g}}^*(\bC P^N,2m)$$\noindent
	to be the fixed point locus of the induced anti-holomorphic involution on $\Mbar_{\tilde{g}}^*(\bC P^N,2m)$, where $\tilde{g} =2g+h-1$. Then $\bR\cB$ is a smooth manifold, and we can decompose it as
	$$\bR\cB = \union{\phi}{\bR\cB^\phi},$$\noindent
	where $\phi$ is an anti-holomorphic involution on a surface of genus $\tilde{g}$ and $\bR\cB^\phi$ is the closure of 
	$$\{u \in C^\infty(\Sigma_{\tilde{g}},\bC P^N) \mid\delbar u = 0,\; u\phi = \fc u\}/\sim$$\noindent
	in $\bR\cB$. Denote by $\phi_*$ the unique anti-holomorphic involution such that 
	$$\Sigma_{\tilde{g}}\sm \text{Fix}(\phi_*) = \text{int}(\Sigma_{g,h})\sqcup \text{int}(\cc{\Sigma}_{g,h})$$\noindent
	with $\phi_*$ interchanging the two components.
	The doubling construction induces a natural surjective map $\Phi\cl \cB \to \bR\cB^{\phi_*}.$
	It is injective as we consider real maps and not just closed stable maps to $\bC P^N$. By the definition of $\cB$, the induced map from the quotient is the restriction of a map between compact moduli spaces. Hence,it is proper and thus closed. It follows that $\Phi$ is a homeomorphism. Declaring it to be a diffeomorphism, we obtain a canonical smooth structure (with corners) on $\cB$.\par
	In order to prove orientability, let $\cV$ be the preimage of $\cB$ under the map 
	$$\Mbar_{g,h;0,0}^{\,m}(\cO(1)^{\oplus N+1},\cO(1)_\bR^{\oplus N+1})\to \Mbar_{g,h;0,0}^{\,m}(\bC P^N,\bR P^{N}),$$\noindent
	where we identify relative homology classes on $\cO(1)^{\oplus N+1}$ with those on $\bC P^N$.
	By assumption on $\cB$, the map $\cV\to \cB$ is a vector bundle, and by \cite[Corollary~1.10]{Ge13}, 
	\begin{equation}\label{eq:tangent-bundle-base}\det(T\cB)\cong\det(\cV) \end{equation} 
	over the locus of maps with smooth domain whose complement is of codimension $1$. To see that $\cV$, and thus $\det(\cV)$, is trivial, let $\sigma_0,\dots,\sigma_N\cl (\bC P^N,\bR P^N) \to (\cO(1),\cO(1)_\bR)$ be the standard sections and write $\sigma_j^{(i)}$ for the $j^{th}$ standard section mapping to the $i^{th}$ factor of $\cO(1)^{\oplus N+1}$. Hence, we can define sections $\wt\sigma_j^{(i)}\cl \cB\to \cV$ by setting 
	$$\wt\sigma_j^{(i)}([\iota,C]) = \iota^*\sigma_j^{(i)},$$\noindent
	where $[\iota,C]$ denotes the stable holomorphic map $\iota \cl (C,\del C)\to (\bC P^N, \bR P^N)$.
	These sections are linearly independent at each point since the complex double $\iota_\bC$ of $\iota$ is non-degenerate. As $$\dim H^0((C,\del C),\iota^*(\cO(1),\cO(1)_\bR)) = N+1,$$\noindent
	they form a global frame of $\cV$. In particular, $\det(T\cB)$ is trivial as well.
\end{proof}

\begin{lemma}\label{lem:action-base-space} The canonical action of $\PGL_\bR(N+1)$ on $\cB$ is orientation-preserving.\end{lemma}

\begin{proof} If $N$ is even, then $\PGL_\bR(N+1)$ is connected, so the claim is immediate. Suppose $N$ is odd and let $A \in \PGL_\bR(N+1)$ be (the image of) the matrix interchanging the basis vectors $e_0$ and $e_1$ and leaving all other basis vectors invariant. Denote by $\phi_A\cl \cB \to \cB$ the associated diffeomorphism. Then 
	$$\phi_A^*\wt\sigma^{(i)}_j = \begin{cases}
		\wt\sigma^{(i)}_1 \quad & j = 0\\
		\wt\sigma^{(i)}_0 \quad & j = 1\\
		\wt\sigma^{(i)}_0 \quad & \text{otherwise}.\end{cases}$$\noindent
	Thus, the determinant of the linear map $\phi_A^*$ with respect to the basis $(\wt\sigma^{(i)}_j)_{i,j}$ is $(-1)^{N+1} = 1$.
\end{proof}

\begin{lemma}\label{lem:orientation-base} Let $\zeta = [\iota,C]$ be an element of the top stratum of $\cB$. Then there exists a canonical isomorphism $$\det(D\delbar_{J_0}(\iota)) \cong \det(\fp\fg\fl_\bR(N+1))\otimes \det(\hodge_\zeta),$$\noindent
	where $\hodge \to \cB$ is the dual of the Hodge bundle. 
\end{lemma}

\begin{proof} Pulling back the Euler sequence $0 \to \cO_{\bC P^N}\to \cO(1)^{\oplus N+1}\to T\bC P^N \to 0$ (with its corresponding sequence of real subbundles over $\bR P^N$) to $C$ and taking the long exact sequence in homology, we get an exact sequence 
	\begin{multline}\label{eq:from-euler} 0 \to H^0(C,(\bC,\bR))\to H^0(C,\iota^*(\cO(1),\cO(1)_\bR))^{\oplus N+1}\to H^0(C,\iota^*(T\bC P^N,T\bR P^N))\\\to H^1(C,(\bC,\bR))\to 0. \end{multline}
	The index bundle of the trivial bundle $(\bC,\bR)\to (C,\del C)$ has a canonical orientation by \cite[Proposition~4.1.1]{WW17}. Furthermore, the basis $\{\iota^*\sigma_j^{(i)}\}_{i,j}$ gives a canonical identification of $\fg\fl_\bR(N+1)$ with $H^0(C,\iota^*(\cO(1),\cO(1)_\bR))^{\oplus N+1}$, and the image of the first map of \eqref{} corresponds to the inclusion of the scalar matrices into $\fg\fl_\bR(N+1)$. Thus \eqref{eq:from-euler} becomes a natural sequence
	$$0 \to \fp\fg\fl_\bR(N+1) \to \ind(D\delbar_{J_0}(\iota))\to \hodge_\zeta\to 0,$$\noindent
	whence the claim follows by taking determinants.\end{proof}

\begin{remark}\label{} We observe that the way we obtain the orientation on $\cB$ in a way which differs from the usual way of orienting moduli spaces of curves with boundary. Indeed, the Lagrangian $\bR P^N$ might not be orientable. However, Lemma \ref{lem:orientation-base} might seem more natural if we consider that adding the decoration of a framing can be seen as exhibiting the moduli space $\Mbar_{g,h;0,0}^{\,J,\,\beta}(X,L)$ as a global quotient of the form $\cN/\PGL_\bR(N+1)$. This will be made more precise in the next subsection; compare also with \cite[Discussion~3.16]{HS22}\end{remark}

\subsubsection{Framings of stable maps with boundary}\label{subsec:framings} In this subsection, we construct framings for stable maps on bordered Riemann surfaces. For this we use the results of \cite{HS22} in the closed case and the complex doubles of our curves described in detail in \cite[\textsection3]{KL06}.

\begin{lemma}[Adapted polarisation] \label{lem:adapted-polarisation} Suppose $(Y,\sigma)$ is any symplectic manifold and $K\sub Y$ is a Lagrangian so that $[\sigma]\ne 0$ in $H^2(Y,K;\bR)$. Then, given any compact subset $\scJ$ of $\sigma$-tame almost complex structures on $Y$, there exists a line bundle $\cO_Y(1)\to Y$ such that $\cO_Y(1)|_K$ admits a real structure $\cO_K(1)$, and for any $J \in \scJ$ and any non-constant $J$-holomorphic map $u \cl C\to Y$ we have $$\mu(u^*\cO_Y(1),u^*\cO_K(1)) \geq 1.$$
\end{lemma}

\noindent
We call such a line bundle an \emph{$K$-adapted polarisation}.

\begin{proof} We show first that there exists a symplectic form $\Omega$ on $Y$ taming each $J$, so that $[\Omega]\in H^2(Y;\bZ)$, $\Omega|_K = 0$, and $[\Omega]\neq 0$ in $H^2(Y,K;\bR)$. Denote by $\cZ$ the set of closed $2$-forms and let $\cZ_s \sub \cZ$ be the open subset of symplectic forms. Let $A := \{\alpha\in \cZ\mid \alpha|_K = 0\}$ and let $A_\bQ$ be the subset of $\alpha\in A$ with $[\alpha]\in H^2(X;\bQ)$. Then $A_\bQ$ is the preimage of a dense set under the open map $A \rightarrow H^2(Y,K;\bR)$ and thus dense. As $\omega\in A\cap \cZ_s$, it follows that there exists $\Omega'\in A_\bQ\cap \cZ_s$ taming $J$. Set $\Omega := N\Omega'$ for $N\gg 1$ so that $[\Omega]\in H^2(Y,K;\bZ)$.\par
	 By \cite[Lemma~3.2]{HS22}, there exists a Hermitian line bundle $(\cL,\conn')\to Y$ so that $F^{\conn'} = -\frac{1}{2\pi\text{i}}\Omega$. Let $U$ be a Weinstein neighbourhood of $K$. Then, $c_1(\cL|_U) = 0$, so there exists a complex trivialisation $\Phi \cl U\times\bC\to \cL|_U$. In particular, we may write $\Phi^*\conn'= d + \frac{\text{i}}{2\pi}\theta'$ for some $\theta'\in \Omega^1(U,\bR)$ with $\Omega|_U = d\theta'$. In particular, $\theta'|_K$ is closed. Fix some Riemannian metric on $g$ and let $\epsilon >0 $ be sufficiently small such that for any $\alpha \in \Omega^1(U)$ with $\norm{\alpha}_{C^1} < \epsilon$, the differential $d(\theta' +\alpha)$ is a symplectic form on $U$. Let $\rho$ be a bump function that is identically $1$ on $K$ and supported in $U$. Since $U$ deformation retracts onto $K$ and $H^1(K;\bQ)$ is dense in $H^1(K;\bR)$, we can find a closed form $\alpha\in \Omega^1(U)$ with $\norm{\alpha}_{C^1} < \frac{\epsilon}{2\norm{\rho}_{C^1}}$ so that $[(\theta'+\alpha)|_K] \in H^1(K;\bQ)$. Set $\theta := N'(\theta' +\rho\alpha)$ for some sufficiently large integer $N'$. Then, $\theta|_K$ is closed with $[\theta|_K] = N'[\theta'|_K+\alpha|_K] \in H^1(K;\bZ)$. Since $\norm{\rho\alpha}_{C^1} \leq 2\norm{\rho}_{C^1}\norm{\alpha}_{C^1} < \epsilon$, the derivative $d\theta$ is symplectic. Thus, we may replace ${\conn'}^{\otimes N'}$ on $\cL^{\otimes N'}$ by the connection $\conn$ given by $d+\theta$ on $U$ and by ${\conn'}^{\otimes N'}$ away from $U$. By construction, $-\frac{1}{2\pi\text{i}}F^\conn$ is still a symplectic form vanishing on $K$. We set $\cO_Y(1) := \cL^{\otimes N'}$.\par
	 We can now construct the real subbundle $\cO_K(1)$. Fix for each component $B$ of $K$ a base point $y_B$ and (using the above fixed trivialisation) let $\cO_K(1)_y$ be the image of $\Phi(\{y_B\}\times \bR)$ under the parallel transport along some path from $y_B$ to $y$. The choice of path is irrelevant since the holonomy along any loop $\gamma$ on $L$ is given by $e^{-\int_\gamma\theta} = e^{-2\pi\text{i}k}$ for some $k \in \bZ$. In particular, $\cO_K(1)$ is a well-defined smooth real subbundle of $\cO_Y(1)|_K$ and preserved by the parallel transport of $\conn$. Therefore, we can apply \cite[Theorem 3.1]{ChSh16} to obtain that 
	$$\mu(u^*\cO_Y(1),u^*\cO_K(1)) = 2\lspan{[\Omega],u_*[C]} > 0$$\noindent
	As the left-hand side is an integer, the claim follows.
\end{proof}
 
 \noindent 
Given a stable smooth map $u \cl (C,\del C)\to (X,L)$ and a complex structure $\fj$ on $C$, denote by 
$$(C_\bC,\fj_\bC) := (C,\fj)\cup_{\del C}(C,-\fj)$$\noindent
 the \emph{complex double} of $C$. It has arithmetic genus $g_{\bC} = 2g + h-1$ by \cite[\textsection 3.3]{KL06}. There exists a canonical anti-holomorphic involution $\sigma_C \cl C_\bC\to C_\bC$ with fixed-point locus given by $\del C$. By \cite[Theorem 3.3.13]{KL06} and the argument at the bottom of p.16 op. cit., the line bundle pair $(u^*\cO_X(1),(u|_{\del C})^*\cO_L(1))$ extends to a holomorphic line bundle $u^*\cO_X(1)_\bC$ on $C_\bC$ which admits an anti-holomorphic involution $\wt\sigma_C$ covering $\sigma_C$. We define
\begin{equation*}\label{eq:complex-polarisation}
	\fL_{u,\bC} = \omega_{C_\bC}\otimes u^*\cO_X(1)_\bC^{\otimes 3}
\end{equation*}
and set 
\begin{equation}\label{eq:real-line-bundle}\fL_u := \fL_{u,\bC}|_{C}. \end{equation} 
Moreover $\fL_{u,\bC}$ admits a canonical lift of $\sigma_C$ given by combining the pullback ${\sigma_C}^*$ and $\wt\sigma_C$.
By \cite{KL06},
$$\deg(\fL_{u,\bC}|_{C'}) \geq -2 + 3\mu(u^*\cO_X(1)|_{C'\cap C},u^*\cO_L(1)|_{C'\cap \del C}) > 0,$$\noindent
for any irreducible component $C'$ of $C_\bC$, on which $u$ is non-constant, while $\omega_{C_\bC}$ is positive on any doubling of a component of $C$ on which $u$ is constant. 

\begin{remark}\label{de:if-only-constant-curves} If $X$ is a non-compact manifold and $L \sub X$ is a closed embedded Lagrangian so that $[\omega] = 0$ in $H^2(X,L;\bR) = 0$, then any pseudo-holomorphic stable map $u \cl (C,\del C)\to (X,L)$ is a constant and thus $C$ stable. In this case, we can define $\fL_u$ to be the dualising line bundle of the domain (possibly twisted by marked points) to obtain a line bundle of positive degree on each component.\end{remark}

\noindent
Thus, by \cite[Lemma~3.2]{HS22}\footnote{If $\omega(\beta) = 0$, the construction has to be adapted as in \cite[Remark~3.14]{HS22}.},
  there exists $p \gg 1$, depending only on the topological type of $C$ (in fact, of $C_\bC$) and $d:=\lspan{\Omega,\beta}$, so that $\fL_{u,\bC}^{\otimes p}$ is very ample and $H^1(C_\bC,\fL_{u,\bC}^{\otimes p}) = 0$. Note that if $p'$ satisfies this condition, then so does any $p > p'$. Thus we can choose $p$ to be divisible by $4$. We have 
$$m := p\,\deg(\fL_u) = p\lbr{3\mu(u^*(O_X(1),O_L(1)) + 2(2g+ h-2)}$$\noindent
 The space of global holomorphic sections $H^0(C_\bC,\fL_{u,\bC}^{\otimes p})$ admits an anti-complex linear involution. A real basis of $H^0(C_\bC,\fL_{u,\bC}^{\otimes p})^{\bZ/2}$ defines a complex basis of $H^0(C_\bC,\fL_{u,\bC}^{\otimes p})$. We call a complex basis $F = (s_0,\dots,s_N)$ of $H^0(C_\bC,\fL_{u,\bC}^{\otimes p})$ \emph{real admissible} if $s_0,\dots,s_N$ lie in $H^0(C_\bC,\fL_{u,\bC}^{\otimes p})^{\bZ/2}$. In this case, $F$ induces a regular holomorphic embedding 
$$\iota_{F,\bC}\cl (C_\bC,\del C)\to (\bC P^N,\bR P^N),$$\noindent
where 
\begin{equation}\label{eq:dimension-base-lag} N = \dim_\bC H^0(C_\bC,\fL_{u,\bC}^{\otimes p})-1 = p \deg(\fL_{u,\bC}) - 2g-h+1.\end{equation} 

\begin{remark}\label{rem:real-admissible-uniqueness}
	A real framing $\iota\cl (C,\del C)\to (\bC P^N,\bR P^N)$ is determined by $[\iota_\bC^*\cO(1)]\in \Pic(C_\bC)$ 
	up to the $\PGL_\bR(N+1)$-action on the space of holomorphic real stable maps $C_{\bC} \rightarrow\bC P^N$.
\end{remark}

\subsubsection{Reducing the structure group}\label{subsec:reducing-group}
We will take a middle path between the approaches of \cite{AMS23} and \cite{HS22} to this problem. The aim is to simplify the construction of the cross-section $\lambda_\cU$ in \cite{HS22}, while not constructing domain-dependent metrics as in \cite{AMS23} and instead using a perturbation scheme similar to \cite{HS22}. To this end, abbreviate
$$\cG := \PGL_\bR(N+1)\qquad \qquad G := \text{PO}(N+1)$$\noindent
and denote by $\cB$ the space defined in \eqref{eq:base-space}. As $\cB$ is a manifold, it admits a universal curve
\begin{center}\begin{tikzcd}
		\cC \arrow[r,"\eva"] \arrow[d,""]&\cB\times\bC P^N\\ 
		\cB \end{tikzcd} \end{center}
By construction, the boundary of a fibre of $\cC\to \cB$ is mapped to $\bR P^N$. Similarly to $\fF_{\scF}$ defined in \cite[\textsection 4.3]{AMS23}, let $\cZ\to \cB$ be the family of maps $u\cl (\cC|_{[\iota]},\del\cC|_{[\iota]})\cl (X,L)$ with $[\iota]\in \cB$ 
such that
\begin{itemize}
	\item $\int_{C'} u^*\omega\geq 0$ for any irreducible component $C'$ of $C$,
	\item $\int_{C'} u^*\omega\geq\hbar$ for any unstable irreducible component $C'$ of $C$.
\end{itemize}

Equip $\cZ$ with the Gromov-Hausdorff metric on the graphs in $\bC P^N\times X$ and with the obvious $\cG$-action lifting the $\cG$-action on $\cB$.
Recall that the action of a topological group $H$ on space $W$ is \emph{Palais proper} if for any $w\in W$ and any subset $V\sub W$ there exists a neighbourhood $U$ of $w$ so that $\{h \in H\mid h\cdot U \cap V\neq\emst\}$ has compact closure in $H$.

\begin{lemma}\label{lem:palais-proper} The $\cG$ action on $\cZ$ is Palais proper. In particular, it admits $\cG$-invariant partitions of unity subordinate to $\cV$ for any $\cG$-invariant open cover $\cV$. 
\end{lemma}

\begin{proof} By verbatim the same proof as the one of \cite[Lemma 4.13]{AMS23}, one shows that the action of $\cG$ on $\cZ$ is Palais proper. Given this, we obtain the claim by \cite[Corollary 3.12]{AMS23}, respectively \cite[Theorem 4.3.4]{Pal61}.\end{proof}

\begin{remark}\label{rem:partition-of-almost-unity}In particular, given any closed $\cG$-invariant subset $Z$ of $\cZ$ and any $\cG$-invariant open covering $\cV$ thereof, we can find a collection of continuous $\cG$-invariant functions $\{\chi_V\}_{V\in \cV}$ subordinate to said open cover such that for each $x\in Z$ we have $\chi_V(x) > 0$ for some $V$.
\end{remark}

\noindent
We adapt the notion of a good covering in \cite[Definition~3.10]{HS22} to our setting.

\begin{definition}[Good covering]\label{good-covering}
	A \emph{good covering}	$\cU=\{(U_i,D_i,\chi_i)\}_{i\in\Lambda}$ is a finite set of tuples such that the following holds.
	\begin{itemize}[leftmargin=20pt]
		\item\label{stabilizing-divisor} Each $U_i$ is an open subset of $\cZ$ and $D_i\sub X\sm L$ is a submanifold with boundary of codimension $2$ satisfying for any $(u,C,\iota)\in U_i$ that
		\begin{enumerate}[\normalfont (i)]
			\item\label{transverse-marking} the preimage $u^{-1}(D_i)$ consists of exactly $3d$ distinct non-nodal points of $C$ and
			\begin{align*}
				\#(u^{-1}(D_i)\cap C') = 3\langle [\Omega],u_*[C']\rangle.
			\end{align*}
			\item\label{avoid-boundary} we have $u(C)\cap\partial D_i=\varnothing$ and $u$ is transverse to $D_i$.
		\end{enumerate}
		for any bordered stable map $(u,C,\iota)\in U_i$. In particular, to each $(u,C,\iota)\in U_i$, we can associate $\text{st}_{D_i}(u,C,\iota)\in\cB_{[3d]}$, by adding $u^{-1}(D_i)$ to $C$ as a set of $3d$ unordered interior marked points.
		\item Each $\chi_i:\cZ\to[0,1]$ is a continuous function with support contained in $U_i$. 
		\item For any $[u,C]$ in $\Mbar_{g,h}^{J,\mu}(X,L;\beta)$ there exists $\iota \cl (C,\del C)\to (\bC P^N,\bR P^N)$ and $i\in\Lambda$ such that $(u,C,\iota)\in U_i$ and $\chi_i(u,C,\iota)>0$.
\end{itemize}\end{definition}

\noindent
By \cite[Lemma 9.4.3]{P16} and Lemma \ref{lem:palais-proper}, good coverings exist. Fixing one such covering $\cU$, let $V_\cU\subset \cZ$ be the open subset where $\sum_{i\in \Lambda}\chi_i$ is strictly positive. Assume that we are additionally given a smooth $\cG$-equivariant map
\begin{align*}
	\lambda:\cB_{[3d]}\to\cG/G,
\end{align*}
where $\cB_{[3d]}\sub \Mbar_{g,h;3d,0}^{\,m}(\bC P^N,\bR P^N)/S_{3d}$ is contained in the locus of regular embedded curves whose domains are stable. The action of the symmetric group $S_{3d} $ is given by permuting the marked points.

\begin{lemma}\label{lem:real-polar-decomposition} Set $\fp(N+1) := \{A\in \fg\fl_\bR(N+1)\mid A^t = A,\; \normalfont\text{trace}(A) = 0\}$. Then the composition 
	$$\fp(N+1)\xra{\exp} \GL_\bR(N+1)\to \cG/G$$\noindent
	is an isomorphism, which we denote by $\exp$ as well, abusing notation.
\end{lemma}

\begin{proof} This follows from the polar decomposition, $\GL_\bR(N+1) \cong \normalfont\text{O}(N+1)\times \normalfont\text{P}(N+1),$ where $\normalfont\text{P}(N+1)$ is the Lie algebra of positive definite symmetric matrices.
\end{proof}

\subsection{Construction}\label{subsec:overview} We now are able to define a global Kuranishi chart with corners for the moduli space $\Mbar_{g,h;0,0}^{\,J,\,\beta}(X,L)$. It requires the choice of certain auxiliary data.

\begin{definition}\label{de:aux-choices-defined}
	An \emph{auxiliary datum} for $\Mbar_{g,h;0,0}^{\,J,\,\beta}(X,L)$ is a tuple $(\nabla^X,\cO_X(1),p,\cU,\lambda,r)$  where
	\begin{enumerate}[i),leftmargin=25pt,ref=\roman*]
		\item $\nabla^X$ is a J-linear connection on the tangent bundle $TX$,
		\item\label{polarization-on-target} $(\cO_X(1),\conn)\to X$ is an $L$-adapted polarisation as in Lemma \ref{lem:adapted-polarisation} with $d := \lspan{\Omega,\beta}$.
		\item\label{framed-maps-etc} $p\ge 1$ is a positive multiple of $4$, 
		\item\label{aux-good-covering} $\cU$ is a good covering in the sense of Definition \ref{good-covering} and $$\lambda\cl \cB_{[3d]}\to \cG/G$$
		is an equivariant smooth map on the moduli space $\cB_{[3d]}$ defined in \textsection\ref{subsec:reducing-group},
		\item $r\geq 1 $ is an odd integer.
	\end{enumerate}
\end{definition}

\begin{remark} The restriction on the parity of $p$ and $r$ serves to make the discussion of orientability more convenient. \end{remark}

\noindent
Given such an auxiliary datum, define 
\begin{equation}\label{eq:perturbation-bundle}
	E_r := \cc{\Hom}_\bC(p_1^*T{\bC P^N},p_2^*TX) \otimes p_1^*\cO(r)\otimes \cc{H^0(\bC P^N,\cO(r))}
\end{equation}
over $\bC P^N\times X$ and let 
$$F_r := \Hom_\bR(p_1^*T{\bR P^N},p_2^*TL) \otimes p_1^*\cO_{\bR P^N}(r)_\bR\otimes \cc{H^0(\bC P^N,\cO(r))}^{\bZ/2}$$
be the subbundle over the Lagrangian $\bR P^N\times L$, where $\cO(r)_\bR$ is the tautological line bundle over $\bR P^N$. Note that the Hermitian structures on $\cO(r)$ and $\cO(r)_\bR$ induce maps $$\lspan{\cdot}\cl E_r\to \cc{\Hom}_\bC(p_1^*T{\bC P^N},p_2^*TX)$$ 
and 
\[\lspan{\cdot}\cl F_r \to \Hom_\bR(p_1^*T{\bR P^N},p_2^*TL).\]
Denote the induced complex structure on $E_r$ by $J_E$. We define an almost complex structure $\wt J$ on the total space of $E_r$ as follows. Write $$TE_r \cong T\bC P^N\oplus TX\oplus E_r,$$ 
using the canonical connection on $E_r$. 
Then define 
$$\wt J_{(w,x,e)}(\hat{w},\hat{x},\hat{e}) = (J_0\hat{w},J_x\hat{x}+\lspan{e}(\hat{w}),{J_E}_e(\hat{e}))$$
with respect to this splitting. Since $L$ is totally real with respect to $J$, we have $J_x\lspan{e}(\wh w) \in T_xL$ only if $\lspan{e}(\wh w) = 0$. Hence, $F_r$ is totally real with respect to $\wt J$.

\begin{construction}[Pre-thickening]\label{construction-thickening}	The \emph{pre-thickening} $\cT^{\normalfont\text{pre}}_r$ consists of equivalence classes of tuples $(u,C,\iota,\gamma,\eta)$ where 
	\begin{enumerate}[label=\roman*),leftmargin=20pt,ref=\roman*]
		\item $([\iota,C],u)\in \cZ$, which we will often write as $(\iota,u)$;
		\item $\alpha\in H^1(C,(\bC,\bR))$ satisfies 
		$$[\iota^*\cO(1)]\otimes [\fL_u^{\otimes -p}] = \exp(\gamma)$$ 
		in $\Pic(C)$, where the line bundle $\fL_u$ is defined by \eqref{eq:real-line-bundle};
		\item $\eta \in H^0(C,(\iota,u)^*(E_r,F_r))$ satisfies the perturbed Cauchy-Riemann equation 
		\begin{equation}\label{eq:perturbed-cr}\delbar_J \wt{u} + \lspan{\eta}\g d\wt{\iota} = 0 \end{equation}
		on the normalisation $\wt{C}$ of $C$, where $\wt u$ and $\wt \iota$ are the pullbacks to $\wt C$.
	\end{enumerate}
	We say $(u,C,\iota,\gamma,\eta)$ and $(u',C',\iota',\alpha',\eta')$ are \emph{equivalent} if there exists a biholomorphism $C\to C'$ preserving the order of the boundary components as well as all other structures. Note that each point of $\cT^{\normalfont\text{pre}}_r$ has a canonical representative where the domain is given by the image of $\iota$. 
\end{construction}

\noindent
Let $\ff \cl \cT^{\normalfont\text{pre}}_r\to \cZ$ be the forgetful map. Replacing $\cT^{\normalfont\text{pre}}_r$ by $\ff\inv(V_\cU)$, define $\lambda_\cU \cl \cT^{\normalfont\text{pre}}_r\to \fp(N+1)$ by 
\begin{equation}\label{eq:slice-map}
	\lambda_\cU(y) = -\text{i}\s{j\in \Lambda}{\chi_j(\ff(y))\exp\inv(\lambda_j(\text{st}_{\ell_j,D_j}(y)))}. 
\end{equation}
In order to make use of standard gluing results, Theorem \ref{thm:representable}, we will exhibit $\cT^{\normalfont\text{pre}}_r$ (almost) as a family $\Mbar_{g,h;0,0}^{*,\wt J_E,(m,\beta)}(\pi;E_r,F_r)$ of pseudo-holomorphic curves to $(E_r,F_r)$ over the base $\cB$. Let $\hodge$ be the dual of the Hodge bundle of $\pi\cl \cC\to \cB$ and denote its pullback to $\Mbar_{g,h;0,0}^{*,\wt J_E,(m,\beta)}(\pi;E_r,F_r)$ by the same symbol.

\begin{lemma}\label{lem:subspace-of-other-holomorphic-maps} There exists an injection $\cT^{\normalfont\text{pre}}_r\hkra \hodge$. \end{lemma}

\begin{proof}
	This is a straightforward adaptation of \cite[Lemma~4.16]{HS22}.
\end{proof}

\noindent
In particular, this equips $\cT^{\normalfont\text{pre}}_r$ with a topology so that the forgetful map $\ff$ is continuous. 

\begin{remark}\label{rem:doubling-perturbations}  By \cite[Theorem 3.3.8]{KL06}, the vector bundle $(\iota,u)^*E_r$ equipped with the induced real structure over $\del C$ extends to a holomorphic vector bundle $E_{\iota,u}\to C_\bC$. By \cite[Theorem 3.3.3]{KL06}, there exists a doubling map 
	$$H^0(C,(\iota,u)^*(E_r,F_r))\xra{\simeq} H^0(C_\bC,E_{\iota,u})^{\bZ/2}$$ 
	where $\bZ/2$ acts via the induced anti-holomorphic involution on $H^0(C_\bC,E_{\iota,u})$.
\end{remark}

\begin{construction}[Obstruction bundle and section]\label{construction-other-data} The \emph{obstruction bundle} $\cE\to \cT_r^{\normalfont\text{pre}}$ has fibres 
	\begin{equation}\label{eq:obstruction-fibre}
		\cE_{(u,C,\iota,\gamma,\eta)} =   H^0(C,(\iota,u)^*(E_r,F_r)) \oplus H^1(C,(\bC,\bR))\oplus \fp(N+1)
	\end{equation}
	The \emph{(preliminary) obstruction section} $\obs_0 \cl \cT_r^{\normalfont\text{pre}}\to \cE$ is given by 
	$$\obs_0(u,C,\iota,\gamma,\eta) = (\eta,\gamma,\lambda_\cU(\iota,u)).$$
	where $\lambda_\cU$ is defined by \eqref{eq:slice-map}. The \emph{covering group} is $G := \normalfont\text{PO}(N+1)$; it acts by post-composition on the framings and the perturbations.
\end{construction}

\noindent
It remains to show that $\cT^{\normalfont\text{pre}}_r$ is a manifold near $\obs_0\inv(0)$ and that $\cE$ is in fact a vector bundle. For this, not any auxiliary datum will do.

\begin{definition}\label{de:unobstruced-aux} We call an auxiliary datum $\alpha = (\conn^X,\cO_X(1),p,\cU,\lambda,r)$ \emph{unobstructed} if for any $[u,C]\in \Mbar^{J,\beta}_{g,h;n,\ell}(X,L;\mu)$ there exists a framing $\cF$ of $\fL_u^{\otimes p}$ so that $\lambda_\cU(\iota,u) = 0$ and
	\begin{enumerate}[label=\roman*),leftmargin=20pt,ref=\roman*]
		\item the linearised operator associated to \eqref{eq:perturbed-cr} is surjective when restricted to 
		\begin{equation}\label{eq:allowed-variation}
			C^\infty(C,u^*(TX,TL))\oplus H^0(C,(\iota,u)^*(E_r,F_r));
		\end{equation}
		\item\label{eq:obstruction-bundle-well-defined} $H^1(C,(\iota,u)^*(E_r,F_r)) = 0$.
	\end{enumerate}
	In this case, we define the \emph{thickening} $\cT\sub \cT^{\normalfont\text{pre}}_r$ to be the subset of elements $(u,\iota,\gamma,\eta)$, where the linearisation of \eqref{eq:perturbed-cr} restricted to \eqref{eq:allowed-variation} is surjective and for which \eqref{eq:obstruction-bundle-well-defined} holds.
\end{definition}

\begin{remark}\label{lem:unobstructedness-differently} By the analogue of \cite[Lemma~4.17]{HS22}, an auxiliary datum $\alpha$ is unobstructed if and only if the canonical map $\cT^{\normalfont\text{pre}}_r\to \Mbar_{g,h;0,0}^{\,\wt J_E,(m,\beta)}(E_r,F_r)$
	maps to the locus of regular curves.
\end{remark}

\noindent
We summarise the existence of a global Kuranishi chart and its relevant properties in the following result. Compare with \cite[Proposition~5.1]{HS22}.

\begin{proposition}\label{prop:open-gw-global-chart-existence} Unobstructed auxiliary data exist. Given an unobstructed auxiliary datum $\alpha$, the following holds.
	\begin{enumerate}[label = \arabic*), leftmargin=20pt,ref=\arabic*]
		\item\label{gkc-indeed} There exists a section $\obs \cl \cT\to \cE$ so that $\cK_\alpha = (G,\cT/\cB,\cE,\obs)$ is a rel--$C^\infty$ global Kuranishi chart for $\Mbar_{g,h}^{J,\beta}(X,L)$ of the correct virtual dimension.
		\item\label{rel-smooth-obstruction-bundle} $\cE/\cB$ is a rel--$C^\infty$-vector bundle over $\cT/\cB$ and $\obs$ is relatively smooth.
		\item\label{rel-smooth-action} $G$ acts by rel--$C^\infty$ diffeomorphisms on $\cT$ and $\cE$.
	\end{enumerate}
\end{proposition}

\begin{proof}[Proof of Proposition~\ref{prop:open-gw-global-chart-existence}]
	For the first claim, it remains to show that our perturbations suffice to achieve regularity once $r$ is sufficiently large. This is a straightforward consequence of \cite[Lemma 3.18]{HS22} and the discussion in \cite[\textsection3.4]{KL06}, and is summarised in the following lemma. Note that the locus of $J$-holomorphic curves $(u,C,\iota,0,0)$ with $\lambda_\cU(u,\iota) = 0$ is independent of $r$.
	
	\begin{lemma}\label{prop:transversality-achieved} There exists $r \gg 1$ so that the following holds. For any $y = (u,C,\iota) \in \obs_0\inv(0)$, the linearised operator 
		$$D_u +\lspan{\cdot}\g d\tilde{\iota} \cl C^\infty(C,u^*(TX,TL)) \oplus H^0(C,(\iota,u)^*(E_r,F_r))\to \Omega^{0,1}_J((\wt C,\del \wt C),\wt u^*(TX,TL))$$\noindent
		is surjective and $H^1(C,(\iota,u)^*(E_r,F_r)) = 0.$
	\end{lemma}
	
	\begin{proof} Given $y$ as in the statement, let $C_\bC$ be the complex double of its domain and let $u^*(TX,TL)_\bC$ be the complex double of the Riemann-Hilbert bundle $(u^*TX,(u|_{\del C})^*TL)$ as in \cite[Theorem~3.3.13]{KL06}. The first statement is equivalent to the surjectivity of $$\Phi_y\cl H^0(C,(\iota,u)^*(E_r,F_r))\to H^1(C,u^*(TX,TL)) : \eta \mapsto [\lspan{\eta}\g d\iota].$$\noindent
		As $\iota$ doubles to a holomorphic map $\iota_\bC$, we have a commutative square 
		\begin{center}\begin{tikzcd}
				H^0(C,(\iota,u)^*(E_r,F_r))\arrow[r,"\Phi_y"]\arrow[d] &H^1(C,u^*(TX,TL))\arrow[d]\\
				H^0(C_\bC,E_{\iota,u})\arrow[r,"\Phi_{y,\bC}"]&H^1(C_\bC,u^*(TX,TL)_\bC)\end{tikzcd} \end{center}
		where $\Phi_{y,\bC}$ is defined analogously. The vertical maps are induced by the doubling construction of \cite[Theorem~3.3.13]{KL06} and biject onto the fixed point locus under the anti-holomorphic involutions. By the Riemann-Roch theorem, \cite[Theorem~C.1.1]{MS12} and the argument of \cite[Proposition~6.26]{AMS21}, there exists a minimal $0 \leq r_y< \infty$ so that $\Phi_{y,\bC}$ is surjective for any $r\geq r_y$; note that they denote the parameter by $k$ instead of $r$. To see that $\sup r_y <\infty$, one can argue as in \cite[\textsection4.3]{HS22}, using the (linear) gluing analysis of \textsection\ref{sec:gluing}. Concretely, one shows that the two conditions of the statement are satisfied for $y \in \cT_r$ if and only if its image in $\Mbar_{g,h;0,0}^{\,\wt J_E,(m,\beta)}(E_r,F_r)$ is regular. Thus, one can use that regularity is an open condition, even when gluing, thanks to Proposition \ref{prop:sc-implicit-function}.
	\end{proof}

\begin{lemma} The image $\cT'$ of $\cT$ in $\Mbar_{g,h}^{\wt J_E,(m,\beta)}(E,F)$ is an open subset of the locus of automorphism-free regular curves. In particular, it is a topological manifold with boundary. Moreover, it admits a canonical rel--$C^\infty$ structure with corners over $\cB$. 
\end{lemma}

\begin{proof} 
	Define the functor $\fF\cl (C^\infty/\cdot)^{\text{op}}\to \Set$ by letting $\fF(Y/T)$ be the set of diagrams 
	\begin{equation}
		\begin{tikzcd}
			Y  \arrow[d,""]&\cC_Y\arrow[l,""]\arrow[d,""]\arrow[r,"F"] & E\arrow[d,"p_{\bP^N}"]\\ 
			T & \cC_T \arrow[r,"f"]\arrow[l,"\pi_T"] & \bP^N
		\end{tikzcd} 
	\end{equation}
	with the following properties.
	\begin{enumerate}[label=(\arabic*),leftmargin=20pt,ref=\arabic*]
		\item The diagram
		\begin{center}
			\begin{tikzcd}
				\cC_T \arrow[d,"\pi_T"] \arrow[r,"f"] & \bP^N \\
				T & 
			\end{tikzcd}
		\end{center}
		is the pullback of $\cC\to \cB$ along a continuous map $T\to\cB$ and the first square is cartesian.
		\item The morphism $(F,f)\cl \cC_Y/\cC_T\to E/\bC P^N$ is of class rel--$C^\infty$ and maps $\del\cC_Y\to F/\bR P^N$.
		\item\label{functor-submersion} For each $y\in Y$ mapping to $t\in T$, the restriction 
		\begin{align*}
			F|_y\cl \pi_T^{-1}(t)\to E    
		\end{align*}
		is a regular pseudo-holomorphic stable map to $E$ representing $\tilde \beta$. Moreover, the projection map 
		$$\ker(D\delbar (F|_y)) \to \ker(D\delbar (f|_t))$$\noindent is surjective.
	\end{enumerate} 
	The proof of \cite[Proposition~3.4]{HS22} carries over verbatim to our setting. The only adaptation to be made is replacing \cite[Theorem~2.19]{HS22}, itself a recollection of \cite{Swa21}, by Theorem \ref{thm:representable}. 
\end{proof}

\begin{lemma}\label{} $\cE\to \cT$ is a rel--$C^\infty$ vector bundle.
\end{lemma}

\begin{proof}Recall the decomposition 
	\begin{equation}\label{eq:obs-bundle} \cE = \cH\oplus \hodge \oplus \fp\end{equation} of the obstruction bundle, where $\cH$ has fibre $H^0(C,(\iota,u)^*(E_r,F_r))$ over $(u,\iota,C,\gamma,\eta)\in \cT$. 
	 The fact that $\cH \to \cT$ is of class rel--$C^\infty$ and the section $\cT\to \cH$ is relatively smooth follows from a straightforward adaption of \cite[Theorem~2.22]{HS22}. Meanwhile $\hodge \oplus \fp \to \cT$ is the pullback of a smooth vector bundle on the base space and thus admits a canonical rel--$C^\infty$ structure.
\end{proof}

\noindent
We defined the obstruction section $\obs_0 \cl \cT\to \cE$ to be 
$$\obs_0(\iota,u,\gamma,\eta) =(\eta,\gamma,\text{i}\log\lambda_\cU(\iota,u)).$$\noindent
The first two components are relatively smooth by \cite[Theorem~2.22]{HS22} and the definition of the rel--$C^\infty$ structure on $\cT$ through \cite[Proposition~5.6]{HS22}. However, the map $\lambda_\cU$ is \textit{a priori} only continuous. To see this, recall the three ingredients in the definition of this map: the stabilisation maps $\stb_{U_i}\cl U_i \to \cB_{[3d]}$, the map $\lambda\cl\cB_{[3d]}\to \cG/G$ and the bump functions $\chi_i$. The maps $\stb_{U_i}$ are relatively smooth by \cite[Proposition~5.7]{HS22}, respectively a straightforward adaptation thereof and the map $\lambda$ is smooth by construction. The bump functions $\chi_i$ are, however, only continuous per se.

\begin{lemma}\label{lem:mollification-of-obstruction-section} There exists a rel--$C^\infty$ $G$-invariant function $\vartheta \cl \cT\to \fp$ so that $$\vartheta\inv(0)\cap \{\gamma = 0,\eta = 0\} = \obs_0\inv(0).$$
\end{lemma}

\begin{proof}  Our argument is inspired by \cite[Lemma~4.55]{AMS23}. Let $\wh\obs \cl \cT\to \fp$ be the projection of $\obs_0$ onto the last summand of $\cE$. All functions we define from now on will be rel--$C^\infty$ unless specified otherwise. Endow $\cT$ with a $G$-invariant metric $d\cl \cT\times \cT\to \bR_{\geq 0}$ and choose $\epsilon > 0$ sufficiently small so that the ball $B_{2\epsilon}(x)$ admits a (possibly non $G$-invariant) rel--$C^\infty$ chart $\phi_x \cl B_{2\epsilon}(x) \to \bR^k$ for any $x \in \obs_0\inv(0)$. Fix $x_1,\dots,x_m$ so that $B_\epsilon(x_1),\dots,B_\epsilon(x_m)$ cover $\obs_0\inv(0)$. Shrinking $\cT$, we may assume without loss of generality that they form an open cover of $\cT$. Let $\{f_1,\dots,f_m\}$ be a not necessarily $G$-invariant partition of unity subordinate to this open cover, and fix a smooth bump function $T \cl \bR^k \to [0,1]$ with $T(0) = 1$, $\text{supp}(T)\sub B_1(0)$ and $\int_{B_1(0)}T dy =1$. Given $r > 0$, set $T_r(y) = r^{-k}T(\frac{1}{r}x)$. Let $\delta \cl \cT\sm\obs_0\inv(0)\to (0,\epsilon)$ be a function so that 
	\begin{equation}\label{} \norm{\wh\obs(x) -\int_{\bR^k}T_{\delta(x)}(\phi_{x_j}(x)-\phi_{x_j}(y))\wh\obs(\phi_{x_j}\inv(y)) dy} < \frac1{2m}\norm{\wh\obs(x)}.\end{equation}
	for any $1 \leq j \leq m$ and $x \in B_\epsilon(x_j)\cap \{\gamma = 0,\eta = 0\}\sm\obs_0\inv(0)$.
	The function $\wt\vartheta_j \cl B_\epsilon(x_j)\to \fp$ given by 
	$$\wt\vartheta_j(x) = f_j(x) \int_{\bR^k}T_{\delta(x)}(\phi_{x_j}(x)-\phi_{x_j}(y))\wh\obs(\phi_{x_j}\inv(y)) dy$$\noindent
	extends to a rel--$C^\infty$ function on all of $\cT$. Denoting the normalised Haar measure on $G$ by $\mu_G$, we can define $\wt\vartheta$ by  
	$$\wt\vartheta(x) := \sum_{j = 1}^m\int_G \wt\vartheta_j(g\inv\cdot x)\,d\mu_G(g)$$\noindent
	 to obtain a $G$-invariant function satisfying for $x \in \{\gamma = 0,\eta = 0\}\sm\obs_0\inv(0)$ the inequality
	 
	\begin{align*}\norm{\wh\obs(x)-\wt\vartheta(x)} &\leq \sum_{j = 1}^m \int_G f_j(g\inv\cdot x)\norm{\wh\obs(g\inv\cdot x) -\int_{\bR^k}T_{\delta(x)}(\phi_{x_j}(g\inv\cdot x)-\phi_{x_j}(y))\wh\obs(\phi_{x_j}\inv(y)) dy}\,d\mu_G(g) 
		\\& < \sum_{j = 1}^m \int_G f_j(g\inv\cdot x)\frac12\norm{\wh\obs(g\inv\cdot x)} d\mu_G(g) 
		\\&= \frac12\norm{\wh\obs(x)}\int_G \sum_{j = 1}^m f_j(g\inv\cdot x) d\mu_G(g) 
		\\&= \frac12\norm{\wh\obs(x)}.\end{align*} 
	In particular, $\wt\vartheta(x) \ne 0$ if $\wh\obs(x)\neq 0$ for $x \in \{\gamma= 0,\eta = 0\}$.
	Let $\rho \cl \cT\to \bR_{\geq 0}$ be a $G$-invariant map so that $\rho$ vanishes along $\obs_0\inv(0)$ with all derivatives and $\rho\inv(0) = \obs_0\inv(0)$. Then $\vartheta := \rho\wt \vartheta$ is the desired function.
\end{proof}

\noindent
Define the final obstruction section $\obs\cl \cT\to \cE$ by 
$$\obs(\iota,u,\gamma,\eta) =(\eta,\gamma,\vartheta(\iota,u,\gamma,\eta)).$$

\noindent By the same argument as in \cite[\textsection 5.2]{HS22}, the group action of $G$ on $\cT$ and $\cE$ is continuous and by rel--$C^\infty$ diffeomorphisms. Finally, the reasoning in the proof of \cite[Lemma~5.10]{HS22} shows that $(G,\cT,\cE,\obs)$ is indeed a global Kuranishi chart for $\Mbar_{g,h;0,0}^{\,J\,\beta}(X,L)$.\end{proof}

\subsection{Orientation lines}
In contrast to the closed case, where the global Kuranishi chart is canonically stably complex, \cite{HS22}, orientability in the open setting is not always given and depends on additional choices, as described in detail in \cite{Ge13}. This is the reason for our specific choice of integers in the definition of auxiliary datum, Definition \ref{de:aux-choices-defined}. Recall that the \emph{orientation line} of a global Kuranishi chart is given by 
\begin{equation}\label{} 
\orl(\cK) = \orl(\cT)\otimes \orl(\fg)\dul \otimes \orl(\cE)\dul,
\end{equation}
see also Definition \ref{de:orientation-gkc}.

\begin{proposition}\label{prop:orientation-of-gkc}
	For each $[u,C]\in \Mbar_{g,h;0,0}^{\,J,\,\beta}(X,L)$ with lift $\wh u\in \obs\inv(0)$, there exists a canonical isomorphism 
	\begin{equation}\label{eq:orientation-lines} 
		\orl(\cK_\alpha)_{\wh u}\cong \orl(D\delbar_J(u))\otimes \orl(\delbar_{\,TC}\dul)\end{equation}
	of $\bZ/2$ torsors. Thus, if $(V,\gamma_)$ is a relative spin structure for $L$, then $\cK_\alpha$ is canonically oriented. 
\end{proposition}

\begin{lemma}\label{lem:obstruction-bundle-orientable} The first summand $\cH$ in the decomposition \eqref{eq:obs-bundle} has even rank. \end{lemma}

\begin{proof} 	We use the Riemann-Roch theorem, \cite{MS12}, to compute, at a point $(u,C,\iota,\alpha,\eta)\in \cT$
	
	\begin{align*}\label{}\rank(\cH) &=\rank(F)\chi(C)+ \mu(C,(\iota,u)^*(E,F))  \\&= \binom{N+r}{r}\lbr{Nn(2-2g-h)+ N\mu(C,u^*(TX,TL)) + nc_1(\iota_\bC^*(T^*\bC P^N\otimes_\bC \cO(r))}\\& \equiv \binom{N+r}{r}\lbr{N\mu_L(\beta)+n(N+1)(r+1)m}\\&\equiv\binom{N+r}{r}N\mu_L(\beta)\quad (\normalfont\text{mod }2) \end{align*}
	If $N$ is even, then the last expression vanishes. If $N$ is odd, then $\binom{N+r}{r}$ is even due to the Vandermonde convolution $\binom{n}{k} = \sum_{j = 0}^{n}\binom{m}{j}\binom{n-m}{k-j}$ (for any $0 \leq m\leq n$) since $r$ is odd.
\end{proof}

\begin{lemma}\label{lem:vertical-orientation} 
	The thickening $\cT$ is orientable exactly if the vector bundle $\cH$ is orientable. If this is the case, the $G$-action on the total space of $\cH \to \cT$ is orientation-preserving. 
\end{lemma}

\begin{proof} We first show that there exists a canonical isomorphism 
	\begin{equation}\label{} 
		\orl( T_{\cT/\cB}) \cong \orl(\delbar_J)\otimes \orl(\cH).
		\end{equation}
	Indeed, the vertical tangent space of $\cT/\cB$ at $y$ is given by the kernel of 
	$(D_u+\lspan{\cdot}\g d\wt\iota)\oplus \delbar_{(\iota,u)^*(E,F)}$,
which agrees with the kernel of $$D_u+\lspan{\cdot}\g d\wt\iota \cl C^\infty(C,u^*(TX,TL))\oplus H^0(C,(\iota,u)^*(E,F))\to \Omega^{0,1}(\wt C,\wt u^*(TX,TL)).$$ 
Thus the claim follows from Lemma \ref{lem:orientation-of-sum}. Hence, the first assertion follows from this isomorphism combined with Lemma \ref{lem:orientation-base} and Lemma \ref{lem:obstruction-bundle-orientable}. The second claim follows from Lemma \ref{lem:action-base-space} and the fact that any block matrix of the form $\begin{pmatrix} A & 0\\ 0& A\end{pmatrix}$ has positive determinant. 
\end{proof}

\begin{proof}[Proof of Proposition~\ref{prop:orientation-of-gkc}] We have the canonical isomorphisms
	\begin{align}\label{eq:virtual-orientation} 
	\notag \orl(\cK) &\cong \orl(T_{\cT/\cB})\orl(\cB)\orl(\fg)\dul\orl(\cE)\dul\\ \notag
			&\cong\orl(\delbar_J)\orl(\cH)\orl(\cB)\orl(\fg)\dul\orl(\fp)\dul\orl(\bE\dul)\dul\orl(\cH)\dul\\\notag
			&\cong \orl(\delbar_J)\orl(\bE\dul)\orl(\fp\fg\fl)\orl(\Mbar)\orl(\fg)\dul\orl(\fp)\dul\orl(\bE\dul)\dul\orl(\cH)\dul\orl(\cH)\\
			&\cong \orl(\delbar_J)\orl(\Mbar)
	\end{align}
where we are throughout using the Koszul sign rule and the isomorphism $\orl(V)\dul\orl(V) \cong \orl(0)$.
	Thus, if $L$ is equipped with a relative spin structure and the boundary components of the curves are ordered, we obtain a canonical orientation of $\cK$ by \textsection\ref{sec:orientations}.
\end{proof}

\begin{remark}\label{rem:right-gkc} Suppose $L$ is relatively spin. It will be useful to have a global Kuranishi chart $\wh\cK = (\wh G,\wh\cT/\wh\cB,\wh\cE,\wh\obs)$ for $\Mbar_{g,h;k,\ell}^{\,J,\,\beta}(X,L)$, so that 
	\begin{enumerate}[\normalfont a),leftmargin=*]
		\item\label{orientable-thickening} $\wh\cT$ is orientable and the covering group $\wh G$ acts by orientation-preserving rel--$C^\infty$ diffeomorphisms,
		\item\label{same-as-curves} $\dim(\wh\cB)-\dim(G) \equiv \dim(\Mbar_{g,h;k,\ell})$ (mod $2$)
		\item\label{even-rank} $\dim(\wh\cT) \equiv\vdim(\Mbar_{g,h;k,\ell}^{\,J,\,\beta}(X,L))$ mod $2$ or, equivalently, $\rank(\wh\cE)$ is even.
	\end{enumerate}
We obtain this from the global Kuranishi chart $\cK$ above by letting $\wh\cK$ be the stabilisation of $\cK$ by the bundle $\cE \to \cT$. It follows from the previous sections that $\wh\cK$ has Properties \eqref{orientable-thickening}-\eqref{even-rank}. Moreover, the evaluation and stabilisation maps pullback to $\wh\cK$. In \cite{HH1}, it will be useful to use $\wh\cK$ instead of $\cK$. 
 \end{remark}

\subsection{Moduli space of maps with marked points}\label{subsec:marked-points}
Given an unobstructed auxiliary datum $\alpha$ for $\Mbar^{J,\beta}_{g,h;0,0}(X,L)$, let $\cK_\alpha = (G,\cT_\alpha/\cB,\cE_\alpha,\obs_\alpha)$ be the global Kuranishi chart with corners constructed in \textsection\ref{subsec:overview}. Let 
$$\cB_{k,\ell}\sub \Mbar_{g,h;k,\ell}^{\,m}(\bC P^N,\bR P^N)$$ 
be the preimage of $\cB$ under the map $\pi_{\lc{k},\lc{\ell}}$ that forgets all marked points. As constant discs and spheres are unobstructed, $\cB_{k,\ell}$ is a smooth manifold with corners of the expected dimension. 

\begin{proposition}\label{prop:marked-points-gkc-simple}
	The pullback global Kuranishi chart 
	\[\cK_{\alpha,k,\ell} := \cK_\alpha \times_\cB\cB_{k,\ell}= (G,\cT_\alpha\times_{\cB}\cB_{k,\ell}/\cB,\cE_\alpha\times_{\cB}\cB_{k,\ell},\obs_\alpha\times\ide)\] is a rel--$C^\infty$ global Kuranishi chart with corners for $\Mbar^{J,\beta}_{g,h;k,\ell}(X,L)$.
	\begin{enumerate}[label=\arabic*),leftmargin=20pt,ref=\arabic*]
		\item (Forgetful maps) Forgetting a marked point induces a morphism \[\cK_{\alpha,k,\ell}\to \cK_{\alpha,k-1,\ell}\qquad \text{ and }\qquad\cK_{k,\ell}\to \cK_{k,\ell-1},\]
		of rel--$C^\infty$ global Kuranishi charts as in \cite[Definition~A.1]{Hir23}. They are rel--$C^\infty$ submersions away from subsets of codimension $2$, respectively codimension $1$.
		\item (Evaluation maps) The evaluation map $\eva$ on $\Mbar_{g,h;k,\ell}^{J,\beta}(X,L)$ lifts to a relatively smooth $G$-invariant map $\eva\cl \cT_{\alpha,k,\ell}\to X^k\times L^{\ell}$.
	\end{enumerate} 
\end{proposition}

\begin{proof} As the forgetful map $\cB_{k,\ell}\to \cB$ is smooth, $\cK_{\alpha,k,\ell}$ is a rel--$C^\infty$ global Kuranishi chart. It remains to show that the forgetful map 
	\[\obs_{\alpha,k,\ell}\inv(0)/G\to \Mbar^{J,\beta}_{g,h;k,\ell}(X,L)\]
	is a homeomorphism. The argument is the same as in the proof of \cite[Lemma~5.2]{HS22}. The assertion about the forgetful maps is a consequence of the fact that the forgetful maps on the level of base spaces are smooth. 
	The isomorphism \eqref{eq:virtual-orientation} yields a canonical isomorphism
	\begin{align}\label{eq:orientation-with-marked-points}\notag\orl(\cK_{\alpha,k,\ell})&\cong \orl(\delbar_J)\otimes \bigotimes\limits_{i = 1}^k\orl(TC)\otimes \bigotimes\limits_{i = 1}^h\bigotimes\limits_{j = 1}^{\ell_i}\orl(T(\del C)_i)\otimes \orl(\delbar_{\,TC})\dul\\&\cong \orl(\delbar_J)\otimes  \bigotimes\limits_{i = 1}^h\bigotimes\limits_{j = 1}^{\ell_i}\orl(T(\del C)_i)\otimes \orl(\delbar_{\,TC})\dul,\end{align}
	which orients the global Kuranishi charts canonically and so proves Theorem \ref{thm:open-gw-global-chart}\eqref{gkc-orientation}.\par 

	We now turn to the evaluation maps. By construction, the \emph{interior} and \emph{boundary evaluation maps} 
	$$\evai_j \cl \Mbar^{J,\beta}_{g,h;k,\ell}(X,L)\to X$$\noindent
	and 
	$$\evab_i \cl \Mbar^{J,\beta}_{g,h;k,\ell}(X,L)\to L$$\noindent
	lift to continuous $G$-invariant maps $\evai_j \cl \cT_{\alpha,k,\ell}\to X$ and $\evab_i \cl \cT_{k,\ell}\to L$. They are relatively smooth by the same argument as used in \cite{HS22}.
	\end{proof}

\noindent	
A single \emph{boundary} evaluation map is a submersion when the auxiliary datum is chosen appropriately as in \cite[Proposition~3.7]{Hir23}. 
However, this joint evaluation map is usually not a submersion. 
To circumvent this problem, we use a general trick for global Kuranishi charts, cf. \cite[Lemma~4.25]{AMS21}.

\begin{lemma}\label{lem:evaluation-relvative-submersion}
	After stabilising $\cK_{\alpha,k,\ell}$, we can extend the evaluation map $\eva$ to a relative submersion.
\end{lemma}

\begin{proof}
	\label{pf:stabilise-by-exponential-map} Stabilise the global Kuranishi chart $\cK_{\alpha,k,\ell}$ by $\cW^{k,\ell}_\alpha := \eva^*(TX^{\bigoplus k}\boxplus TL^{\boxplus \ell})$. Fix a Riemannian metric $g$ on $X$ with respect to which $L$ is totally geodesic and let $\exp$ be the exponential map of $g$. Then we define the `new' evaluation maps to be
	\begin{equation}\label{de:new-evaluation}\evai_j(y,v_*,v'_*) = \exp_{\evai_j(y)}(v_j) \qquad \qquad \evab_i(y,v_*,v'_*) = \exp_{\evab_i(y)}(v'_i),\end{equation}
	which manifestly yield a submersion to $X^k\times L^\ell$ on the total space of $\cW^{k,\ell}_\alpha$.
\end{proof}

\begin{remark}
	Stabilising by the tangent bundle of $X$ and $L$ yields very strong submersivity properties. The cost of this is that even operations at energy zero, which could have been defined without relying on a global Kuranishi chart construction, end up being deformed.
\end{remark}

\noindent
 We now discuss the change in orientation that occurs when permutating boundary components, respectively interior and boundary marked points.

\begin{lemma}\label{lem:permutation-marked-points} Let $\alpha$ be an unobstructed auxiliary datum for $\Mbar_{g,h;0,0}^{J,\,\beta}(X,L)$ and $n := \dim(L)$.
	\begin{enumerate}[label=\arabic*),leftmargin=20pt,ref=\arabic*]
		\item\label{permute-circle} Let $\tau = (i,i+1)$ be a transposition in the symmetric group $S_h$. The associated isomorphism $\cK_{\alpha,k,\ell}\to \cK_{\alpha,k,\tau_*\ell}$ of global Kuranishi charts changing the ordering of the boundary components has orientation sign $$(-1)^{n+1+\ell_i\ell_{i+1}}.$$
		\item\label{permute-interior} The isomorphism $\cK_{\alpha,k,\ell}\to \cK_{\alpha,k,\ell}$ of global Kuranishi charts, which changes the label of two interior marked points, is orientation-preserving.
		\item\label{permute-boundary} The isomorphism $\cK_{\alpha,k,\ell}\to \cK_{\alpha,k,\ell}$ of global Kuranishi charts that changes the label of the two marked points $x^b_{i,j_1}$ and $x^b_{i,j_2}$ with $j_1 < j_2$, has orientation sign $(-1)^{\normalfont\text{sign}(\tau_{j_1,j_2})}$, where $\tau_{j_1j_2}$ is the transposition of $j_1$ and $j_2$.
	\end{enumerate}
\end{lemma}

\begin{proof} The first claim, in the case of $(k,\ell) = (0,0)$, follows from the definition of the orientation on the thickening, in particular, the isomorphism \eqref{eq:virtual-orientation}, and the way the ordering of the boundary components induces an orientation on the index bundle of a Cauchy-Riemann operator, see \cite[CROrient $1\fo\fs$(2)]{CZ24} or the discussion before Example 2.3.10 in \cite{WW17}. In the case of marked points, we can write an open subset of $\cT_{k,\ell}$ as 
	\begin{equation}\label{eq:open-stratum}U \cong V\times (C\sm \del C)^k \times \prod_{j = 1}^{h}(\del C)_j^{\ell_j} \end{equation}
		for any open subset $V\sub \cT$. The additional term $\ell_i\ell_{i+1}$ thus arises due to the permutation of the factors of the product.
		The other two cases can be deduced from the action on $U$ in \eqref{eq:open-stratum}.
\end{proof}
\subsection{Uniqueness up to equivalence and cobordism}\label{subsec:uniqueness} We will need a somewhat more general notion than that of equivalence between global Kuranishi charts, possibly with boundary, see also \cite[\textsection2.3]{HS22}. Let us thus make the following definition.

\begin{definition}\label{} Let $Y$ be a smooth manifold. A \emph{global Kuranishi chart over $Y$} is a rel--$C^\infty$ global Kuranishi chart $\cK$ equipped with a $G$-invariant rel--$C^\infty$ map $f \cl \cT\to Y$. 
\end{definition}

\begin{definition}\label{de:equivalence of gkc over Y} We call two oriented rel--$C^\infty$ global Kuranishi charts  $(\cK,f)$ and $(\cK',f')$ \emph{equivalent} if they are related by the equivalence relation generated by 
	\begin{enumerate}[\normalfont i),leftmargin=20pt,ref=\roman*]
		\item \label{de:equivalence of gkc over Y germ equiv} \textit{(germ equivalence)} if $U \sub \cT$ is a $G$-invariant open neighbourhood of $\obs\inv(0)$, then $$(\cK,f)\sim (\cK|_U) := (G,U,\cE|_U,\obs|_U);$$
		\item \label{de:equivalence of gkc over Y isomorphism}\textit{(isomorphism)} if there exists an isomorphism $G\to G'$, an equivariant rel--$C^\infty$ diffeomorphism $\psi \cl \cT/\cB\to \cT'/\cB'$ and an equivariant \emph{fibrewise affine} isomorphism $\wt\psi \cl \cE\to \psi^*\cE'$, which is linear over $\obs\inv(0)\cup\psi\inv({\obs'}\inv(0))$, so that $\wt\psi \g \obs = \obs'\g \psi$, then $\cK\sim \cK'$;
		\item \label{de:equivalence of gkc over Y stabilisation}\textit{(stabilisation)} if $p \cl \cW\to \cT$ is a $G$-vector bundle that is rel--$C^\infty$ with respect to $\cB$, then 
		$$\cK \sim (\cK_\cW,f_\cW):= (G,\cW/\cB,p^*\cE\oplus p^*\cW,p^*\obs\oplus \Delta_\cW);$$
		\item \label{de:equivalence of gkc over Y group enlargement}\textit{(group enlargement)} if $q \cl \cP\to \cT$ is a principal $G'$-bundle rel--$C^\infty$ with respect to $\cB$ and equipped with a compatible $G$-action, then $$\cK \sim (\cK_\cP,f_\cP):= (G\times G',\cP,q^*\cE,q^*\obs);$$
		\item \label{de:equivalence of gkc over Y base modification}\textit{(base modification)} if the map $\cT\to \cB$ factors through a smooth submersion $\wt \cB \to \cB$ so that $f\cl \cT/\wt\cB\to Y$ is still a relative submersion, then 
		$$\cK\sim \cK_{/\wt\cB} := (G,\cT/\wt\cB,\cE,\obs).$$ 
	\end{enumerate}
In each case, there exists, according to our conventions in \textsection\ref{subesc:gkc-conventions}, a canonical isomorphism 
\begin{equation}\label{eq:equivalent-det}\det(\cK)\cong \det(\cK') \end{equation} 
for equivalent global Kuranishi charts $\cK$ and $\cK'$. We call them \emph{oriented equivalent} if $\cK$ and $\cK'$ are oriented and the isomorphism~\eqref{eq:equivalent-det} maps one orientation to the other.
\end{definition}

\begin{proposition}\label{prop:uniqueness-up-to-equivalence} Suppose $\alpha_i = (\conn^{X,i},\cO_{X,i}(1),p_i,\cU_i,\lambda_i,r_i)$ is an unobstructed auxiliary datum for $\Mbar_{g,h}^{J,\beta}(X,L)$ for $i \in \{0,1\}$. Then $\cK_{\alpha_0}$ and $\cK_{\alpha_1}$ are rel--$C^\infty$ equivalent. If $L$ admits a relative spin structure, then they are oriented equivalent.
\end{proposition}

\begin{proof} Define the doubly thickened rel--$C^\infty$ global Kuranishi chart 
		$$\cK_{01} = (G_0\times G_1,\cT_{01}/\cB_{01},\cE_{01},\obs_{01})$$ 
		by letting $\cB_{01}$ be the preimage of $\cB_0\times\cB_1$ under the map 
		$$\Mbar_{g,h}^{(m_0,m_1)}(\bC P^{N_0}\times \bC P^{N_1},\bR P^{N_0}\times \bR P^{N_1})\to \Mbar_{g,h}^{m_0}(\bC P^{N_0}, \bR P^{N_0})\times \Mbar_{g,h}^{m_1}(\bC P^{N_1}, \bR P^{N_1}).$$ 
		By the same argument as in Proposition \ref{prop:base-space} and \cite[Lemma~3.2]{HS22}, $\cB_{01}$ is a smooth manifold (with boundary) of the expected dimension and the natural forgetful maps $\pi_i\cl \cB_{01}\to\cB_i$ are submersions with  
		\begin{equation}\label{eq:orientation-double-base}\ker(d\pi_0([\iota,C])) = \ker(D\delbar(\iota_1)) \end{equation}
		and similarly for $\pi_1$.
		In particular, by Lemma \ref{lem:orientation-base}, we see from the short exact sequence  $$0 \to \ker(d\pi_0)\to T\cB_{01} \to \pi_0^*T\cB_0\to 0$$ 
		that $\cB_{01}$ is orientable and we have an isomorphism
		$$\Lambda^{\normalfont\text{max}}(T_{[\iota,C]}\cB_{01}) \cong \det(\delbar_{TC}) \otimes_\bR\Lambda^{\normalfont\text{max}}(\fp\fg\fl_\bR(N_0+1))\otimes \det(\hodge)\otimes \Lambda^{\normalfont\text{max}}(\fp\fg\fl_\bR(N_1+1))\otimes\det(\hodge),$$
		which is canonical up to multiplication by positive scalars.
		Define $\cZ_{01}\to \cB_{01}$ as in \textsection\ref{subsec:reducing-group}. Then 
		$$\cZ_{01} = \cB_{01}\times_{\cB_i}\cZ_i$$ 
		for $i \in \{0,1\}$.
		Let $(E_i,F_i)$ be defined by \eqref{eq:perturbation-bundle} using the unobstructed auxiliary datum $\alpha_i$. Then, the thickening $\cT_{01}$ consists of tuples $\{(\zeta,\gamma_0,\gamma_1,u,\eta_0,\eta_1)\}$, where
		 \begin{itemize}[leftmargin=20pt]
		 	\item $(\zeta,u)\in \cZ_{01}$ with $\zeta$ mapping to $[\iota_0,C]$, respectively $[\iota_1,C]$;
		 	\item $\gamma_i \in H^1(C,(\bC,\bR))$ satisfies 
		 	$$[\iota_i^*\cO_{\bC P^{N_i}}(1)]\otimes [\fL_{i,u}]^{\otimes -p_i} = \exp(\gamma_i),$$ where $\exp \cl H^1(C,(\bC,\bR))\to \Pic(C,\del C)$ is the exponential map,
		 	\item $\eta_i \in H^0(C,(\iota_i,u)^*(E_i,F_i))$, and $\eta_0$ and $\eta_1$ satisfy
		 	\begin{equation*}\label{}\delbar_J \wt u +\lspan{\eta_0}\g d\wt\iota_0 + \lspan{\eta_1}\g d\wt\iota_1  = 0\end{equation*}
		 	on the normalisation $\wt C\to C$.
		 \end{itemize}
	The fibre of $\cE_{01}$ over $(\zeta,\gamma_0,\gamma_1,u,\eta_0,\eta_1)$ is given by
		\begin{align}\label{double-obstruction}
			\bigoplus_{i=0,1}\left( H^0(C,(\iota_i,u)^*(E_i,E_i')) \oplus H^1(C,(\bC,\bR))\oplus \fp(N_i+1)\right)
		\end{align}
		and we set 
		$$\obs_{01}(\zeta,\gamma,u,\eta) = (\gamma_i,\eta_i,\vartheta_i(\zeta,u,\gamma,\eta))_{i = 0,1},$$
		where $\vartheta = (\vartheta_0,\vartheta_1) \cl \cT_{01}\to \fp_0 \oplus \fp_1$ is defined as follows. Let $U_i \sub \cT_{01}$ be an open $G_0\times G_0$-invariant neighbourhood of $\wt \cT_i:= \{\gamma_j= 0,\eta_j = 0\}$, where $\{i,j\} = \{0,1\}$, so that there exists an equivariant rel--$C^\infty$ retraction $r_i \cl U_i \to \wt \cT_i$. Let $\wt \vartheta_i$ be the pullback of the section $\vartheta_i$ on $\cT_i$ constructed in Lemma~\ref{lem:mollification-of-obstruction-section}.
		Let $t_i \cl \cT_{01}\to [0,1]$ be a $G_{01}$-invariant rel--$C^\infty$ map that is identically $1$ in a neighbourhood of $\wt\cT_i$ and is supported in $U_i$. We claim that $\vartheta := (t_0 r_0^*\wt\vartheta_0,t_1 r_1^*\wt\vartheta_1)$ has the desired properties in a neighbourhood of $\obs_0\inv(0) = \obs_1\inv(0)\sub \cT_{01}$. Indeed, restricting to the interior of the support of $t_0$ and $t_1$, we see that 
		$$\vartheta\inv(0)\cap \{\gamma_0,\gamma_1= 0,\eta_0 = 0,\eta_1\} = \wt\vartheta_0\inv(0)\cap \wt\vartheta_0\inv(0)\cap \{\gamma_0,\gamma_1= 0,\eta_0 = 0,\eta_1\}$$
		which is exactly the zero locus we need by Lemma~\ref{lem:mollification-of-obstruction-section}.\par
		By abuse of notation, we denote the vector bundle summands for $i=0,1$ from \eqref{double-obstruction} by $\cE_i$ and denote the projection of $\obs_{01}$ to the respective summand by $\obs_i$. The group $G_0\times G_1$ acts on $\cT_{01}\to\cB_{01}$ and $\cE_{01}$ in the evident way.
		The arguments of Proposition \ref{prop:open-gw-global-chart-existence} show that $\cT_{01},\cE_{01}$ and $\obs_{01}$ are rel--$C^\infty$ with respect to the forgetful maps to $\cB_0$ and $\cB_1$. By symmetry, it remains to show that $\cK_0$ and $\cK_{01}$ are equivalent. By (base modification), we may replace $\cK_{01}$ by $\cK:= (G_0\times G_1,\cT_{01}/\cB_0,\cE_{01},\obs_{01})$. The proof of \cite[Lemma~6.1]{HS22} carries over to our setting, showing that the vertical linearisation 
		$$ D\obs_1\cl T_{\cT_{01}/\cB_0}\to \cE_1$$
		is surjective in a neighbourhood $U$ of $\obs_{01}\inv(0)$. Then, 
		$\cK' := (G_0\times G_1,U\cap {\obs'}\inv(0)/\cB_0,\cE_0,\obs_0)$ 
		is a rel--$C^\infty$ global Kuranishi chart, and $\cK'$ is related to $\cK_0$ by (group enlargement). Finally, $\cK_{01}$ and $\cK'$ are related by (stabilisation) via $\cE_1\to \obs_1\inv(0)$ by Lemmas \ref{lem:tubular-neighbourhood-for-stabilisation-1} and \ref{lem:tubular-neighbourhood-for-stabilisation-2} since the tubular neighbourhood theorem generalises to the relatively smooth setting.\par
		Suppose now that $L$ admits a relative spin structure $(V,\obs)$. Then, by the same argument as in Lemma \ref{lem:vertical-orientation}, $\cK_{01}$ is orientable. We will show that the horizontal maps in the diagram
		\begin{equation}\label{eq:uniquness-orientation}\begin{tikzcd}
				\orl(\cK_0) \arrow[r,""]&\orl(\cK_{01}) & \orl(\cK_1) \arrow[l,""] \\ & \orl(D\delbar_J)\otimes \orl(\delbar_{TC})\dul \arrow[ur,""] \arrow[ul,]\arrow[u,""]\end{tikzcd} \end{equation}
	have the same sign when we endow the upper row via an orientation of $\det(D\delbar_J)\otimes \det(\delbar_{TC})\dul$. The orientation on $\cK_0$ at a point $y \in \cK_0$ corresponding to $\wh y \in \cK_{01}$ is given by 
	$$\fo_1 := \fo_D\wedge \fo_{\cH_0}\wedge \fo_{\fp\fg\fl_0}\wedge \fo_{\hodge}\wedge \fo_{\delbar\dul}\wedge \fo_{\cE_1}\wedge \fo_{\fg_1}\wedge \fo_{\fg_1}\dul \wedge \fo_{\fg_0}\dul\wedge  \fo_{\cE_1}\dul \wedge \fo_{\cE_0}\dul$$
	while the orientation induced via the vertical map on $\cK_{01}$ is 
	$$\fo_2 := \fo_D \wedge \fo_{\cH_0}\wedge \fo_{\cH_1} \wedge \fo_{\fp\fg\fl_0}\wedge \fo_{\hodge}\wedge \fo_{\fp\fg\fl_1}\wedge \fo_{\hodge}\wedge \fo_{\delbar\dul}\wedge \fo_{\fg_1}\dul \wedge \fo_{\fg_0}\dul \wedge \fo_{\cE_1}\dul \wedge \fo_{\cE_0}\dul.$$
	Thus it suffices to compare their first parts, determining an orientation on the thickening $\cT_{01}$. Then,
	\begin{align*} \fo'_1 &:= \fo_D\wedge \fo_{\cH_0}\wedge \fo_{\fp\fg\fl_0}\wedge \fo_{\hodge}\wedge \fo_{\delbar\dul}\wedge \fo_{\cE_1}\wedge \fo_{\fg_1}
		\\ & = \fo_D\wedge \fo_{\cH_0}\wedge \fo_{\fp\fg\fl_0}\wedge \fo_{\hodge}\wedge \fo_{\delbar\dul}\wedge \fo_{\cH_1}\wedge \fo_\hodge \wedge \fo_{\fp_1}\wedge \fo_{\fg_1}
		\\& = \fo_D \wedge \fo_{\cH_0}\wedge \fo_{\cH_1} \wedge \fo_{\fp\fg\fl_0}\wedge \fo_{\hodge}\wedge \fo_{\delbar\dul}\wedge \fo_{\hodge}\wedge \fo_{\fp\fg\fl_1}
		\\& = \fo_D \wedge \fo_{\cH_0}\wedge \fo_{\cH_1} \wedge \fo_{\fp\fg\fl_0}\wedge \fo_{\hodge}\wedge \fo_{\delbar\dul}\wedge \fo_{\hodge}\wedge \fo_{\fp\fg\fl_1}
		\\&= (-1)^{\ind(\delbar_{\,TC})((N_1+1)^2-1 + \rank(\hodge))+\rank(\hodge)((N_0+1)^2-1)}\fo_2'
	\end{align*}
	where $\fo_2'$ is the orientation of $\cT_{01}$ induced by $\fo_2$. Here the first equality just uses the definition of $\cE_1$, while the second equality holds because $\rank(\cH_1)$ is even by Lemma \ref{lem:obstruction-bundle-orientable}. We have $$\rank(\hodge) = 2g+h-1\equiv h-1\;(\normalfont\text{mod } 2),$$ while by \eqref{eq:dimension-base-lag}
	$$(N_i+1)^2 -1 \equiv N_i \equiv h-1\;(\normalfont\text{mod } 2).$$ 
	Therefore, $\fo_1' = (-1)^{h-1} \fo_2'$ and the first horizontal map in \eqref{eq:uniquness-orientation} has sign $(-1)^{h-1} $. In order to determine the sign of the second horizontal map, we also have to compute the parity of 
	$$\delta = \dim(G_0)\dim(G_1)+ \rank(\cE_0)\rank(\cE_1) + (\rank(\hodge)+\dim\fp\fg\fl_0) (\rank(\hodge)+\dim\dim\fp\fg\fl_1)$$
	as the sign $(-1)^{\delta}$ arises from the fact that the ordering of the groups and obstruction bundle summands given by $\fo_2$ in $\cK_{01}$ is not oriented equivalent to $\cK_1$. By the computation above and again Lemma \ref{lem:obstruction-bundle-orientable}, we see that 
	\begin{align*}\delta &\equiv \dim(G_0)\dim(G_1)+ (\rank(\hodge)+\dim\fp_0)(\rank(\hodge)+\dim\fp_1) \\&\equiv  \dim(G_0)\dim(G_1)+ (N+\dim\fp_0)(N+\dim\fp_1)\\&\equiv 0\; (\normalfont\text{mod }2) \end{align*}
	As the sign is independent of the auxiliary datum, $\cK_0$ and $\cK_1$ are oriented equivalent.
\end{proof}

\begin{lemma}\label{lem:tubular-neighbourhood-for-stabilisation-1} Suppose $\pi \cl \wt W\to B$ is a rel--$C^\infty$ vector bundle and $t \cl B\to \wt W$ is a section with $t \pf 0$ and $X := t\inv(0)$. Given any rel--$C^\infty$ retraction $r \cl B\to X$, there exist rel--$C^\infty$ diffeomorphisms $\psi \cl U\sub B\to W := \wt W|_X$ and $\wt\psi \cl \wt W|_U \to \pi^*W$ so that $\wt\psi \g t|_U = \Delta_W\g \psi$. Moreover, if a compact Lie group $G$ acts on $B$, $\pi$ is a $G$-vector bundle and $t$ and $r$ are $G$-equivariant, then we can choose $\psi$ and $\wt\psi$ to be $G$-equivariant as well. 
\end{lemma}

\begin{proof} Let $\Phi \cl \wt W|_U \to r^*W$ be an equivariant isomorphism and let $p' \cl r^*W\to W$ be the canonical projection. Define $\psi \cl B \to W$ by $\psi(y) = p'(\Phi(t(y)))$. As $\psi|_X$ is the zero section and $d\psi(x)$ is an isomorphism for each $x \in X$ by assumption on $s$, we can find a neighbourhood $U$ of $X$ so that $\psi|_U$ is an open embedding. 
	Replacing $U$ by $\inter{g\in G}{g\cdot U}$, we may assume that $U$ is $G$-invariant. Setting
	$$\wt\psi \cl \wt W|_U \to\pi^*W: (x,w) \mapsto (p'(\Phi(t(y))),p'(\Phi(w))),$$ 
	we obtain an equivariant isomorphism satisfying, by construction, $\wt\psi \g t|_U = \Delta_W\g\psi$.\end{proof}

\begin{lemma}\label{lem:tubular-neighbourhood-for-stabilisation-2} Given the situation of Lemma \ref{lem:tubular-neighbourhood-for-stabilisation-1}, suppose additionally that there exists a vector bundle $\rho\cl \wt E\to B$ with a section $\wt s$ and with $E := \wt E|_X$ and $s := \wt s|_X$. Given any isomorphism $\phi \cl \wt E\to r^*E$, let $s'$ be the composite $B\xra{t}\wt E\to r^*E\to E$. Then, we can find a fibrewise affine isomorphism $$\Psi \cl \pi^*E\to \pi^*E$$ 
	which is the identity on $(\pi^*E)|_X$ and with $\Psi(s'(e)) = s(\pi(e))$ for $e \in \pi^*E$.
\end{lemma}

\begin{proof} Simply set $\Psi(w,e) = (w,e + s(\pi(w))-s'(w))$. 
\end{proof}

\begin{remark}\label{} Suppose $\Phi \cl V\to V$ is an affine transformation of an orbibundle bundle $p\cl V\to B$ so that $\norm{\Phi-\ide}$ is bounded in some Finsler norm on $V$. Then, if $\eta$ is a Thom form for $V$, the pullback $\Phi^*\eta$ is still a Thom form since it is closed, has compact vertical support and $p_*\Phi^*\eta = p_*\eta$ by \cite[Proposition~2.1(4)]{ST16}.\end{remark}

\noindent
We also record the following `invariance' under changes of the tamed almost complex structure. 

\begin{lemma}\label{} Suppose $J'$ is another $\omega$-tame almost complex structure. Then we can choose unobstructed auxiliary data $\alpha$ and $\alpha'$ so that the associated global Kuranishi charts are oriented cobordant.
\end{lemma}

\begin{proof} Except for the orientability claim, the proof is verbatim the same as in \cite[\textsection6.2]{HS22}. The claim about orientations follows from the isomorphism \eqref{eq:virtual-orientation}.
\end{proof}
\section{Boundary strata and orientation signs}\label{subsec:boundary-strata} 
This section investigates the boundaries of the moduli spaces of stable maps and the changes in orientation sign, Theorem \ref{thm:boundary-strata-equivalent}. Recall that $(X,\omega)$ is a closed symplectic manifold and $L \sub X$ is an embedded Lagrangian with $n = \dim(L)$. The Lagrangian is equipped with a relative spin structure $(V,\fs)$ with background class $w_\fs$. Proposition~\ref{prop:open-gw-global-chart-existence}\eqref{gkc-existence} implies that any boundary stratum of the thickening is an open subset of the preimage of the boundary stratum of the base space, defined and discussed in \textsection\ref{subsec:base-space}. We thus first determine the possible boundary strata of the base spaces. Recall that $\cB_{k,\ell}$ is the preimage of $\cB\sub \Mbar_{g,h;0,0}^{\,J_0,\,m}(\bC P^N,\bR P^N)$ defined in \textsection\ref{subsec:base-space} under the forgetful map.

\begin{lemma}\label{prop:boundary-base-space-hg} Any boundary stratum of $\cB_{k,\ell}$ is the image of one of the following  maps
	\begin{enumerate}[label=\arabic*),leftmargin=20pt,ref=\arabic*]
		\item\label{i:clutching-1} the clutching map 
		\begin{equation*}\label{} \varphi_{\bR P^N}\cl \Mbar_{g_0,h_0;k_0,\ell_0+\delta_{h_0}}^{\,m_0}(\bR P^N)\times_{\evab_i,\evab_j}\Mbar_{g_1,h_1+1;k_1,\ell_1+\delta_1}^{\,m_1}(\bR P^N)\to \Mbar^{\,m}_{g,h;k,\ell}(\bC P^N,\bR P^N)\end{equation*}
		where $x_0 + x_1 = x$ for $x \in \{g,h,k,\ell\}$.
		\item\label{i:clutching-2}  the self-clutching maps
		$$\psi_{\bR P^N}^{(1)}\cl \Delta_{\bR P^N}\times_{\evab_{a,i}\times \evab_{a,j}}\Mbar_{g,h-1;k,\ell+2\delta_a}^{\,m}(\bC P^N,\bR P^N)\to \Mbar_{g,h;k,\ell}^{\,m}(\bC P^N,\bR P^N),$$\noindent
		where the marked points that are to be identified lie on the same boundary circle, and 
		$$\psi_{\bR P^N}^{(2)}\cl \Delta_{\bR P^N}\times_{\evab_{a,i}\times \evab_{b,j}} \Mbar_{g-1,h+1;k,\ell + \delta_a +\delta_b}^{\,m}(\bC P^N,\bR P^N)\to \Mbar_{g,h;k,\ell}^{\,m}(\bC P^N,\bR P^N),$$\noindent
		where $a \neq b$, i.e., the two marked points at which we clutch lie on different boundary circles;
		\item\label{i:collapse} the `collapse' map
		$$\rho_{\bR P^N}\cl \eva_{k+1}\inv(\bR P^N)\sub \Mbar_{g,h-1;k+1,\ell}^{\,m}(\bC P^N)\to \Mbar_{g,h;k,\ell}^{\,m}(\bR P^N).$$
	\end{enumerate}
	restricted to the respective preimage of $\cB_{k,\ell}$.\par 
	 The map $\varphi_{\bR P^N}$ is an embedding unless $k =\ell = 0$, $\beta_0 = \beta_1$ and $h_0+1 = h_1$, in which case it has degree $2$. Meanwhile, $\psi_{\bR P^N}^{(1)}$ and $\psi_{\bR P^N}^{(2)}$ are local embeddings of degree $2$, and $\rho_{\bR P^N}$ is always an embedding. 
\end{lemma}

\begin{proof} This follows from the corresponding statement for the maps between moduli spaces of closed stable maps using the argument of \cite[Example~1.5]{ST16}, respectively Proposition~\ref{prop:base-space}. That these images enumerate all boundary components follows from the classification of bordered surfaces with one boundary node, \cite[\textsection3.2]{Liu20}.
\end{proof}

As the subscripts suggest, each of these maps is also defined between the respective moduli spaces of stable pseudo-holomorphic maps with boundary on an arbitrary Lagrangian. We denote the domain of $\varphi_L$ by $\Mbar_{g,h;k,\ell}^{\,J,\beta}(X,L;\varphi)$ and similarly for the other maps as well as the base space. If the Lagrangian is clear from context or does not matter, we omit the subscript from the map. These maps are associated to the different types of boundary nodes in the following way
\begin{itemize}
	\item $\varphi$ creates an (H3) node, 
	\item $\psi^{(1)}$ and $\psi^{(2)}$ create an (H1), respectively an (H2) node,
	\item $\rho$ creates a node of type (E).
\end{itemize}

There exist similar clutching maps at marked point whose images are strata of codimension $2$, which therefore are not part of the boundary. Their normal bundles are canonically complex lines bundles, so these strata obtain an orientation from the ambient space (if that is oriented). We first discuss how this `normal orientation' compares to the orientation of the domain of the clutching map.

\begin{proposition}\label{prop:interior-clutching-orientation} Given an unobstructed auxiliary datum $\alpha$ for $\M{0,0}(\beta)$, there exists an open and closed subchart $\cK_{\alpha,\vartheta,\beta}$ of 
	$$\cK_{\alpha,\vartheta}:= (G,\cB(\vartheta)\times_{\cB_{k,\ell}}\cT,\cB(\vartheta)\times_{\cB_{k,\ell}}\cE,\ide\times\obs)$$\noindent
	that defines a global Kuranishi chart for $\Mbar^{\,J,\,\beta}_{g,h;k,\ell}(X,L;\vartheta)$. It is oriented equivalent to the global Kuranishi chart of $\Mbar^{\,J,\,\beta}_{g,h;k,\ell}(X,L;\vartheta)$ given by
	$$(\cK_{\alpha_0,k_0+1,\ell_0}\times_{\evab_i,\evab_1}\cK_{\alpha_1,k_1+1,\ell_1})^{(-1)^{\diamondsuit}},$$\noindent
	where $\alpha_i$ is a $(k_i+1,\ell_i)$-unobstructed auxiliary datum for $\Mbar_{g_i,h_i;0,0}^{\,J,\,\beta_i}(X,L)$ and
	\begin{equation}\label{sign-interior}\diamondsuit =  \begin{cases}
			\omega_\fs(\beta_0) \quad & h_0 = 0,\;h_1 > 0\\
			\omega_\fs(\beta_1)\quad & h_1 = 0,\;h_0 > 0\\
			|\ell_1|h_0\quad & \text{otherwise}
	\end{cases} \end{equation}
	if in the ordering of the boundary circles in $\Mbar^{\,J,\,\beta}_{g,h;k,\ell}(X,L)$, those boundary circles "coming from" $\Mbar_{g_0,h_0;k_0+1,\ell_0}^{\,J,\,\beta_0}(X,L)$ are ordered first and if the order is preserved for both factors.
\end{proposition}

\begin{proof} 
	The equivalence between $\cK_{\alpha,\vartheta,\beta}$ and the fibre product global Kuranishi chart follows from a double thickening construction as in Proposition~\ref{prop:uniqueness-up-to-equivalence}, except that we have to stabilise the evaluation map unless one chart consists of curves with no boundary. The details are given in the case of boundary clutching below. It remains to check the statement about the orientation signs, which by the isomorphism~\eqref{eq:orientation-lines} reduces to a discussion of the orientation of the moduli space of stable curves and of the orientations of the Cauchy--Riemann operators. For the latter, fix a stable map $u \cl (C,\del C)\to (X,L)$ in the image of 
	$$ \vartheta_L\cl \Mbar_{g_0,h_0;k_0+1,\ell_0}^{\,J,\,\beta_0}(X,L)\times_X\Mbar_{g_0,h_1;k_1+1,\ell_1}^{\,J,\,\beta_1}(X,L)\;\longrightarrow\;\Mbar^{\,J,\,\beta}_{g,h;k,\ell}(X,L).$$\noindent
	In the language of \S\ref{sec:orientations} (borrowed from \cite{CZ24}), we have to determine the difference between the intrinsic and limiting orientation of the Cauchy-Riemann operator $D\delbar_J(u)$. Thus, \cite[CROrient~7C(1)]{CZ24} yields the sign in Equation~\ref{sign-interior}. The sign discrepancy between the case of $h_i = 0$ and $h_i > 0$ is due to the fact that we orient the index of $u^*TX|_{C'}$ on a closed irreducible component $C'$ of $C_i$ with the complex orientation if $C_i$ is a closed curve. However, if $C_i$ has boundary, we equip the index of $u^*TX|_{C'}$ with the complex orientation if and only if $w_\fs(u_*[C']) = 0$.  The orientation sign of the clutching map on the level of the moduli space of stable curves follows from the same computation, using that the spin structure of the boundary of a Riemann surface has vanishing background class. In the case where $h_0,h_1 > 0$, the boundary marked points on curves in $\Mbar_{g_1,h_1;k_1+1,\ell_1}$ have to be permuted past the index bundle of $\Mbar_{g_0,h_0;k_0+1,\ell_0}$, whence we obtain the last sign in~\eqref{sign-interior}.
\end{proof}

\noindent
We now turn to determining the difference between the fibre product and the boundary orientation of a (normalised) boundary stratum. For simplicity, we make the following assumptions on ordering of the boundary components and boundary marked points.

\begin{assum}\label{as:orientation} We first assume that all boundary marked point are ordered anti-clockwise on each boundary circle according to their labelling.
	\begin{enumerate}[ a),leftmargin=*]
		\item\label{as:clutching}  For $[\iota,C] = \varphi([\iota_0,C_0],[\iota_1,C_1])$ we have that
		\begin{enumerate}[ 1)]
			\item\label{ordering-clutching} the ordering $\fii$ of the boundary components of $C$ is given by $\fii = (\fii_0,\fii_1)$, 
			\item\label{ev-1-clutching} $\evab_i$ is the evaluation at the $i^{th}$ marked point of the last component in $\fii_0$
			\item\label{ev-2-clutching} $\evab_j$ is the evaluation at the first marked point on the first component in $\fii_1$,
			\item\label{marked-points-clutching}  the ordering of the boundary marked points on the clutched component is given as by
			\begin{equation}\label{fig:clutching-1}\includegraphics[scale=0.28]{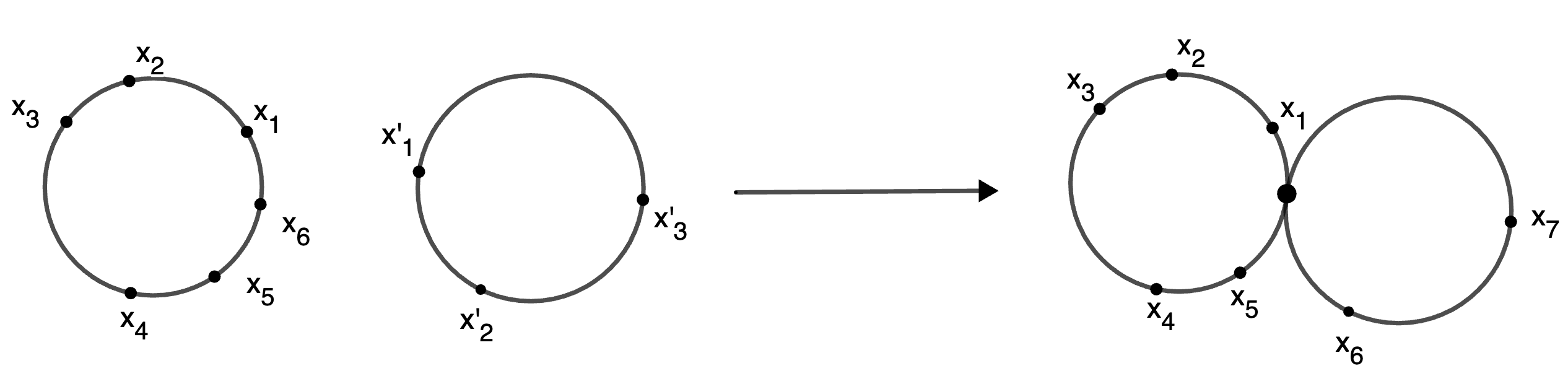}
			\end{equation}
			for $\ell_{0,h_0+1} = 5$, $\ell_{1,1} = 2$ and $i = 6$, while the ordering of all other boundary marked points remains the same.
		\end{enumerate} 
	\item\label{as:self-clutching} For $\psi^{(1)}$ we require
	\begin{enumerate}[ 1)]
		\item if $[\iota,C]$ is a smooth surface in $\Mbar_{g,h;k,\ell}^{\,m}(\bC P^N,\bR P^N)$, then the two boundary circle that are clutched in the image of $\psi^{(1)}$ are adjacent in the ordering of the boundary components of $C$,
		\item $i < j$
		\item we have the following order of the boundary marked points on the nodal boundary component as shown for $i = 2$ and $j = 6$:
		\begin{equation}\label{fig:clutching-2}\includegraphics[scale=0.25]{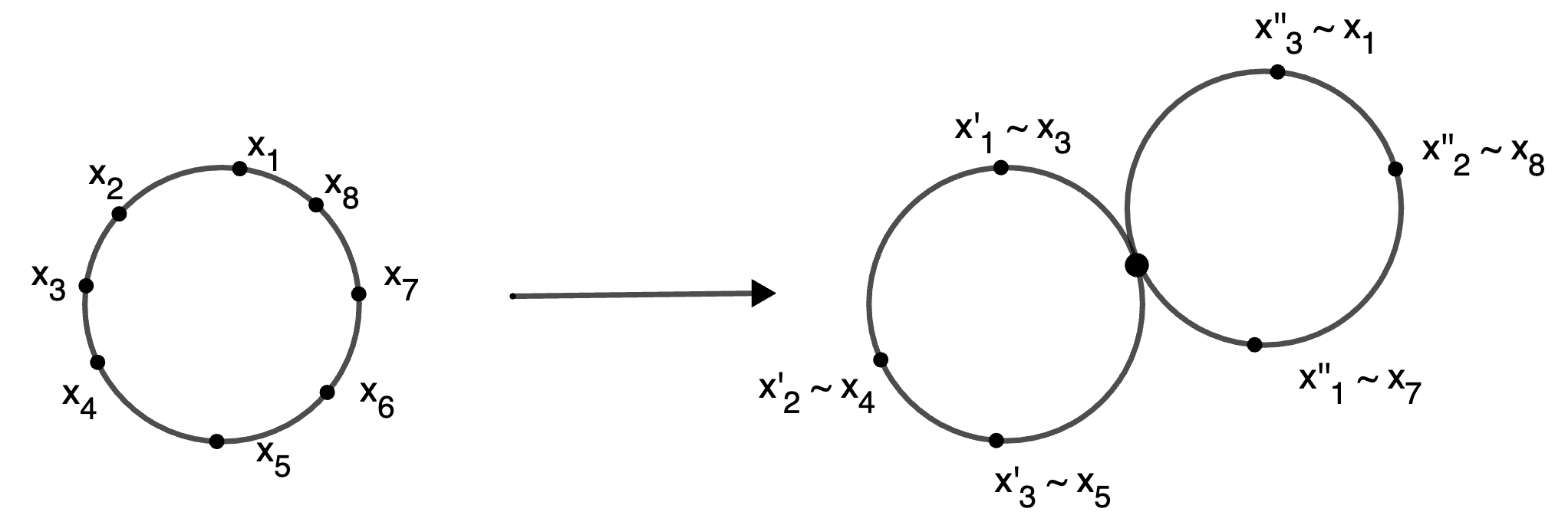}
		\end{equation}
		while the ordering of all other boundary marked points remains the same; for $\psi^{(2)}$ we assume that 
\begin{enumerate}[ 1)]
	\item\label{ordering-self-clutching-1} $b =  a+1$,
	\item\label{ev-2-self-clutching-1} $j = 1$, i.e., $\evab_{b,j}$ is the evaluation at the first marked point on $b^{th}$ boundary component,
	\item\label{marked-points-self-clutching-1} the boundary marked points on the image of curves under $\psi^{(2)}$ are ordered as in the case of $\varphi$;
\end{enumerate} 
\end{enumerate}
\item\label{as:collapse} $\rho$ preserves the ordering of all boundary marked points and the boundary component that is collapsed comes last in the ordering of boundary components.
	\end{enumerate}
\end{assum} 

\begin{remark}\label{rem:maslov-index-if-orientable} We observe that if $L$ is orientable, then $\mu_L$ takes values in $2\bZ$. Indeed, if $v \cl (C,\del C)\to (X,L)$ is a map from a surface with nonempty boundary, then 
	$$\mu_L([v]) = w_1(v|_{\del C}^*TL) = (v|_{\del C})^*w_1(TL) \equiv 0 \quad (\text{mod }2)$$\noindent
	while for $\beta \in \im(H_2(X;\bZ)\to H_2(X,L;\bZ))$ we have $\mu_L(\beta) = 2\lspan{c_1(TX),\beta}$. 
\end{remark}

\noindent
If $L$ is not relatively spin, the following statement is still true when removing any mention of orientations.

\begin{theorem}\label{thm:boundary-strata-equivalent} Given $\vartheta \in \{\varphi,\psi^{(2)},\psi^{(1)},\rho\}$ and an unobstructed auxiliary datum $\alpha$ for $\M{0,0}(\beta)$, there exists an open and closed subchart $\cK_{\alpha,\vartheta,\beta}$ of 
		$$\cK_{\alpha,\vartheta}:= (G,\cB(\vartheta)\times_{\cB_{k,\ell}}\cT,\cB(\vartheta)\times_{\cB_{k,\ell}}\cE,\ide\times\obs)$$\noindent
		that defines a global Kuranishi chart for $\Mbar^{\,J,\,\beta}_{g,h;k,\ell}(X,L;\vartheta)$.
		Suppose $\cK_{\alpha,\vartheta,\beta}$ is equipped with the boundary orientation from $\cK_\alpha$ and the same conditions on the ordering of the boundary component and marked points as in Assumption \ref{as:orientation} hold. Then, $\cK_{\alpha,\vartheta,\beta}$ is oriented rel--$C^\infty$ equivalent to the global Kuranishi chart of 
		\begin{enumerate}[label=\arabic*),leftmargin=20pt,ref=\arabic*]
			\setlength\itemsep{3pt}
			\item\label{or-gluing-hg-1} $\Mbar^{\,J,\,\beta}_{g,h;k,\ell}(X,L;\varphi)$ given by $$(\cK_{\alpha_0,k_0,\ell_0+\delta_{h_0}}\times_{\evab_i,\evab_1}\cK_{\alpha_1,k_1,\ell_1+\delta_{1}})^{(-1)^{\dagger}},$$\noindent
			where 
			\begin{equation}\label{sign-h3}\dagger = |\ell_1|(h_0-1) + |\ell_0|-\ell_{0,h_0} + \ell_{1,1}(\ell_{0,h_0}+ 1-i) + i-1 + n(h +h_1(h_0 + |\ell_0|+1))\end{equation}
			reducing to $n+\ell_1(\ell_0+1-i) +i-1\;(\normalfont\text{mod }2)$ in the case of $(g,h) = (0,1)$. 
			\item\label{or-gluing-hg-2}  $\Mbar^{\,J,\,\beta}_{g,h;k,\ell}(X,L;\psi^{(1)})$ given by $$(\Delta_L\times_{\evab_{a,i}\times\evab_{a,j}}\cK_{\alpha',k_1,\ell+2\delta_a})^{(-1)^{\ddagger_1}},$$\noindent
			where 
			\begin{equation}\label{sign-h1}\ddagger_1 = (\ell_a-j)(j-i-1) + h-a-1 +n(h-a)+\frac{n(n-1)}{2} \end{equation}
			respectively,  $\Mbar^{\,J,\,\beta}_{g,h;k,\ell}(X,L;\psi^{(2)})$ given by $$(\Delta_L\times_{\evab_{a,i}\times\evab_{b,1}}\cK_{\alpha',k_1,\ell+\delta_a +\delta_b})^{(-1)^{\ddagger_2}},$$\noindent
			where 
			\begin{equation}\label{sign-h2}\ddagger_2 =   (h-a)n +|\ell| + a +( \ell_{b}+1)(\ell_{a}+1-i).\end{equation}
			\item\label{or-gluing-hg-3} $\Mbar_{g,h;k,\ell}^{\,J,\,\beta}(X,L;\rho)$ given by the fibre product 
			$$(L \times_{\evai_{k+1}}\cK_{\wt\alpha,k+1,\ell})^{(-1)^{\heartsuit}},$$\noindent 
			with
			 \begin{equation}\label{sign-e}\heartsuit =  (h-1-i)n-1+i +|\ell|+\delta_{1,h}\,\omega_\fs(\beta)\end{equation}
			where the fibre product is taken over $X$, the collapsed boundary circle is in $i^{th}$ position, and $\wt\alpha$ is $(k+1,\ell)$-unobstructed. If $(g,h) = (0,1)$, this reduces to $n + w_\fs(\beta)$.
		\end{enumerate}
	\end{theorem}

	\noindent
	We will first prove the assertion about equivalence in Theorem \ref{thm:boundary-strata-equivalent} in the case of \eqref{or-gluing-hg-1}. The other cases are similar and left to the interested reader. 
	As the proof involves a lot of notation, let us describe the general strategy, which is similar to the one in Proposition~\ref{prop:uniqueness-up-to-equivalence}. Explicitly, we construct a doubly-thickened global Kuranishi chart that ``interpolates" between $\cK_{\alpha,\vartheta,\beta}$ and the respective fibre product of global Kuranishi charts. Due to the isomorphism \eqref{eq:virtual-orientation}, the determination of the orientation sign becomes an application of \cite{CZ24}. See \cite[Proposition~2.8 and Proposition~2.11]{ST23} for the argument in the regular case when $(g,h) = (0,1)$.

\begin{proof}[Proof of equivalence in Theorem~\ref{thm:boundary-strata-equivalent}] We consider the clutching map 
		$$ \varphi:=\varphi_{L,i}\cl \Mbar_{g_0,h_0;k_0,\ell_0+\delta_{h_0+1}}^{\,J,\,\beta_0}(X,L)\times_{\eva_i,\eva_1}\Mbar_{g_1,h_1+1;k_1,\ell_1+\delta_1}^{\,J,\,\beta_1}(X,L)\to \Mbar^{\,J,\,\beta}_{g,h;k,\ell}(X,L).$$\noindent
		Observe that 
		$$\cB(\varphi) \times_\cB\cT\cong \djun{\substack{\beta'_0 +\beta'_1 = \beta\\ m(\beta'_i) = m_i}}\cT\times_{\Mbar_{g,h;k,\ell}^{*,\, \wt J_E,\,(m,\beta)}(E,F)}\lbr{\Mbar_{g_0,h_0;k_0+1,\ell_0}^{*,\, \wt J_E,\,(m_0,\beta'_0)}(E,F)\times_{\evab_i,\evab_1}\Mbar_{g_1,h_1;k_1+1,\ell_1}^{*,\, \wt J_E,\,(m_1,\beta'_1)}(E,F)},$$\noindent
		where $m(\beta'_i)$ is the degree of $\fL_u^{\otimes p}$ for a stable map $u$ in $\M{k_i+1,\ell_i}^{J,\,\beta_i}(X,L)$\footnote{To be precise, it is the degree of $\fL_{u'}^{\otimes p}$, where $u'$ is the stabilised map given forgetting all marked points on the domain of $u$ and contracting unstable components.}. Set 
		$$\wh\cT := \cT\times_{\Mbar_{g,h;k,\ell}^{*,\, \wt J_E,\,(m,\beta)}(E,F)}\Mbar_{g_0,h_0;k_0+1,\ell_0}^{*,\, \wt J_E,\,(m_0,\beta_0)}(E,F)\times_{\evab_i,\evab_1}\Mbar_{g_1,h_1;k_1+1,\ell_1}^{*,\, \wt J_E,\,(m_1,\beta_1)}(E,F)$$\noindent
		and let 
		\begin{equation}\label{eq:pullback-chart-hg} \cK_{\varphi,\beta}  := (G,\wh\cT,\wh\cE:=\cE\times_\cT\cT_0,\wh \obs := \obs|_{\cE_0})\end{equation}
		be the associated global Kuranishi chart.\par 
		Write $\cK_{\alpha_0,k_i,\ell_i+1} = (G_i,\cT_i,\cE_i,\obs_i)$, with $\cT_i$ mapping homeomorphically to an open subset of $$\Mbar^*(F_0):=\Mbar_{g_0,h_0+1;k_0,\ell_0+\delta_{h_0+1}}^{*,\wt J_{E_0},\,(m'_0,\beta_0)}(E_0,F_0)$$\noindent
		and 
		$$\Mbar^*(F_1):=\Mbar_{g_1,h_1;k_1,\ell_1+\delta_1}^{*,\wt J_{E_1},\,(m'_1,\beta_1)}(E_1,F_1)$$\noindent
		so that $\cT_0\times_L \cT_1$ admits a codimension-$0$ embedding to $\Mbar^*(F_0)\times_L\Mbar^*(F_1)$ and a submersion to $\cB_0\times\cB_1$, where $\cB_i$ is the base space of $\cK_i$. Let 
		$$\wt \cB \sub \Mbar_{g_0,h_0+1;k_0,\ell_0+\delta_{h_0+1}}^{\,(m'_0,m_0)}(\bR P^{N_0}\times \bR P^N)\times_{\bR P^N}\Mbar_{g_1,h_1;k_1,\ell_1+\delta_1}^{\,(m'_1,m_1)}(\bR P^{N_1}\times \bR P^N)$$\noindent
		be the preimage of $\cB_0\times\cB_1 \times \cB(\varphi)$ under the canonical map to the product of moduli spaces. By the same argument as in the proof of Lemma~\ref{prop:uniqueness-up-to-equivalence}, we deduce that $\wt \cB$ is a smooth orientable manifold with corners. Moreover, the forgetful maps 
		$q\cl \wt \cB \to \cB_0\times\cB_1$ and $\wh q\cl \wt\cB \to \wh\cB$
		are submersions and invariant under the action by $G$, respectively $G_0\times G_1$. Denote by
		$$\wt E_i := E_i \times_X  E\qquad \qquad \wt F_i := F_i \times_L  F$$\noindent
		the induced complex, respectively real, vector bundles over $\bC P^{N_i}\times \bC P^N \times X$, respectively $\bC P^{N_i}\times \bC P^N \times X$. Let $\wt J_{\wt E_i}$ be the almost complex structure on the total space of $\wt E_i$ given by
		 
		\begin{align*}
			\wt J_{(y_i,y,x,e_i,e)}^{\wt E_i} := 
			\begin{pmatrix}  
				J_{y_i}^{\bC P^{N_i}} & 0&0 &0\\
				0 & J_y^{\bC P^N} & 0& 0\\
				J_x & 2 J_x\lspan{e_i} & 2 J_x\lspan{e}& 0\\
				0 & 0 &0 & J^{\wt E_i}_{(y_i,y,x)}
			\end{pmatrix},
		\end{align*}
		using the splitting $T\wt E_i \cong \bC P^{N_i}\times\bC P^N\times X\times \wt E_i$. Abbreviate 
		$$\Mbar^*(\wt E_i) := \Mbar_{g_i,h_i+\delta_{i,0},k_0,\ell_i +\delta'_i}^{*,\wt J_{\wt E_i},\,(m'_i,m_i,\beta_i)}(\wt E_i')$$\noindent
		where $\delta'_0 = \delta_{h_0+1}$ and $\delta'_1 = \delta_1$, and let
		$$\wt \cT\cong \wt\cT'\sub \Mbar^*(\wt E_0)\times_{F}\Mbar^*(\wt E_1)$$\noindent
		be subset mapping to $\wt\cB$ contained in the regular locus of the forgetful map to $\wt\cB$. The map $\wt\cT\to \wt\cT'$ is given by forgetting the parameter $\gamma_i$ that measures whether the framing $\iota_i$ pulls back $\cO(1)$ to the reference line bundle $\fL_{u_i}$. We omit it from the notation from now on and identify $\wt \cT$ with $\wt\cT'$.\par 
		Given $\wt u_i \in \Mbar^*(\wt E_i)$, we write $(\iota_i,\wh\iota_i,u_i)$ for its composition with the map to $\bC P^{N_i}\times \bC P^N\times X$ and denote by $\eta_i$, and $\wh\eta_i$, the associated holomorphic section of $(\iota_i,u_i)^*E_i$ and $(\wh\iota_i,u_i)^*E$. By the definition of $\wt J_{\wt E_i}$, 
		\begin{equation}\label{}\delbar_J u_i + \lspan{\eta_i}\g d\iota_i + \lspan{\wh\eta_i}\g d\wh\iota_i = 0\end{equation}
		on the normalisation of $C_i$. By abuse of notation, we write $\cE_i$ for the bundle $\cE_i \to \wt \cT$ with fibre 
		$$(\cE_i)_{(\wt u_0,\wt u_1)} = H^0(C_i,(\iota_i,u_i)^*(E_i,F_i))\oplus H^1(C_i,(\bC,\bR))\oplus \fp\fo(N_i+1),$$\noindent
		while $\cE\to\wt\cT$ is defined by $\cE_{(\wt u_0,\wh u_1)} = V_{(\wt u_0,\wh u_1)}\oplus \fp\fo(N+1)$, where 
		
		$$\setlength\delimitershortfall{-0.5pt} V_{(\wt u_0,\wh u_1)} := \left\{(\wh\eta_0,\wh\eta_1)\in \bigoplus_{i = 0,1}H^0(C_i,(\wt\iota_i,u_i)^*(E,F))\mid \wh\eta_0(x^b_{\ell_0+1}) = \wt \eta_1(x^b_{1})\right\}.$$\noindent
		Set $\wt \cE := \cE_0\oplus \cE_1\oplus \cE$ and let $\wt\obs\cl \wt\cT\to \wt\cE$ be the section defined analogously to the obstruction section in Construction \ref{construction-other-data}. The covering group is $\wt G := G_0\times G_1 \times G$ acting on $\wt\cT$ and $\wt \cE$ via its action on $E_0, E_1$ and $E$. Possibly after shrinking $\wt \cT$, the arguments of the proof of Proposition~\ref{prop:open-gw-global-chart-existence} show that 
		$$\wt \cK= (\wt G,\wt\cT/\wt\cB,\wt\cE,\wt\obs)$$\noindent
		is a global Kuranishi chart with corners for $\Mbar_{g,h;k,\ell}^{J,\beta}(X,L;\varphi)$. The arguments of Proposition~\ref{prop:uniqueness-up-to-equivalence} show that $\wt\cK$ is orientable and equivalent to both $\cK_0\times_L \cK_1$ and $\cK_{\varphi,\beta}$, which was defined in \eqref{eq:pullback-chart-hg}.
	\end{proof}	

\noindent
We now discuss the difference in orientations between the respective fibre product and the boundary stratum. 
		
	\begin{proof}[Proof of Theorem~\ref{thm:boundary-strata-equivalent}\eqref{or-gluing-hg-1}]
		We may assume that $\cN_i := \Mbar_{g_i,h_i+\delta_{i,0};k_i,\ell_i+1}$ is nonempty for $i \in \{0,1\}$, stabilising by interior marked points if necessary. Set $\cN:= \Mbar_{g,h;k,\ell}$. The canonical isomorphism \eqref{eq:virtual-orientation} on the zero locus is also defined for any other point of the respective thickening, so we may work with the index bundles directly.\par The fibre product orientation on $\Mbar_{g,h;k,\ell}^{\,J,\,\beta}(X,L;\varphi)$ at a point $((u_0,C_0,x^0_*),(u_1,C_1,x^1_*))$ corresponding to $(u,C,x_*)$ is given by 
		
  \begin{align*}
  	\fo_{01} = (-1)^{(n + \ind(D_1))\dim(\cN_0)} \fo_{D_0}\wedge \fo_L\dul \wedge \fo_{D_1} \wedge \fo_{\cN_0}\wedge \fo_{\cN_1}
  \end{align*}
	while the boundary orientation is given by (the pullback of)
	 \begin{align*}
		\fo_{\del} = (-1)^{\ind(D)} \fo_{D}\wedge \fo_{\del\cN}
	\end{align*}
	Here, $D_i := D\delbar_J(u_i)$ and $D := D\delbar_J(u)$. In the language of \cite{CZ24}, $\fo_{D_0}\wedge \fo_{D_1}\wedge \fo_L\dul$ is the \emph{intrinsic orientation} of $\det(D)$, see also \textsection\ref{subsec:orientation-construction},  while $\fo_{D}$ is the \emph{limiting orientation}, defined loc. cit. and induced by the orientation of the Cauchy-Riemann operator on the smoothed surface. Thus, by \cite[CROrient~7H3(a)]{CZ24} and Assumption \ref{as:orientation},
	\begin{equation}\label{eq:vertical-orientations-h3} \fo_D = (-1)^{n h_1+n}\fo_{D_0} \wedge \fo_{D_1}\wedge \fo_L\dul = \fo_{D_0}\wedge \fo_L\dul \wedge \fo_{D_1} .\end{equation}
	We observe that the first sign differs from the sign given by \cite[CROrient~7H3(a)]{CZ24} by $(-1)^n$. The reason is that for the intrinsic orientation (in the regular case), we use the short exact sequence 
	$$0 \to \ker(D)\to \ker(D_0)\oplus \ker(D_1)\xra{(\xi_0,\xi_1)\mapsto \xi_1(x_{1,1})-\xi_0(x_{\ell_0,})} T_{u(\text{nd})}L \to 0$$\noindent
	by \cite[Equation~(7.39)]{CZ24}, while for the fibre product orientation we use the orientation of $\ker(D)$ obtained from  
	$$0 \to \ker(D)\to \ker(D_0)\oplus \ker(D_1)\xra{(\xi_0,\xi_1)\mapsto \xi_0(x_{\ell_0,})-\xi_1(x_{1,1})} T_{u(\text{nd})}L \to 0$$\noindent
	by \cite[Convention~7.2(b)]{Joy}. Thus, we obtain an orientation difference of $(-1)^{n}$.
	In summary, a first part of the sign is given by 
	\begin{align}\label{h3-index}\notag\delta_D &= \ind(D) + (n+ \ind(D_1))\dim(\cN_0)\\
		\notag&\equiv nh + \mu_L(\beta)  + (nh_1 + \mu_L(\beta_1))(h_0 + |\ell_0|+1) \quad (\text{mod }2)\\
		&\equiv n(h +h_1(h_0 + |\ell_0|+1))\quad (\text{mod }2)
	\end{align}
	where the last equality follows from Remark \ref{rem:maslov-index-if-orientable}. In particular, $\delta_D\equiv n \; (\text{mod }2)$ if $(g,h) = (0,1)$.\par
	It remains to determine the orientation sign of the map $\varphi\cl \cN_0 \times\cN_1 \to \del\cN$, which we reduce to the computation of the orientation sign in the case of disks. For this recall from the proof of Proposition~\ref{prop:interior-clutching-orientation} that the clutching at an interior marked point is orientation-preserving (using the canonical orientation of the normal bundle of the divisor stratum). Thus, fix a curve $C_0 \in \cN_0$ with a single interior node, where the boundary component that is involved in the clutching is bubbled off on a separate disk, i.e., $C_0$ is the clutching of the smooth curves $C_0'$ and $C_0'' = \bD$ at an interior node. Write $\cN_0'$ for the moduli space of stable curves containing $C_0'$ and similarly for $C_0''$. Then, there exists a canonical isomorphism 
	\begin{equation}\label{iso-1} 
		\orl(T_{C_0}\cN_0)\;\cong\; \orl(\bC)\orl(T_{C'_0}\cN_0')\orl(T_{C_0''}\cN_0'')
	\end{equation}
 	that has orientation sign $(-1)^{(\ell_{0,h_0}+1)(h_0-1)}$ since we had to permute the marked points lying on the boundary of $C_0''$ past $\orl(\delbar_{TC_0'}\dul)$ which has degree $h_0-1$ mod $2$. On the other hand, letting $C_1 \sub \cN_1$ be a curve with a unique interior node splitting it into a disk $C''_1$ (with boundary the first boundary component) and a curve $C'_1$, the isomorphism 
	\begin{equation}\label{iso-2} 
	\orl(T_{C_1}\cN_1)\;\cong\; \orl(\bC)\orl(T_{C''_1}\cN_1'')\orl(T_{C_1'}\cN_1')
	\end{equation}
	has orientation sign $(-1)^{|\ell_1|-\ell_{1,1}}$. Hence, we have a commutative diagram, omitting $\orl(\bC)$, of orientation lines
	\begin{center}\begin{tikzcd}
			\orl(\bR)\orl(T_{C_0}\cN_0)\orl(T_{C_1}\cN_1)  \arrow[r,""] \arrow[d,"{(1)}"]&\orl(T_C\cN) \\
			\orl(\cN'_0)\orl(\bR)\orl(\cN''_0)\orl(\cN''_1)\orl(\cN_1) \arrow[r,"{(2)}"] &\orl(\cN'_0)\orl(\cN'')\orl(\cN_1)\arrow[u,"{(3)}"]\\
	  \end{tikzcd} \end{center}
	where $(i)$ has orientation sign $(-1)^{\delta_i}$ with 
	\begin{align*}\delta_1 \;&=\; (\ell_{0,h_0}+1)(h_0-1)+|\ell_1|-\ell_{1,1}+ (h_0-1)  +|\ell_0|-\ell_{0,h_0}\\
		&=\; \ell_{0,h_0} +|\ell_1|-\ell_{1,1}+|\ell_0|
	\end{align*}
	given by the signs of \eqref{iso-1} and \eqref{iso-2} and the sign from permuting $\orl(\bR)$ past $\orl(\cN_0')$, while
	\[\delta_2 = \ell_{1,1}(\ell_{0,h_0}+1-i)+i-1\]
	due to the sign computation of \cite[Proposition~2.8]{ST16}. There is a shift because they label the first marked point by $0$.
	Finally,
	\[\delta_3 = (\ell_{0,h_0}+\ell_{1,1})(h_0-1) + (|\ell_1|-\ell_{1,1})h_0\]
	due to Proposition~\ref{prop:interior-clutching-orientation} in the case where the target manifold is a point. Then, 
	\begin{equation}\label{eq:or-sign-h3}
		\orl(\varphi) = (-1)^{|\ell_1|(h_0-1) + |\ell_0|-\ell_{0,h_0} + \ell_{1,1}(\ell_{0,h_0}+ 1-i) + i-1} 
		\end{equation}
	since
	\begin{align*}\label{h3-curves}
		\delta \;&=\; \delta_1 +\delta_2 + \delta_3\\
				 &\equiv\; \ell_{0,h_0}h_0 +|\ell_1|-\ell_{1,1} +|\ell_0|+\ell_{1,1}(\ell_{0,h_0}+1-i)+i-1 +(\ell_{0,h_0}+\ell_{1,1})(h_0-1) + (|\ell_1|-\ell_{1,1})h_0\\
				 &\equiv\;(|\ell_1|-\ell_{1,1})(h_0-1) + |\ell_0|+\ell_{1,1}(h_0-1)\\
				 &\equiv\;|\ell_1|(h_0-1) + |\ell_0|-\ell_{0,h_0} + \ell_{1,1}(\ell_{0,h_0}+ 1-i) + i-1  \mod{2}.
	\end{align*}
If $(g,h) = (0,1)$, this reduces to $\fo(\varphi) = (-1)^{\ell_1(\ell_0+1-i)+i-1}$, recovering the sign in \cite[Proposition~2.8]{ST16} for the case $n = 0$. Note that they have $i$ instead of $i-1$, due to the reason that they index the marked points by $\{0,\dots,\ell_i\}$. The final sign is now $(-1)^{\delta_D +\delta}$, where $\delta_D$ was computed in~\eqref{h3-index}.
\end{proof}	

\noindent
We now turn to the proof of the orientation sign in the case of the nodes obtained from self-clutching.
	
\begin{proof}[Proof of Theorem~\ref{thm:boundary-strata-equivalent}\eqref{or-gluing-hg-2}]  By adding interior marked points, we may assume that $(g,h,k,\ell)$ lies in the stable range. Set $\cN:= \Mbar_{g,h,k,\ell}$ and let
	$$\cN_1:= \Mbar_{g,h-1;k,\ell+2\delta_a}\qquad \quad\text{and}\qquad\quad\cN_2:= \Mbar_{g-1,h+1;k,\ell+\delta_a+\delta_b}.$$\noindent
	The fibre product orientation on $\Mbar_{g,h;k,\ell}(X,L,\psi^{(1)})$ at a point $(u',C',x'_*)$ mapping to $(u,C,x_*)\in \Mbar_{g,h;k,\ell}(X,L)$ is given by 
	$$\fo_1 := \fo_L\dul \wedge \fo_{D'} \wedge \fo_{\cN_1},$$\noindent
	while the boundary orientation is given by 
	$$\fo_{\psi^{(1)}} = (-1)^{\ind(D)}\fo_D\wedge \fo_{\del\cN}.$$\noindent
 	By Proposition \ref{prop:orientation-for-H1-node} and the difference between the intrinsic orientation and the fibre product orientation explained below Equation~\eqref{eq:vertical-orientations-h3}, we have 
 	\begin{equation}\label{} 
 		\fo_L\dul\wedge\fo_{D'} = (-1)^{n(h-a+1)+\ind(D')n}\fo_D = (-1)^{n\, a}\fo_D.
 	\end{equation}
 	To compare the orientation sign $\psi^{(1)}\cl \cN' \to \del\cN$, we employ a similar argument as in the previous proof, reducing to the case where the gluing happens on a disk. For this, let $C\in \cN'$ be a curve with a two interior nodes, splitting into $C'\in \cN_1$, carrying the first $a-1$ boundary components of $C$, $C'' = \bD$, defining an element of $\cN_2$, where the $a$th boundary component of $C$ is the boundary of $C''$, and $C''\in \cN_3$ carrying the last boundary components of $C$. The canonical isomorphism 
 	\begin{equation}\label{iso-h1}\orl(\cN')\cong \orl(\bC)\orl(\cN_1)\orl(\cN_2)\orl(\cN_3) \end{equation}
 	has orientation sign $(-1)^{\ell_a(a-1)+ (\ell_{a+1}+\dots+\ell_{h-1})a}$ by Proposition~\ref{prop:interior-clutching-orientation}. Then we have as before a commutative square
 	\begin{center}\begin{tikzcd}
 			\orl(\bR)\orl(\cN')  \arrow[r,""] \arrow[d,"{(1)}"]&\orl(\cN) \\
 			\orl(\cN_1)\orl(\bR)\orl(\cN_2)\orl(\cN_3) \arrow[r,"{(2)}"] &\orl(\cN_1)\orl(\cN_{23}) \arrow[u,"{(3)}"]\\
 	\end{tikzcd} \end{center}
 where $\cN_{23}$ is the target of the clutching map on $\cN_2\times\cN_3$ and $(i)$ has orientation sign $(-1)^{\delta_i}$ with 
 \begin{align*}\delta_1 \;&=\;\ell_a(a-1)+ (\ell_{a+1}+\dots+\ell_{h-1})a+ \dim(\cN_1')
 \end{align*}
 given by the orientation sign of \eqref{iso-h1} and the sign from permuting $\orl(\bR)$ past $\orl(\cN_1')$, while
 \[\delta_2 = (\ell_a-j)(j-i-1)\]
 by Proposition~\ref{prop:orientation-for-H1-node} (which gives $0$ since here the rank of the vector bundle is $1$) and the reordering of the marked points.
 Finally,
 \[\delta_3 = \ell_a(a-1)+ (\ell_{a+1}+\dots + \ell_{h-1})(a+1) \]
 due to Proposition~\ref{prop:interior-clutching-orientation} again, noting that in $\cN$ we have one boundary component more. Thus, the orientation sign of $\psi^{(1)}\cl \cN_1 \to \del\cN$ is 
 \begin{equation}\label{} 
 	\orl(\psi^{(1)}) = (-1)^{(\ell_a-j)(j-i-1) + h-a-1 }.
 \end{equation}
 	
 	\medskip
 
 	\noindent The case of an (H2) boundary node, reduces essentially to the case of \eqref{or-gluing-hg-3} by the definition of the intrinsic orientation, cf. \S\ref{subsec:orientation-construction}. Write $D'$ for the Cauchy--Riemann operator of an element $u'$of $\Mbar_{g,h;k,\ell}(X,L,\psi^{(1)})$ and $D$ for the Cauchy--Riemann operator of its image in $\Mbar_{g,h;k,\ell}(X,L)$. Then, fibre product orientation at this point is given by 
 		$$\fo_2 := \fo_L\dul \wedge \fo_{D'} \wedge \fo_{\cN_1},$$\noindent
 		while the boundary orientation is given by 
 		$$\fo_{\psi^{(2)}} = (-1)^{\ind(D)}\fo_D\wedge \fo_{\del\cN}.$$ \noindent
 		Invoking \cite[CROrient~7H3(a)]{CZ24} and the same reasoning as for Equation~\eqref{eq:vertical-orientations-h3}, we obtain for $b = a+1$ that 
 		$$\fo_L\dul \wedge \fo_{D'} = (-1)^{\ind(D')n}\fo_{D'}\wedge \fo_L\dul = (-1)^{\ind(D')n+ (h-a)n+n}\fo_D = (-1)^{n a}\fo_D.$$\noindent
 		To compute the sign of $\psi^{(2)}\cl \cN'' \to \del\cN$, we consider a curve $C\in \cN''$ with three interior nodes splitting $C$ into $C_1$ carrying the first $a-1$ boundary components, $C_2 = \bD$ carrying the $a$th boundary component, $C_3$ carrying the boundary component labelled by $b = a+1$ and $C_4$ carrying the remaining boundary components. Write $\cN_i$ for the moduli space containing $C_i$. The isomorphism
 		\begin{equation}\label{} 
 			\orl(\bC)^{\otimes 3}\orl(\cN_1)\orl(\cN_2)\orl(\cN_3)\orl(\cN_4)\;\cong\;\orl(\cN'')
 		\end{equation}
 		has sign $(-1)^{(\ell_a+1)(a-1) + (\ell_b+1)(b-1) + (\ell_{b+1}+\dots +\ell_h)b}$ by Proposition~\ref{prop:interior-clutching-orientation}. As before, we have a commutative square
 		\begin{center}\begin{tikzcd}
 				\orl(\bR)\orl(\cN'')  \arrow[r,""] \arrow[d,"{(1)}"]&\orl(\cN) \\
 				\orl(\cN_1)\orl(\bR)\orl(\cN_2)\orl(\cN_3)\orl(\cN_4) \arrow[r,"{(2)}"] &\orl(\cN_1)\orl(\cN_{23})\orl(\cN_4) \arrow[u,"{(3)}"]\\
 		\end{tikzcd} \end{center}
 		where $(i)$ has orientation sign $(-1)^{\delta_i}$ with 
 		\begin{align*}\delta_1 \;&=\;(\ell_a+1)(a-1) + (\ell_b+1)(b-1) + (\ell_{b+1}+\dots +\ell_{h+1})b+\dim(\cN_1)
 		\end{align*}
 		given by the orientation sign of \eqref{iso-h1} and the sign from permuting $\orl(\bR)$ past $\orl(\cN_1)$, while
 		\[\delta_2 = \ell_{b}(\ell_{a}+1-i)+i-1\]
 		by \cite[Proposition~2.8]{ST16}, which we may apply because $C_2$ and $C_3$ are disks. As above, the formula is slightly different because we start labelling by $1$ instead of $0$ as they do. The last exponent is
 		 \[\delta_3 = (\ell_a+\ell_b)(a-1)+ (\ell_{b+1}+\dots + \ell_{h+1})a \]
 		due to Proposition~\ref{prop:interior-clutching-orientation} again, noting that in $\cN$ we have one boundary component less. Thus, the orientation sign of $\psi^{(2)}\cl \cN'' \to \del\cN$ is 
 		\begin{equation}\label{} 
 			\orl(\psi^{(2)}) = (-1)^{|\ell| + a +( \ell_{b}+1)(\ell_{a}+1-i)}.
 		\end{equation}
 		This completes the proof.
\end{proof}

\begin{proof}[Proof of Theorem~\ref{thm:boundary-strata-equivalent}\eqref{or-gluing-hg-3}] Let $(u',C')$ be an element of $L\times_X \Mbar_{g,h-1;k+1,\ell'}^{\,J,\,\beta}(X,L)$ with image $(u,C)$ in $\Mbar^{\,J,\,\beta}_{g,h;k,\ell}(X,L)$ and let $D'$ and $D$ be their respective Cauchy--Riemann operator. Let $\cN' =\Mbar_{g,h-1;k+1,\ell'}$ and $\cN= \Mbar_{g,h;k,\ell}$, where $\ell = (\ell',0)$. We have to compare 
		$$\fo_\rho := \fo_L\dul \wedge \fo_{D'} \wedge \fo_{\cN'},$$\noindent
	with the boundary orientation  given by 
	$$\fo_\del = (-1)^{\ind(D)}\fo_D\wedge \fo_{\del\cN}.$$\noindent
	By Lemma~\ref{lem:permutation-marked-points}\eqref{permute-boundary}, we may assume $i = h$. By \cite[Corollary~7.3]{CZ24}, we have 
	$$\fo_D = (-1)^{\omega_\fs(\beta)} \fo_{D'}\wedge\fo_L\dul = (-1)^{\omega_\fs(\beta)+nh + n\mu_L(\beta)}\fo_{L}\dul \wedge \fo_{D'}.$$\noindent
	if $h = 1$. If $h > 1$, then by \cite[Corollary~A.13]{CZ24} and \cite[Proposition~3.1.7]{WW17}, the stable trivialisation of the pullback of $TL$ to the boundary of the surface is determined by its restriction to the smooth boundary circles. Thus, we have 
	$$\fo_D = (-1)^{\ind(D)n}\fo_L\dul\wedge \fo_{\wt D} = (-1)^{(h-1)n}\fo_L\dul\wedge \fo_{\wt D}$$\noindent
	by the definition of the intrinsic orientation. We now turn to the orientation sign of the clutching map $\rho \cl \cN'\to \del\cN$ on the level of moduli space of stable curves. Picking up a sign of $(-1)^{h-i}$, we may assume $i = h$. Let $C\in \cN'$ be a curve with a unique interior node, splitting $C$ into $C_1\in \cN_1'$ and $C_2= S^2$ lying in $\cN'_2$, where $C_2$ carries the $k+1$st marked point and one other interior marked point (since we may assume without loss of generality that $k > 0$). Then the isomorphism $\orl(\bC)\orl(\cN_1')\orl(\cN_2')\cong\orl(\cN')$ is orientation-preserving and the isomorphism induced by $\rho$ factors as 
	\begin{center}\begin{tikzcd}
			\orl(\bR)\orl(\cN') \arrow[r,""] \arrow[d,"(1)"]&\orl(\cN)\\ 
			 \orl(\cN_1')\orl(\bR)\orl(\cN'_2)\arrow[r,"(2)"] & \orl(\cN_1')\orl(\cN_2) \arrow[u,"(3)"] \end{tikzcd} \end{center}
	where $\cN_2$ is the moduli space of disks obtained by smoothing of the boundary node and we omit $\orl(\bC)$. Then, $(1)$ has orientation sign $(-1)^{h-1+|\ell|}$, $(2)$ has orientation sign $(-1)$ by \cite{ST16} and $(3)$ has orientation sign $1$ by Proposition~\ref{prop:interior-clutching-orientation}. Hence,
	\begin{equation}\label{} 
		\orl(\rho) = (-1)^{i+|\ell|}
	\end{equation}
	and the claimed sign follows.
\end{proof}

\section{Cubical cobordisms and Thom systems}\label{subsec:cubical-cobordisms}
If a compact space admits an oriented global Kuranishi chart without boundary, it has a virtual fundamental class. In our case, where the boundary is nonempty, one has to work on the level of chains. A literal virtual fundamental chain requires the choice of a \emph{Thom form}, that is, a representative of the Thom class of the obstruction bundle, as well as a fundamental chain of the thickening.
Instead of a virtual fundamental chain, we construct operations based on integration, which can be considered as `virtual pushforwards' of differential forms. These pushforwards need to be constructed compatibly in order to define strict algebraic structures. Instead of making the auxiliary choices necessary for the construction of global Kuranishi charts coherently, we use the equivalence statements of Theorem~\ref{thm:boundary-strata-equivalent} and Proposition~\ref{prop:interior-clutching-orientation} to interpolate between different choices.

\subsection{Cubical cobordisms}
 Throughout, let $(X,\omega)$ be a closed symplectic manifold and $L\sub X$ an embedded Lagrangian with a relative spin structure $(V,\fs)$. Define 
\begin{equation*}\label{} 
	\cA := \set{(\beta,g,h,k,\ell_1,\dots,\ell_h)\mid\beta \in H_2(X,L;\bZ),\,\omega(\beta) \geq0,\; g,h,k,\ell_i \geq 0,\;\beta\in H_2(X;\bZ)\text{ if }h = 0}.
\end{equation*}
Fixing $J \in \cJ_\tau(X,\omega)$, we abbreviate 
$$\Mbar_{\mathsf{a}}(X,L) := \Mbar_{g,h;k,\ell}^{\,J,\,\beta}(X,L)$$\noindent 
for $\mathsf{a} = (\beta,g,h,k,\ell)\in \cA$ and denote by $\Mbar(X,L)$ the collection $\{\Mbar_{\mathsf{a}}(X,L)\}_{\mathsf{a}\in \cA}$.
Given a stable map graph $\Gamma$ of an open stable map, we can associate to each vertex $v \in V(\Gamma)$ a unique element $\mathsf{a}_v \in \cA$. Denote by 
$$\Mbar_{\Gamma}(X,L)\sub \p{v \in V(\Gamma)}{\Mbar_{\mathsf{a}_v}(X,L)}$$ \noindent
the associated fibre product. Its quotient admits an embedding 
\begin{equation}\label{eq:embedding-of-fibre-product}\Mbar_{\Gamma}(X,L)/\Aut(\Gamma)\to \Mbar_{\mathsf{a}}(X,L) \end{equation}
onto a (virtual) boundary or divisor stratum $\del_\Gamma\Mbar_{\mathsf{a}}(X,L)$ of $\Mbar_{\mathsf{a}}(X,L)$. The stabilisation map induces the map
\begin{equation}\label{} \stb \cl \Mbar_\Gamma(X,L)\,\to \,\Mbar_{\Gamma^{\stb}}\end{equation}
where $\Gamma^{\stb}$ is the graph obtained by forgetting the degrees associated to vertices of $\Gamma$ and collapsing any vertices that become unstable.\par
Given $\mathsf{a} = (\beta,g,h,k,\ell)$ in $\cA$, write $\mathsf{a}_0 := (\beta,g,h,0,0)$ for the tuple without marked points. Let $\cK_{\mathsf{a}_0} = \cK_{\alpha_{\mathsf{a}}}$ be the global Kuranishi chart of Proposition~\ref{prop:open-gw-global-chart-existence}\eqref{gkc-indeed} and let $\cK_{\mathsf{a}}$ be the pullback of $\cK_{\mathsf{a}_0}$ as in Proposition~\ref{prop:marked-points-gkc-simple}. Given a stable map graph $\Gamma$, we can associate to each vertex $v \in V(\Gamma)$ an element $\mathsf{a}_v\in \cA$ and thus obtain a derived orbifold chart 
\begin{equation}\label{eq:graph-gkc}\cK_\Gamma \sub \p{v\in V(\Gamma)}{\cK_{\mathsf{a}_v}}\end{equation}
for $\Mbar_\Gamma(X,L)$, where we have used the `stabilised' evaluation maps from Lemma~\ref{lem:evaluation-relvative-submersion} \emph{except in the case where $\Mbar_{\mathsf{a}_v}(X,L)$ is a moduli space of constant genus zero maps with at most one outgoing marked point, which lies on the boundary if the curves have boundary}. Therefore, the fibre product is a well-defined global Kuranishi chart.

\begin{theorem}\label{thm:cubical-cobordisms-enhanced} 
	Given a choice of unobstructed auxiliary datum $\alpha_{\mathsf{a}}$ for each $\mathsf{a}\in \cA$, there exists for each stable map graph $\Gamma$ a \emph{smooth} global Kuranishi chart ${\cKc_\Gamma}$ for $I^{E(\Gamma)}\times\Mbar_{\Gamma}(X,L)$ so that the following holds
		\begin{enumerate}[label=\arabic*),leftmargin=20pt,ref=\arabic*]
		\item ${\cKc_{\mathsf{a}}} =\cK_{\mathsf{a}}$ for any $\mathsf{a} \in \cA$.
		\item ${\cKc_\Gamma}$ is oriented.
		\item\label{cubical-1-boundary-enhanced} for any edge $e$ of $\Gamma$ we have a local smooth embedding
		\begin{equation}\label{eq:restricted-e-0-enhanced} 
			{\cKc_{\Gamma}}|_{\{t_e = 0\}} \;\stackrel{\sim\,}{\longrightarrow}\; (\pm 1)\,\del_\Gamma{\cKc_{\Gamma_e}}
		\end{equation}
	of degree \[d_1(\Gamma,\Gamma_e) = |\{\phi\in \Aut(\Gamma)\mid f\g\phi = f\},\] where $f\cl \Gamma\to \Gamma_e$ is the underlying contraction of graphs.
	\item\label{cubical-0-boundary-enhanced} The restriction to the other boundary face induces a smooth embedding
	\begin{equation}\label{eq:restricted-e-1-enhanced} {\cKc_{\Gamma}}|_{\{t_e = 1\}} \;\stackrel{\sim\,}{\longrightarrow}\; (\pm 1)\,{\cKc_{\Gamma_1}}\times_{Y}{\cKc_{\Gamma_2}}
	\end{equation}
	where $\Gamma_1*_e \Gamma_2 = \Gamma$ and $Y = X$ or $L$, depending on whether $e$ is an interior or a boundary edge.
	\item\label{evaluation-cubical} ${\cKc_\Gamma}$ admits a smooth submersive lift $\eva_\Gamma\cl {\cKc_\Gamma}\to X^k\times L^\ell$ of the evaluation map on $\Mbar_\Gamma(X,L)$ and the (local) embeddings of~\eqref{eq:restricted-e-0-enhanced} and~\eqref{eq:restricted-e-1-enhanced} intertwine these lifts.
	\end{enumerate}
\end{theorem}

\begin{definition}\label{de:cubical-cobordism} We call the system $\{\cKc_{\Gamma}\}_{\Gamma_*}$ a \emph{cubical cobordism} for $\{\Mbar_{\mathsf{a}}(X,L)\}_{\mathsf{a}\in \cA}$.
\end{definition}

\begin{proof} 
	 The proof relies on the observation that the construction of the double thickening, e.g. in the proof of Proposition~\ref{prop:uniqueness-up-to-equivalence} or Theorem~\ref{thm:boundary-strata-equivalent}\eqref{or-gluing-hg-3}, can be adapted to yield a global Kuranishi chart $\cK $ with a submersion $\cT\to [0,1]$ so that for $i = 0,1$, the restriction $\cK|_{\{i\}}$ admits a map 
	\begin{equation}\label{} \cK|_{\{i\}}\stackrel{\sim\;}{\to} \cK_i\end{equation}
	i.e., a morphism that is given by a stabilisation and a group enlargement as in Definition~\ref{de:equivalence of gkc over Y}. Since we stabilised the evaluation map to a relative submersion (Lemma~\ref{lem:evaluation-relvative-submersion}), we can take that morphism to preserve the evaluation map.
	Recall that, roughly, the tuples of the double thickening of the proof of Theorem~\ref{thm:boundary-strata-equivalent}\eqref{eq:or-sign-h3} are tuples \[(\varphi,\varphi_1,\varphi_2,u,C = C_1\vee C_2,\eta,\eta_1,\eta_2s),\]
	where 
	\begin{itemize}
		\item $C_1,C_2$ are nodal Riemann surfaces clutched at marked points to form the curve $C$
		\item $u \cl C\to X$ is a smooth stable map 
		\item $\varphi \cl C\to \bC P^N$ is a framing of $C$ with $\deg(\varphi) = \deg(\fL_u^{\otimes p})$,
		\item $\varphi_i \cl C_i \to \bC P^{N_i}$ is a framing of $C_i$ with $\deg(\varphi_i) = \deg(\fL_{u|_{C_i}}^{\otimes p_i})$ for $i = 1,2$
		\item $\eta \in H^0(C,(\varphi,u)^*E_k)$ and $\eta_i \in H^0(C_i,(\varphi_i,u|_{C_i})^*E_{k_i})$ are perturbation terms so that 
		\begin{equation}\label{eq:equivalent-delbar} \delbar_J(u|_{C_i}) +\lspan{\eta}\g d\varphi|_{C_i} + \lspan{\eta_i}\g d\varphi_i  = 0\end{equation}
		for $i = 1,2$.
	\end{itemize}
	Instead, one can consider tuples $(t,\varphi,\varphi_1,\varphi_2,u,C = C_1\vee C_2,\eta,\eta_1,\eta_2s)$, where $t\in I = [0,1]$ and $(\varphi,u,\eta_0,\eta_1)$ are the same data as before except that they satisfy the following perturbed Cauchy-Riemann equation
	\begin{equation}\label{eq:cobordism-delbar} 
		\delbar_J (u|_{C_i}) + t\,\lspan{\eta}\g d\varphi|_{C_i} + (1-t)\,\lspan{\eta_i}\g d\varphi_i = 0
	\end{equation}
	for $i = 1,2$ instead of \eqref{eq:equivalent-delbar}. This yields the cubical cobordism in the case where the underlying stable map graph $\Gamma$ has a single edge.
	 
	To define the cubical cobordism $\cKc_\Gamma$ of global Kuranishi charts for an arbitrary graph $\Gamma$, we define for a contraction $\Gamma\to \Gamma'$ of stable map graphs the function 
	\begin{equation}\label{} \chi_{\Gamma'} = \chi_{\Gamma\to \Gamma'}\cl I^{E(\Gamma)} \rightarrow [0,1]\end{equation} 
	by 
	\begin{equation}\label{eq:interpolating-perturbations} \chi_{\Gamma'}(t) = \prod_{e\in E(\Gamma')} t_e \prod_{e\notin E(\Gamma')} (1-t_e)\end{equation}
	using that there exists a canonical embedding $E(\Gamma')\hkra E(\Gamma)$. 
	We record a few properties of these functions here that will be used later.
	\begin{enumerate}[label=($\chi$\arabic*),leftmargin= 30pt,ref=$\chi$\arabic*]
		\item\label{vanishing} If $e$ is contracted by $\Gamma'\to \Gamma$, then $\chi_{\Gamma'}$ vanishes on $\{t_e = 1\}$.
		\smallskip
		\item\label{restricting} If the map $\Gamma\to \Gamma'$ contracts the edge $e$, then $\chi_{\Gamma\to\Gamma'}|_{\{t_e = 0\}} = \chi_{\Gamma_e\to\Gamma'}$.
		\smallskip
		\item\label{summing} We have for any $\Gamma$ that $\s{\Gamma\to \Gamma'}{\chi_{\Gamma'}}\equiv 1$.
		\item\label{clutching} If $\Gamma$ is the clutching of $\Gamma_1$ and $\Gamma_2$ along the edge $e_*$, then we have for any contraction $\Gamma\to \Gamma'$, which does not contract $e_*$, that $\Gamma'$ is the clutching of contractions $\Gamma_1\to \Gamma'_1$ and $\Gamma_2\to \Gamma_2'$ and  we have for $t\in I^{E(\Gamma)}$ with $t_{e_*} = 1$ that 
		$$\chi_{\Gamma'}(t) = \chi_{\Gamma_1'}((t_e)_{e\in E(\Gamma_1)})\;\chi_{\Gamma_2'}((t_e)_{e\in E(\Gamma_2)}).$$
	\end{enumerate}
	The third property follows from induction; the other properties are immediate from the definition. We can now start with the definition of the global Kuranishi charts $\cKc_\Gamma$. The construction is divided into three steps.\par

	
	\noindent\textbf{Step 1:} We first define a rel--$C^\infty$ global Kuranishi chart $\wh{\cKc_\Gamma}$ for $[0,1]^{E(\Gamma)}\times \Mbar_\Gamma(X,L)$ for each suitable stable map graph $\Gamma$ so that all claims of the proposition hold but with relative smoothness instead of smoothness and with the caveat that we do get a `uniform' evaluation map since we stabilise them in a very specific way. Fix a suitable $\Gamma$ and define the base space $\wh{\cBc_\Gamma}$ to be simply the base space $\cBc_\Gamma$ defined above. The thickening $\wh{\cTc_\Gamma}$ consists of tuples 
	\begin{equation}\label{eq:point-in-cubical-cobordism}
		(t,\varphi= (\varphi_{\Gamma'})_{\Gamma\to \Gamma'},u = (u_v)_{v\in V(\Gamma)},(\gamma_{\Gamma'})_{\Gamma\to \Gamma'},\eta = (\eta_{\Gamma'})_\Gamma,(\xi_{v,e})_{(v,e)\in F(\Gamma)})
	\end{equation}
	where 
	\begin{itemize}[leftmargin=20pt]
		\item $t\in [0,1]^{E(\Gamma)}$,
		\item $\varphi\in \cBc_\Gamma$, so that $\varphi_\Gamma = (\varphi_v)_{v\in V(\Gamma)}$
		\item $u_v \cl C_v:=\cC_{\varphi_v}\to X$ is a smooth map of degree $\beta_v$
		\item $\xi_{v,e}\in T_{u_v(z_{v,e})}Y_e$ is a tangent vector, where $Y_e = X$ if $e$ is an interior edge and $Y_e = L$ otherwise,
		\item $\gamma_{\Gamma'} =(\gamma_{\Gamma',v'})_{v'\in V(\Gamma')}$ is a sequence of elements $\gamma_{\Gamma',v'}\in H^1(C_{v'},\cO_{C_{v'}})$ satisfying
			\begin{equation*}\label{eq:line-bundle-matchup} [\varphi_v^*\cO_{\bC P^{N_v}}(-1)]\otimes [\fL_{u_v}^{\otimes p_v}] = \exp(\gamma_{\Gamma',v})\end{equation*} 
			with $\varphi_{\Gamma',v'}$,
		\item $\eta_{\Gamma'} = (\eta_{\Gamma',v})_{v\in V(\Gamma)}$ is a sequence of perturbations with 
		\[\eta_{\Gamma',v} \in H^0(C_v,(\varphi_{\Gamma'}|_{C_v},u_v)^*E_{\Gamma',f(v)}) \] where $f \cl \Gamma\to \Gamma'$ is the underlying contraction, 
	\end{itemize}
 so that 
 \begin{enumerate}[label=\roman*),leftmargin=23pt,ref=\roman*]
 	\item  for any edge $e = \{v,v'\}$ of $\Gamma$ we have
 	\begin{equation*}\label{eq:almost-agreeing-maps} 
 		\exp_{u_v(z_{v,e})}(t_e\xi_{v,e}) = \exp_{u_{v'}(z_{v',e})}(t_e\xi_{v',e})
 	\end{equation*}
 	and
 	\begin{equation}
 		\label{eq:almost-agreeing-perturbation} 
 		d\exp_{u_v(z_{v,e})}(t_e\xi_{v,e})\g\eta_{\Gamma',v}(z_{v,e}) = d\exp_{u_{v'}(z_{v',e})}(t_e\xi_{v',e})\g\eta_{\Gamma',v'}(z_{v',e})
 	\end{equation}
 	\item any $u_v$ satisfies the perturbed Cauchy--Riemann equation 
 	\begin{equation}\label{eq:cubical-perturbation} 
 		\delbar_J\, u_v + \s{\Gamma\to \Gamma'}{\chi_{\Gamma',\Gamma}(t)\lspan{\eta_{\Gamma',v}}\g d\varphi_{\Gamma'}|_{C_v}} = 0,
 	\end{equation}
 	 where $\chi_{\Gamma,\Gamma'}$ is defined in Equation~\eqref{eq:interpolating-perturbations}, and the perturbed operator
 	 \begin{gather*}\label{eq:perturbed-enhanced-cob} 
 	 	\notag D\delbar_J(u_v) + \s{\Gamma\to \Gamma'}{\chi_{\Gamma',\Gamma}(t)\lspan{\cdot}\g d\varphi_{\Gamma'}|_{C_v}}\\ 
 	 	 C^\infty(C_v,u_v^*TX)\oplus \bigoplus\limits_{\Gamma\to \Gamma'}H^0(C_v,(\varphi_{\Gamma'}|_{C_v},u_v)^*E_{v'}) \to \Omega^{0,1}(C_v,u_v^*(TX,TL))
 	 \end{gather*}
  is surjective when restricted to $C^\infty(C_v,u_v^*TX)_{D_v} \oplus \bigoplus\limits_{\Gamma\to \Gamma'}H^0(C_v,(\varphi_{\Gamma'}|_{C_v},u_v)^*E_{v'})$.
  \item for any $f\cl \Gamma\to \Gamma'$, the evaluation map 
  \begin{gather*}\label{} 
  	\bigoplus\limits_{v\in V(\Gamma)}H^0(C_v,(\varphi_{\Gamma'}|_{C_v},u_v)^*E_{v'}) \to \Omega^{0,1}(C_v,u_v^*(TX,TL)) \ \to\ \bigoplus\limits_{e\in E(\Gamma)}T_{\exp_{u_v(z_{v,e})}(t\xi_{v,e})}Y_e\\
  	\wh\eta_{\Gamma'} \ \mapsto \ (d\exp_{u_v(z_{v,e})}(t_e\xi_{v,e})(\wh\eta_{\Gamma',v}(z_{v,e})) - d\exp_{u_{v'}(z_{v',e})}(t_e\xi_{v',e})(\wh\eta_{\Gamma',v'}(z_{v',e})))_{e\in E(\Gamma)}
  \end{gather*}
is surjective.\end{enumerate} 
We define the obstruction bundle $\wh{\cEc_\Gamma}\to \wh{\cTc_\Gamma}$ to have fibre
\begin{equation*}
	\label{eq:fibre-obstruction-cubical-cobordism-enhanced}
	 \bigoplus\limits_{\Gamma\to \Gamma'}\lbr{\bigoplus\limits_{v'\in V(\Gamma)}H^1(C_{v'},(\bC,\bR))\oplus \fp\fo(N_{v'} +1)\oplus \bigoplus\limits_{v\in f\inv(v')}H^0(C_v,(\varphi_{\Gamma'}|_{C_v},u_v)^*E_{v'})}
\end{equation*}
over the element~\eqref{eq:point-in-cubical-cobordism}. We let $\wh\fs_\Gamma$ be the canonical section. The covering group agrees with the covering group of $\cKc_\Gamma$. Since regularity is an open condition, the canonical map $\wh\fs\inv(0)\to I^{E(\Gamma)}\times \Mbar_\Gamma(X,L) $ induces an isomorphism on the quotient. Applying the mollification procedure of Lemma~\ref{lem:mollification-of-obstruction-section} inductively to the obstruction sections, we thus obtain a system of rel--$C^\infty$ global Kuranishi charts $\wh{\cKc_\Gamma}$ satisfying the first four properties. We define the evaluation map $\wh\eva_\Gamma\cl \wh{\cKc_\Gamma}\to X^k \times L^\ell$ by stabilising the natural evaluation maps \emph{except} in the case where $\Gamma$ has a constant vertex of genus zero with at most one boundary component and at most one outgoing marked point (which is a boundary marked point in the case of discs). Thus, the system $\set{\wh{\cKc_\Gamma}}_\Gamma$ satisfies all the claimed properties except for smoothness and~Property \eqref{evaluation-cubical}.\\

\noindent\textbf{Step 2:} We now equip the cubical cobordisms constructed in the first step with smooth structures. Since the relative smoothing theory discussed in \S\ref{subsec:smoothing-theory} only yields a smooth structure, which is unique up to concordance (Lemma~\ref{lem:invariance-of-smoothing}), we have to add a bigger outer collar to also interpolate between possibly different choices of smooth structures. These outer collars will also allow us to interpolate between the stabilised and non-stabilised evaluation maps on the global Kuranishi charts for moduli spaces of constant discs and spheres with at most one outgoing marked point (which is on the boundary in the case of discs). To do so, shrink $\wh{\cKc_\Gamma}$ down to the interval $[\frac14,\frac34]$, rescaling the functions $\chi_{\Gamma,\Gamma'}$ appropriately. Then, use Theorem~\ref{thm:to-smooth} to lift the rel--$C^\infty$ structure on $\wh{\cKc_\Gamma} \rightarrow \cBc_\Gamma \times Y_\Gamma$, where $Y_\Gamma$ denotes the target of the evaluation map, to a smooth structure on $\wh{\cKc_\Gamma}$. This ensures the evaluation maps are smooth. If $\Gamma$ represents a constant map of genus $0$ with at most one outgoing marked point, we can just use the usual smooth structure on the trivial global Kuranishi chart $\Mbar_\Gamma$. 

Since one can lift a $G$-smoothing through a rel--$C^\infty$ covering, we obtain for any stable map graph $\Gamma$ and any vertex $v$ of $[0,1]^{E(\Gamma)}$ a relative $G$-smoothing. Using Lemma~\ref{lem:invariance-of-smoothing} inductively, one can extend these over the strips $[0,\frac14]$ and $[\frac34,1]$ to obtain the charts ${\cKc_\Gamma}$ of the claim. We illustrate this step for the case of $\Gamma = (v_1 \stackrel{e_1}{-}v_2\stackrel{e_2}{-}v_3)$ with the following figure over $[0,1]^2$.
\begin{figure}[H]
	\includegraphics[scale=1.5]{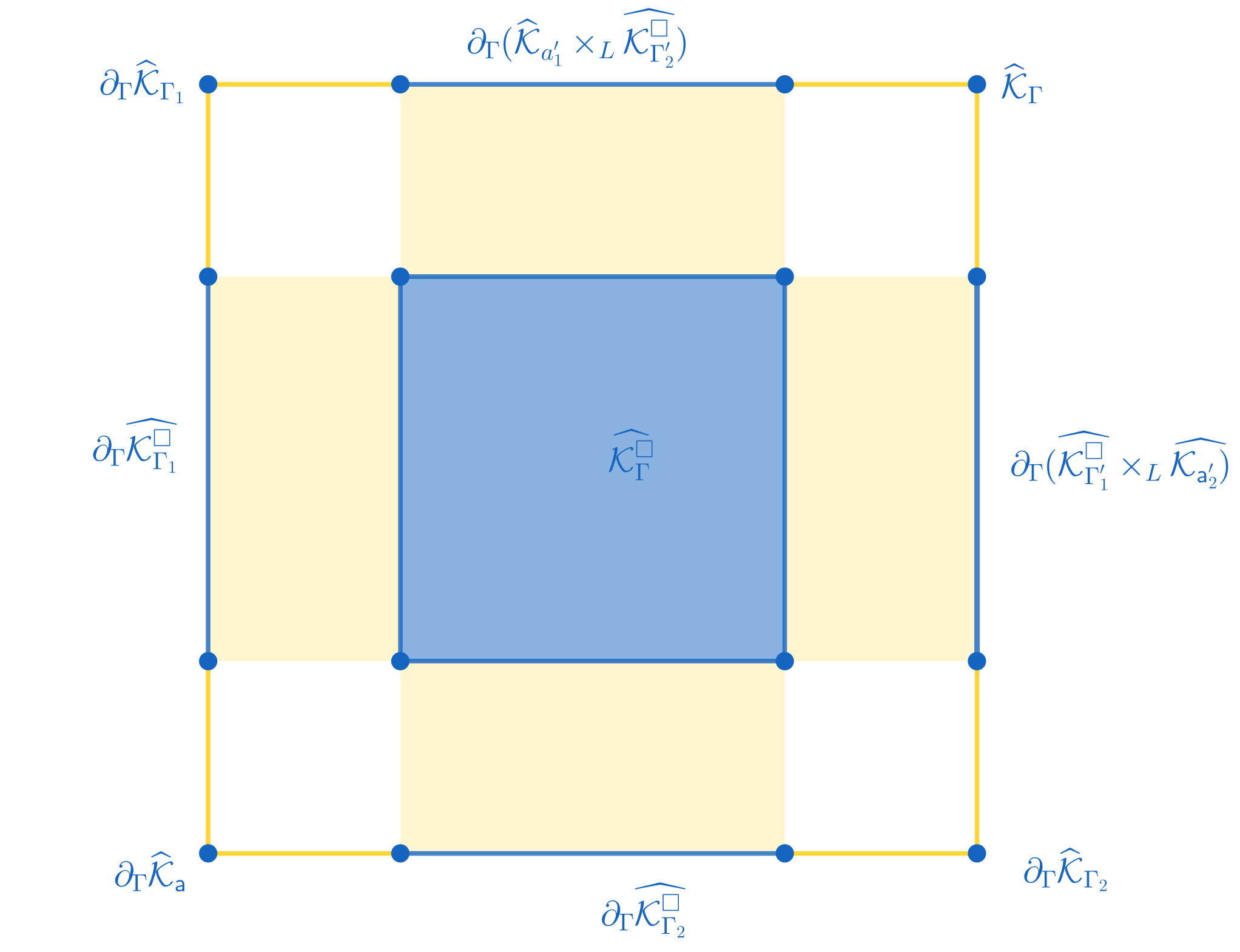}
	\caption{Constructing the smooth structure on ${\cK^\square_\Gamma}$}
	\end{figure}
	
	\noindent
	Here $\Gamma_1$ is obtained from $\Gamma$ by collapsing the edge $e_2$ and similarly for $\Gamma_2$, while $\mathsf{a}$ is the graph obtained from $\Gamma$ by collapsing all edges. 
 We start off with smooth structures on the blue parts which are given by Theorem~\ref{thm:to-smooth}, and are, respectively, pulled back from the respective boundary stratum. Then, we apply Lemma~\ref{lem:invariance-of-smoothing} a first time to extend the smooth structures from the blue parts (after stabilisation) over the yellow lines and rectangles. Finally, we apply Lemma~\ref{lem:invariance-of-smoothing} a second time to extend the smooth structure over the white squares. We also use these `additional' `collars to interpolate between the stabilisation of the evaluation map and its non-stabilised version on moduli spaces of constant stable maps of genus $0$.\\
 
 \noindent\textbf{Step 3:} We will modify the construction of the smooth structure on the cubical cobordisms so that the forgetful maps are smooth. For graphs which have a minimal number of marked points to be stable, i.e. those graphs for which the only components with marked points are constant curves of genus $0$ or $1$, construct the smooth structure as in Step 2. Then, for any other stable map graph $\Gamma$, let $\Gamma_0$ be the graph obtained by forgetting all marked points, except those that are needed on constant components. Define the smooth structure on ${\cKc_\Gamma}'$ by the fibre product 
 \begin{equation*}
 	\begin{tikzcd}
 		{\cKc_\Gamma}' \arrow[r,""] \arrow[d,""]&{\cKc_{\Gamma_0}} \arrow[d,""]\\ 
 		\cBc_\Gamma \arrow[r,""] &\cBc_{\Gamma_0}.
 		\end{tikzcd} 
 	\end{equation*}
 	The evaluation maps on ${\cKc_\Gamma}'$ might not be smooth. Thus, we define ${\cKc_\Gamma}$ to be the stabilisation of ${\cKc_\Gamma}'$, and take the smooth structure to be the one induced by Lemma \ref{lem:smoothing-maps}. Inducting over all stable map graphs gives the result.
\end{proof}

\begin{definition}
	\label{def:stabilisation-map-on-cube}
	Let the map $\cKc_\Gamma \rightarrow I^{E(\Gamma^{\stb})} \times \Mbar^{\Gamma^{\stb}}$ be given as follows.
	The map $\cKc_{\Gamma} \rightarrow \Mbar^{\Gamma^{\stb}}$ is the usual map stabilising the domain of the map.
	The map $\cKc_\Gamma \rightarrow I^{E(\Gamma^{\stb})}$ is given by the composition $\cKc_\Gamma \rightarrow I^{E(\Gamma)} \rightarrow I^{E(\Gamma^{\stb})}$, where the map $I^{E(\Gamma)} \rightarrow I^{E(\Gamma^{\stb})}$ is defined as follows.
	The stable graph $\Gamma^{\stb}$ is obtained from $\Gamma$ by contracting certain edges or by combining edges together (in the case of interior vertices being forgotten).
	Suppose $e_1, \dots, e_k \in \Gamma$ get combined into the edge $e \in \Gamma^{\stb}$. Then the map $I^{E(\Gamma)} \rightarrow I^{E(\Gamma^{\stb})}$ sends \[(t_{e_1}, \dots, t_{e_k}) \mapsto t_e = 1 - \prod_{i} (1-t_{e_i}).\]
	Moreover it forgets the parameters associated to contracted edges.
\end{definition}

\begin{example}
	Consider the stabilisation of a stable map graph shown in the figure:
	\begin{figure}[H]
		\label{fig:stable-graph-contraction}
		\includegraphics[scale = 0.7]{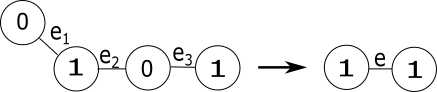}
	\end{figure} 
	\noindent All vertices denote Riemann surfaces with one boundary component, and the number in the vertex denotes the genus. The stabilisation map contracts the edge $e_1$, and the edges $e_2$ and $e_3$ get combined to form the edge $e$.
	The map $I^{E(\Gamma)} \rightarrow I^{E(\Gamma^{\stb})}$ is thus given by 
	\begin{equation*}\label{eq:stabilisation-on-cube} 
		(t_1, t_2, t_3) \mapsto t = 1-(1-t_2)(1-t_3).
	\end{equation*}
\end{example}

\begin{lemma}\label{lem:extended-clutching-maps} Given stable map graphs $\Gamma$ and $\Gamma'$ with marked points labelled by $y \sqcup y'$, respectively $y'\sqcup y''$, there exists a smooth embedding
	\begin{equation}
		\label{eq:extended-clutching-map-gkc}
		\psi\, \cl\, {\cKc_{\Gamma_2}}{\times}_{X^{y'_c}\times L^{y'_o}} {\cKc_{\Gamma_1}}\,\to \,{\cKc_{\Gamma}}, 
		\end{equation}
	where $\Gamma = \Gamma_1 *_{y'}\Gamma_2$, covering the clutching map 
	\begin{equation*}
		\label{eq:extended-clutching-curves} 
		I^{E(\Gamma_1^{\stb})}\times\Mbar_{\Gamma_1^{\stb}}\times I^{E(\Gamma_2^{\stb})}\times  \Mbar^{\Gamma_2^{\stb}}\,\to \,I^{E(\Gamma^{\stb})}\times \Mbar_{\Gamma^{\stb}} 
		\end{equation*} 
	defined by the clutching map on the moduli spaces of stable curves and by $(t_1,t_2)\mapsto (t_1,t_2,1)$ on the cubes.
\end{lemma}

\begin{proof} Replace ${\cKc_\Gamma}$ and ${\cKc_{\Gamma'}}$ by the group-enlarged moduli space, where we allow for more framings (see e.g., the first step in the proof of Proposition~\ref{prop:uniqueness-up-to-equivalence}). Then, the clutching map is defined by simply combining the data coming from the two factors. By construction, this is a well-defined smooth embedding.
\end{proof}

\subsection{Thom forms} Recall that the virtual fundamental class of an oriented derived orbifold chart $(\cT,\cE,\obs)$ for $Z$ is given by 
\begin{equation}\label{}\lspan{\gamma,\vfc{Z}} \,=\, \int_{\cT}\gamma\cdot \obs^*\text{Th}_\cE,\end{equation}
where $\text{Th}_\cE$ is the Thom class of $\cE$. This virtual fundamental class is unique and only depends on the equivalence class of the chart. In the case with boundary, there is no such unique invariant, similar to how manifolds with boundary have a (non-unique) fundamental chain. In our setting, we will also have to choose a representative of the Thom form $\text{Th}_\cE$. Since these representatives will be ubiquitous in the next section, we make the following definition.

\begin{definition}\label{} A \emph{Thom form} is a representative of the Thom class of an orbi-bundle. A Thom form $\eta$ of a derived orbifold chart $(\cT,\cE,\obs)$ is a Thom form of its obstruction bundle. We call it \emph{proper} if $\obs^*\eta$ has compact support.\end{definition}


\begin{remark}\label{rem:properness-thom-form} Given a Thom form $\eta$ for a derived orbifold chart $\cK= (\cT,\cE,\obs)$ and a neighbourhood $U \sub \cT$ of $\obs\inv(0)$, we can pull back $\eta$ to $\cE|_U$ to obtain a Thom form for $\cK|_U$. However, this operation does not preserve the properness of Thom forms.\end{remark}

\noindent
We first show that (proper) Thom forms can be chosen to behave compatibly under stabilisations. To this end, let $(W, p^*\cE \oplus p^*W, \widetilde{s} :=p^*s \oplus \Delta_{W} )$ be obtained from $(\cT, \cE, s)$ by stabilisation along the orbi-bundle $p: W \rightarrow \cT$. We consider the commutative diagram
\begin{equation*}	
	\begin{tikzcd}
		p^*\cE \oplus p^*W \arrow[rrd,bend left,"\pi_{\cE}"]
		\arrow[ddr,bend right,swap,"\pi_{W}"]\\
		& \cE \oplus W \arrow[lu,"\Delta_W"] \arrow[d] \arrow[r,"p_{\cE}"] & \cE \arrow[d]  \\
		& W \arrow[r,swap,"p"]  & \cT
	\end{tikzcd}
\end{equation*}
where the map $\Delta_W(e,w) = (w,e,w)$ for $(e,w)\in \cE\oplus W$ is induced by the diagonal map.

\begin{lemma}
	\label{lem: from base to stabilisation}
	Given Thom forms $\eta_{\cE}$ and $\eta_{W}$ for $\cE \rightarrow \cT$ and $W \rightarrow \cT$, the form 
	$$\widetilde{\eta} := \pi_{\cE}^*\eta_{\cE} \wedge \pi_{W}^*\eta_W$$\noindent
	is a Thom form for $p^*\cE \oplus p^*W$. It satisfies $p_*\wt\obs^*\widetilde{\eta} = \obs^*\eta_\cE$. If $\eta_\cE$ is proper, then so is $\wt\eta$.
\end{lemma}

\begin{proof} First note that 
	$$\wt\obs^*\widetilde{\eta} = (p^*s)^* \pi_{\cE}^* \eta_{\cE} \wedge \Delta_W^* \pi_W^* \eta_W = p^*s^*\eta_\cE \wedge \eta_W,$$\noindent whence it has support contained in the support of $\eta_W$. Thus $p_*(\wt\obs^*\widetilde{\eta})$ is well defined. Next observe that $\widetilde{\eta}$ is closed, and for $w \in W$, we have: \begin{equation*}
		\int_{(p^*\cE \oplus p^*W)_w} \pi_{\cE}^*\eta_{\cE} \wedge \pi_{W}^*\eta_W = \left(\int_{\cE_{p(w)}} \eta_\cE\right)\left(\int_{W_{p(w)}} \eta_W\right) = 1,
	\end{equation*}
	deducing that $\widetilde{\eta}$ is indeed a Thom form. Finally, we compute 
	\begin{equation*}
		p_*\wt\obs^*\widetilde{\eta} = p_*(p^*s^*\eta_\cE \wedge \eta_W) = \obs^*\eta_{\cE} \wedge p_*\eta_W = \obs^*\eta_{\cE}.
	\end{equation*}
	As the support of $\wt\obs^*\wt\eta$ agrees with the support of $\obs^*\eta_\cE$, the last assertion is immediate.
\end{proof}

\begin{lemma}
	\label{lem: from stabilisation to base}
	If $\widetilde{\eta}$ is a Thom form  for $p^*\cE \oplus p^*W \xrightarrow{\pi_W} W$,
	the form $$\eta_{\cE} := (p_\cE)_*\Delta_W^*\widetilde{\eta}$$\noindent is a Thom form for $\cE \rightarrow \cT$, satisfying $p_*\wt\obs^*\widetilde{\eta} = \obs^*\eta_\cE$. If $\wt\eta$ is proper, then so is $\eta_\cE$ and if $\wt\eta = \pi_\cE^*\wh\eta_\cE\wedge \pi_W^*\eta_W$, then $\eta_\cE = \wh\eta_\cE$.
\end{lemma}

\begin{proof}
	Since $\Delta_W$ is a section of $p^*(\cE\oplus W)$, $\Delta_W^*\widetilde{\eta}$ has compact vertical support on $\cE \oplus W$. Hence, $(p_\cE)_*\Delta_W^*\widetilde{\eta}$ is well defined. Moreover, $\eta_{\cE}$ is closed as $\widetilde{\eta}$ is closed and the map $p_\cE$ has no vertical boundary. The final property of a Thom form can be checked on cohomology, i.e., we need that $[\eta_\cE]$ is the Thom class for the zero section of $\cE\oplus\cW$. As $\Delta_W$ intersects the zero section of $p^*(\cE\oplus W)$ transversely in $\cT$, it pullbacks back $[\wt\eta]$ to the Thom class of $\cT$ inside $\cE\oplus W$. It follows from the Thom isomorphism theorem that $(p_\cE)_*\Delta_W^*\wt \eta$ is a Thom form for $\cE\to \cT$. The second assertion is a consequence of the projection formula, \cite[Proposition~2.1(4)]{ST16}, while the third assertion is due to the fact that the support of $\obs^*\eta$ is contained in the support of $\wt\obs^*\wt\eta$. The last assertion follows from the fact that $\pi_\cE\Delta_W= p_\cE$ and $\pi_W\Delta_W=  p_W$ and \cite[Proposition~2.1]{ST16}.
\end{proof}

\begin{definition}\label{de:compatible-thom-forms} Let $\cK$ be a global Kuranishi chart and $\cK'$ its stabilisation along a vector bundle $p \cl \cW\to \cT$. Let $P \cl \cE'= p^*\cE\oplus p^*\cW\to \cE$ be the canonical projection. Then, we call a Thom form $\eta$ of $\cK$ \emph{compatible} with a Thom form $\eta'$ for $\cK'$ if $(p_\epsilon)_*\Delta_W^*\eta'$.
\end{definition}

\begin{remark}\label{rem:compatible-thom-forms-consequence} The upshot of having compatible Thom forms is the following. Suppose we are given maps $g^{\pm}: W \rightarrow Y^{\pm}$ that factor through maps $f^{\pm}: \cT \rightarrow Y^{\pm}$ under the projection $p: W \rightarrow \cT$. Then, for  compatible Thom forms $\eta_{\cE}$ and $\wt \eta$, we have 
	\[
	f^+_*((f^-)^*\alpha \wedge \eta_{\cE}) = g^+_*((g^-)^*\alpha \wedge \wt \eta) \]
	for any $\alpha \in \Omega^*(Y^-)$.
\end{remark}

\begin{proposition}\label{prop:thom systems exist} Given a system of (enhanced) cubical cobordisms $\{\cKc_{\Gamma}\}_{\Gamma}$ and a choice of Thom form $\eta_{\mathsf{a}}$ for $\cK_{\mathsf{a}}$ so that $\obs_{\mathsf{a}}^*\eta_{\mathsf{a}}$ is compactly supported, there exists a system of Thom forms $\eta_{\Gamma}$ for $\cE_{\Gamma}$ with the following properties.
	\begin{enumerate}[label=\arabic*),leftmargin=20pt,ref=\arabic*]
		\item The pullback $\obs_{\Gamma}^*\eta_{\Gamma}$ is compactly supported in $\cTc_{\Gamma}$.
		\item If $\Gamma = \mathsf{a}$, then $\eta_{\Gamma} = \eta_{\mathsf{a}}$.
		\item It is compatible with boundary restrictions. More precisely, for any edge $e$ of $\Gamma$, we have that
		\begin{equation}
			\label{eq:thom-contracting-edge}
			\eta_\Gamma|_{\{t_e = 0\}} \sim \eta_{\Gamma_e}|_{\del_\Gamma\cKc_{\Gamma_e}}, 
		\end{equation} 
		and
		\begin{equation}
			\label{eq:thom-separating-edge}
			\eta_\Gamma|_{\{t_e = 1\}} \sim \eta_{\Gamma_1}\times \eta_{\Gamma_2}|_{\cTc_{\Gamma_1}\times_Y \cTc_{\Gamma_2}} 
		\end{equation}
		if $e $ separates $\Gamma$ into $\Gamma_1$ and $\Gamma_2$ with $Y = X$ or $L$ depending on $e$, and 
		\begin{equation}
			\label{eq:thom-nonseparating-edge}
			\eta_\Gamma|_{\{t_e = 1\}} \sim \eta_{\Gamma\sm e}|_{Y\times_{Y^2}\cKc_{\Gamma\sm e}} 
		\end{equation}
		if $e$ is a non-separating node and $\Gamma\sm e$ is the graph given by cutting the edge $e$. Finally, if the edge corresponds to a collapsed boundary circle, then
		\begin{equation}
			\label{eq:thom-boundary-edge-type-E}
			\eta_\Gamma\mid_{\{t_e = 1\}} \sim \eta_{\Gamma\sm e}|_{L\times_{X}\cKc_{\Gamma\sm e}}.
		\end{equation} 
	\end{enumerate}
\end{proposition}

\begin{proof} 
	We construct the system of Thom forms by a double induction. The outer induction is over the partial order on $\cA$ and the inner induction, for each $\mfa \in \cA$, is on $\# E(\Gamma)$ for stable map graphs $\Gamma$ contracting to $\mfa$. The point is that by the time we consider $\Gamma$ we will have already chosen Thom forms for the boundaries appearing in~\eqref{eq:thom-contracting-edge}-\eqref{eq:thom-boundary-edge-type-E}. Thus, we can apply Proposition~\ref{prop:extend-thom-form} to conclude.  
	Explicitly, for $\Gamma$ with $\# E(\Gamma) = 0$, we are given a choice of Thom form already. Suppose, thus, we have constructed the system of Thom forms for all $\cKc_{\Gamma'}$ with $\mfa_{\Gamma'} < \mfa$ and a Thom form for each $\cKc_{\Gamma'}$ with $\#E(\Gamma') < \# E(\Gamma)$. By construction $\cEc_{\Gamma}$ decomposes as a direct sum 
	\[
	\cEc_\Gamma = \bigoplus_{\Gamma \to \Gamma' \to \mfa} \cE_{\Gamma'},
	\]
	where we abuse notation and write $\cE_{\Gamma'}$ for the summands appearing in equation \eqref{eq:fibre-obstruction-cubical-cobordism-enhanced}. For each $\cE_{\Gamma'}$, we obtain a Thom form over certain faces of $I^{E(\Gamma)}$ by stabilising the already obtained Thom forms for the adjacent cubes, or in the case of faces with parameters $t=1$, by stabilising the Thom forms appearing in \eqref{eq:thom-separating-edge}-\eqref{eq:thom-boundary-edge-type-E}. By construction, these agree on lower dimensional faces, so we can use Appendix \ref{sec:extend-Thom-forms} to obtain a Thom form for $\cE_{\Gamma'}$ over the whole $\cKc_\Gamma$. Doing this for all $\Gamma'$ and taking the wedge product yields the required Thom form.
\end{proof}
\subsection{Counting graphs correctly}\label{subsec:graphs} 
To each cubical cobordism $\cKc_\Gamma$, where we now omit the tilde from the notation, we can associate rel--$C^\infty$ maps 
$$\eva^\pm_\Gamma\cl \cTc_\Gamma\to X^{k^\pm}\times L^{\ell^\pm}$$
with $\eva^+_\Gamma$ being a submersion. This allows us to define the so-called $\fq$-operations 
\begin{equation*}
	\fq_\Gamma\cl \Omega^*(X^{k^-}\times L^{\ell^-})\,\to\, \Omega^*(X^{k^+}\times L^{\ell^+}).
\end{equation*}
using the pullback and pushforward along the evaluation maps. In \cite{HH1}, we will construct the morphisms of the DMFT of Theorem~\ref{thm:OCDMFT associated to L} by summing over the operations associated to suitable stable map graphs $\Gamma$. To this end we consider three categories of graphs. 

\begin{definition}\label{de:prestable-graph} An \emph{(open-closed) prestable graph} $\Gamma$ consists of a finite set of vertices $V(\Gamma)$ and a finite set of half-edges $H(\Gamma)$ together with the data of 
	\begin{itemize}[leftmargin=20pt]
		\setlength\itemsep{1.2pt}
		\item an involution $\iota\cl H(\Gamma)\to H(\Gamma)$
		\item a source map $j \cl H(\Gamma)\to V(\Gamma)$
		\item a genus function $g \cl V(\Gamma)\to \bZ_{\ge 0}$,
		\item a finite (ordered) set of boundary circles $S_v$ for each $v \in V(\Gamma)$ allowing us to decompose $V(\Gamma)$ into closed vertices $V_c(\Gamma) = \{v\mid S_v = \emst\}$ and open ones $V_o(\Gamma) = \{v \mid S_v \neq \emst\}$,
		\item a decomposition $H(\Gamma) = H_i(\Gamma) \sqcup H_b(\Gamma)$ into interior and boundary half-edges together with a function 
		\begin{equation}\label{}
			\wt j \cl H_b(\Gamma)\to \union{v}{S_v}
		\end{equation}
		so that $\wt j(h)\in S_{j(h)}$ for each $h \in H_b(\Gamma)$; we set 
		\[S_v\inn \coloneqq \{s\in S_v \mid \wt j\inv(\{s\}) = \emst\}\]
		for $v \in V_o(\Gamma)$;
		\item an ordering of the boundary half-edges $\wt j\inv(s)$ for each $s \in \union{v}{S_v}$,
		\item an orientation (or sign) $\epsilon(e,v)\in \{\pm\}$ for each $h \in H(\Gamma)$ that is either fixed by $\iota$ or with $j(v)$ stable and $h$ being part of a path to another stable vertex.
	\end{itemize}
	We call $\Gamma$ \emph{stable} if each vertex $v \in V(\Gamma)$ satisfies the stability condition 
	\begin{equation}\label{eq:stability-conditon} 2g(v) + |S_v| +2 H_{v,i} + H_{v,b} > 2\end{equation}
	or if $(g(v),h(v)) = 0$ and $v$ has one interior half-edge that is not fixed by $\iota$. We call such unstable vertices \emph{collapsed boundaries}.\footnote{Such unstable vertices do model collapsed boundaries.}
\end{definition}

\begin{definition}\label{de:morphism-prestable-graphs} A \emph{morphism of prestable graphs} $f \cl \Gamma\to \Gamma'$ is the data of functions 
	\begin{enumerate}
		\item $f_V \cl V(\Gamma)\to V(\Gamma')$,
		\item $f_S\cl \union{v}{S_v}\to \union{v'}{S'_{v'}}$,
		\item $f_H \cl H(\Gamma')\to H(\Gamma)$ respecting the decomposition into interior and boundary half-edges,
	\end{enumerate} 
	and satisfying
	\begin{itemize}[leftmargin=15pt]
		\item $f_V$ is surjective and $f_H$ is injective,
		\item $f_H$ intertwines the involutions $\iota$ and $\iota'$,
		\item we have 
		\[j' = f_V \g j \g f_H \qquad \qquad \wt j' = f_S \g \wt j \g f_H\]
		on $H(\Gamma')$,
		\item the genus map satisfies $$g'(v') = b_1(\Gamma_{v'}) + \s{v \in f\inv(v')}{g(v)}$$ for any $v'\in V(\Gamma')$, where $\Gamma_{v'}$ is the graph that is contracted onto $v'$ by $f$ and $b_1$ is the first Betti number;
		\item for any $v'\in V(\Gamma')$ the map $\union{v\in f\inv(v')}{S_v}\to S'_{v'}$ is surjective and $f(S\inn_{v,o})\sub S\inn_{f(v),o}$,
		\item $f_H$ preserves the ordering of the boundary half-edges as described by Figure~\ref{fig:clutching-1} and Figure~\ref{fig:clutching-2} up to cyclic ordering if all boundary half-edges $h$ satisfy $\sigma(h) \neq h$.
	\end{itemize} 
	If $f_H$ preserves the orientation sign of those $h \in H_b(\Gamma')$, which have one, we call it a \emph{orientation-preserving}. It is an \emph{elementary contraction} if $H(\Gamma)\sm\im(f_H)$ consists of a single $\iota$-orbit. We write $\scP\scG(\Gamma,\Gamma')$ for the space of morphisms from $\Gamma$ to $\Gamma'$ and $\scP\scG^+(\Gamma,\Gamma')$ for the subset of orientation-preserving morphisms.
\end{definition}

\begin{remark}\label{} In particular, we have two groups of automorphisms of $\Gamma$: $\Aut(\Gamma)$ and $\Aut^+(\Gamma)$.
\end{remark}

Up to the last point, this is simply an abstract formulation of the clutching maps between different moduli spaces. It may be helpful to compare this definition to Lemma~\ref{prop:boundary-base-space-hg}. We define the category $\scP\scG$ of prestable graphs to with objects prestable graphs, \emph{whose only unstable vertices are of type $(g,h) = (0,1)$}, and morphisms of prestable graphs. It contains the full subcategory $\scS\scG$ of stable graphs. 

\begin{definition}\label{de:stable-map-graph} A \emph{stable map graph} $\Gamma = (\Gamma^{ps},\beta)$ \emph{in $(X,L)$} consists of a prestable graph $\Gamma^{ps}$ together with functions $V_c(\Gamma^{ps})\to H_2(X;\bZ)$ and $ V_o(\Gamma^{ps})\to H_2(X,L;\bZ)$ so that for any unstable vertex $v$ the associated homology class $\beta_v$ satisfies $\omega(\beta_v) > 0$. We define the \emph{total degree} of $\Gamma$ to be 
	$$\deg(\Gamma) \coloneqq \s{v\in V(\Gamma)}{\beta_v},$$ 
	\noindent where we implicitly apply the map $H_2(X;\bZ)\to H_2(X,L;\bZ)$ if $\Gamma$ has both open and closed vertices.
\end{definition}

A morphism of stable map graphs $f\cl (\Gamma,\beta)\to (\Gamma',\beta')$ is a morphism of underlying prestable graphs so that $\beta'_{v'} = \s{f(v) = v'}{\beta_v}$, where we apply the map $H_2(X;\bZ)\to H_2(X,L;\bZ)$ if necessary. This yields a third category $\scS\scG_m$ of graphs. 

\begin{lemma} There exist functors 
	\begin{equation*}\label{}
		\scS\scG_m \,\xra{(-)^{ps}}\, \scP\scG \,\xra{(-)^{\stb}} \scS\scG
	\end{equation*}
	where the first functor is faithful, while the second one is neither full nor faithful in general.\qed
\end{lemma}

From now on, we fix subsets $\text{SG}\sub \text{obj}(\scS\scG)$ and $\text{PG}\sub \text{obj}(\scP\scG)$, so that they contain exactly one representative for each orientation-preserving isomorphism class of graph in the respective category. We define 
\[\text{SG}_m \coloneqq \set{\Gamma\in \text{obj}(\scS\scG_m)\mid \Gamma^{ps} \in \text{PG}}.\]
All graphs that appear in the remainder of this and in the next section will implicitly be taken to be elements of these sets. Given a stable map graph $\Gamma = (\Gamma^{ps},\beta)$, define 
\[\scP\scS^+(\Gamma)\coloneqq \set{(\Gamma^{ps},\beta')\in \text{SG}_m\mid (\Gamma^{ps},\beta')\cong_+ \Gamma }\]
where $\cong_+$ means that they are related by an orientation-preserving isomorphism.
An easy verification shows that 
\begin{equation}\label{eq:prestable-graphs-double-counting}
	|\scP\scS^+(\Gamma)| = \frac{|\Aut^+(\Gamma^{ps})|}{|\Aut^+(\Gamma)|}.
\end{equation}

\begin{example} Consider the (closed) stable graph $\Gamma^{ps}$ with two vertices $v_1,v_2$ of genus $2$, connected by a single edge whose two associated half-edges carry both the sign $+$. If we take $\Gamma = (\Gamma^{ps},\beta)$ where $\beta_{v_1}\neq \beta_{v_2}$, then, $\scP\scS^+(\Gamma) = \{\beta,\beta'\}$, where $\beta'_{v_1} = \beta_{v_2}$ and $\beta'_{v_2} = \beta_{v_1}$. If one half-edge carries the sign $+$ and the other the sign $-$, then $\scP\scS(\Gamma) = \{\beta,\beta'\}$ as before but $\scP\scS^+(\Gamma) = \{\beta\}$.
\end{example}

\begin{remark}\label{rem:permuting-disc-bubbles} For any prestable graph $\Gamma$ with stabilisation $\ov\Gamma$ the inclusion $\Aut^+(\Gamma/\ov\Gamma)\hkra \Aut(\Gamma/\ov\Gamma)$ is an isomorphism. Indeed, any $\phi\in \Aut(\Gamma/\ov\Gamma)$ is given by permuting unstable chains of discs ending in unstable leaves and which are attached to the same stable vertex. By our definition of half-edge orientations, any such automorphism is automatically orientation-preserving.
\end{remark}

Define for a prestable graph $\Gamma$ and a total degree $\beta$ the derived orbifold
\begin{equation}\label{} \cK(\Gamma^{ps},\beta) \coloneqq \scalebox{1.5}{[}\djun{\substack{ps(\Gamma) = \Gamma^{ps}\\\deg(\Gamma) = \beta}}\cKc_\Gamma/\Aut^+(\Gamma^{ps}/\ov\Gamma)\,\scalebox{1.5}{]}
\end{equation}
and for a stable graph $\ov\Gamma$ 
\begin{equation}\label{eq:gkc-of-stable-graph}
	\cK(\ov\Gamma,\beta) \,\coloneqq \djun{\stb(\Gamma^{ps}) = \ov\Gamma}\cK(\Gamma^{ps},\beta)^{\sqcup A^{ps}_\Gamma},
\end{equation}
where 
\begin{equation}\label{eq:non-liftable-automorphisms}
	A_{\Gamma^{ps}}\coloneqq \Aut^+(\ov\Gamma)/\im(\Aut^+(\Gamma^{ps})\to \Aut^+(\ov\Gamma)) 
\end{equation} 
is the set of automorphisms of $\ov\Gamma$ that cannot be lifted to the prestable graph $\Gamma^{ps}$. 

\begin{remark}\label{rem:additional-exponent} 
	We take this additional disjoint union because we are undercounting prestable graphs by taking only one representative per isomorphism class. However, a prestable graph $\Gamma^{ps}$ with stabilisation $\ov\Gamma$ can admit multiple contractions $\Gamma^{ps}\to \ov\Gamma$. Ignoring degrees, the associated clutching maps between moduli spaces (of stable maps) encode different boundary strata if and only if the underlying contractions are \emph{not} related by an automorphism of $\Gamma$. In summary, since 
	$$A_{\Gamma^{ps}} \cong \scP\scG^+(\Gamma^{ps},\ov\Gamma)/\Aut^+(\Gamma^{ps})$$
	we have $|A_{\Gamma^{ps}}|$-many such boundary strata of $\cK_{\ov\Gamma,\beta}$.
\end{remark}

\begin{nota}
	Given an edge $e\in E(\ov\Gamma)$, we let $\cK(\ov\Gamma,\beta)_{|\stb\inv(e)|=1} \sub 	\cK(\ov\Gamma,\beta)$ be the disjoint union indexed by the prestable graphs $\Gamma^{ps}$ where the stabilisation map $\Gamma^{ps}\to \ov\Gamma$ maps a single edge to the edge $e$. By abuse of notation, we denote said edge by $e$ as well.
\end{nota}

\begin{lemma}\label{lem:gkc-fibre-product}
	Suppose $\ov{f}\cl \ov{\Gamma}\to \ov{\Gamma}'$ is an elementary contraction. Then, for any prestable graph ${\Gamma'}^{ps}$, any total degree $\beta$ and the unique edge $e$ contracted by $\ov{f}$, the diagram 
	\begin{equation}\label{dig:contraction-square}
		\begin{tikzcd}
			\cK(\ov\Gamma,\beta)_{|\stb\inv(e)|=1}|_{t_e = 0} \arrow[r] \arrow[d] & \cK(\ov{\Gamma'},\beta)\arrow[d]\\
			I^{E(\cc \Gamma)}\times \Mbar_{\cc \Gamma} |_{t_e = 0} \arrow[r] &I^{E(\cc \Gamma')}\times\Mbar_{\cc \Gamma'} 
		\end{tikzcd}
	\end{equation}
	is a Cartesian square of derived orbifolds up to codimension $1$ in $\cK(\cc\Gamma,\beta)_{|\stb\inv(e)|=1}|_{t_e = 0}$.
\end{lemma}

\begin{proof} It suffices to prove the claim separately for each $\cK({\Gamma'}^{ps},\beta)^{\sqcup A_{\Gamma^{ps}}}$, where ${\Gamma'}^{ps}$ is a prestable graph with stabilisation $\ov\Gamma'$. Suppose first that ${\Gamma'}^{ps}$ is stable itself. In this case, the automorphism groups in \eqref{eq:gkc-of-stable-graph} and $A_{\Gamma^{ps}}$ are trivial. It suffices to show the claim for a single global Kuranishi chart $\cKc_{\Gamma'}\sub \cK(\ov\Gamma',\beta)$. Moreover, because the retraction $\ov{f}\cl \ov\Gamma\to \ov\Gamma'$ is fixed, any $\cKc_{\Gamma}\sub \cKc(\ov\Gamma,\beta)$ is mapped to a unique such $\cKc_{\Gamma'}$. Thus, we can reduce the claim further to showing that
	\begin{equation}\label{dig:reduced-twice}\begin{tikzcd}
			\djun{\substack{\Gamma\\ \ov{f}(\Gamma) = \Gamma'}}\cKc_{\Gamma}|_{\{t_e = 0\}} \arrow[r,""] \arrow[d,""]&\cKc_{\Gamma'} \arrow[d,""]\\ 
			I^{E(\ov\Gamma)}\times\Mbar_{\ov\Gamma}|_{t_e = 0} \arrow[r,""] & I^{E(\ov\Gamma')}\times\Mbar_{\ov\Gamma'}\end{tikzcd} \end{equation}
	is Cartesian up to strata of codimension at least $1$. For this, let $\cU\sub \cKc_{\Gamma'}$ be the subchart whose thickening is given by curves whose stabilisation is of type `at most $\ov\Gamma$' and let $\cV\sub \djun{}\cKc_{\Gamma}|_{\{t_e = 0\}}$ be the preimage of $\cU$ under the canonical clutching map. Then, 
	\[\cU\,=\, \djun{[f]}\,\cU_{[f]},\] 
	where $f \cl \Gamma\to \Gamma'$ is a retraction covering $\ov{f}$ and $[f]= [\wh{f}]$ if there exists an isomorphism $\phi\cl \Gamma\to \wh\Gamma$ with $\wh{f}\g \phi = f$. Since the underlying prestable graph of $\Gamma$ and $\wh\Gamma$ is simply $\ov\Gamma$, we must have $\phi\in \Aut(\ov\Gamma)_{\ov{f}}$. Thus, either $\phi$ is the identity, or $\phi(\Gamma)\in \scP\scS(\Gamma)\sm\{\Gamma\}$ and the graph appears in the indexing set of the left top corner of~\eqref{dig:reduced-twice}. Thus, the collection of trees we consider classifies the top-dimensional (possibly non-connected) strata of $\cU$. The remainder of the proof (of the stable case) is similar to the proof of \cite[Lemma~7.3]{HS22}. By construction of the cubical cobordisms, it suffices to show that the square
	\begin{equation}\label{eq:square-with-thickenings}
		\begin{tikzcd}
			\djun{\Gamma}\cTc_{\Gamma,\cV}|_{\{t_e = 0\}} \arrow[r,""] \arrow[d,""]&\cTc_\cU \arrow[d,""]\\ 
			I^{E(\ov\Gamma)}\times\Mbar_{\ov\Gamma}|_{\{t_e = 0\}}\arrow[r,"\psi_e"] & I^{E(\ov\Gamma')}\times\Mbar_{\ov\Gamma'} \end{tikzcd}		
	\end{equation}
	of orbifolds is Cartesian up to strata of codimension $1$, where $\wh{\cTc}_{\cV}$ is the locus of curves $(\varphi,u,\eta)$ in the thickening of $\cV$ with no perturbation from $\cE_\Gamma$ and where $\lambda_\Gamma(\varphi_{\Gamma},u) = 0$. This orbifold can be identified with an orbifold with the same base space as $\cTc_{\Gamma'}$ and thus also $\cTc_\cU$. Note that the assumption on $\cU$ ensures that the restriction $\cTc_{\cU}\to I^{E(\ov\Gamma')}\times\Mbar_{\ov\Gamma'}$ is a submersion. Let $\cW$ be the fibre product of the square. There exists a canonical map $\wh{\cTc}_\cV\to \cW$. We can define the inverse morphism $\cW\to \cU$ in the category of orbifolds by 
	\[((C_{v},\fj_v,D_v)_{v\in V(\Gamma)},\phi,(\varphi,u,\eta))\,\to\,(\wh\varphi,\wh u,\wh \eta),\]
	where $\phi =(\phi_{v'})_{v'\in V(\Gamma')}$ is a collection of automorphisms of the clutched surfaces $C'_{v'}$ obtained from $(C_v)_v$, and  
	\[\wh\varphi_{v} = (\varphi_{f(v)}\g \phi_{f(v)})|_{C_v}\qquad \quad\wh{u}_v = (u_{f(v)}\g \phi_{f(v)})|_{C_v}\qquad \quad\wh\eta_v = (\eta_{f(v)}\g \phi_{f(v)})|_{C_v}\]
	for $v \in V(\Gamma)$. By~\eqref{eq:almost-agreeing-perturbation} and \eqref{eq:cubical-perturbation} and the reduction we made just below~\eqref{eq:square-with-thickenings}, this is a well-defined map. Note that the composition $\cU\to \cW\to \cU$ does not necessarily map $\cU\cap \cKc_\Gamma$ to $\cKc_\Gamma$, but possibly to $\cKc_{\wh\Gamma}$, where $\wh\Gamma$ and $\Gamma$ are related by an automorphism of $\ov\Gamma$ preserving the contraction $f \cl \ov\Gamma\to \ov\Gamma'$. In particular, the image of $(\varphi,u,\eta)$ under $\cU \to \cW\to \cU$ is related to $(\varphi,u,\eta)$ exactly by an automorphism of the stabilisation of its domain.\par 
	Suppose now ${\Gamma'}^{ps}$ is an unstable graph with stabilisation $\ov\Gamma'$. Let $E' \sub E({\Gamma'}^{ps})$ be the subset of edges that are either combined or vanish under the stabilisation procedure. By the observation about the top-dimensional strata of $\cU$ in the previous step, it suffices to show the claim restricted to $\cK({\Gamma'}^{ps},\beta)|_{{t_E' = 0}}$ with a similar restriction on the left hand side. Then, we can  expand~\eqref{dig:contraction-square} to 
	\begin{equation}\label{dig:contraction-square-expanded}
		\begin{tikzcd}
			&\djun{\substack{\stb(\Gamma^{ps}) = \ov\Gamma\\ \stb\inv(e) = e\\\Gamma^{ps}_e \cong_+ {\Gamma'}^{ps}}} \cK(\Gamma^{ps},\beta)^{\sqcup A_{\Gamma^{ps}}}|_{\{t_{e\sqcup E'} = 0\}} \arrow[r] \arrow[d] & \cK({\Gamma'}^{ps},\beta)^{\sqcup A_{{\Gamma'}^{ps}}}|_{\{t_{E'} = 0\}} \arrow[d]\\
			&\cK(\cc \Gamma,\beta)|_{\{t_e = 0\}} \arrow[r] \arrow[d] & \cK(\cc \Gamma',\beta)\arrow[d]\\
			&I^{E(\cc \Gamma)}|_{\{t_e = 0\}}\times \Mbar_{\cc \Gamma}  \arrow[r] & I^{E(\cc \Gamma')} \times \Mbar_{\cc \Gamma'}  
		\end{tikzcd}
	\end{equation}
	We observe first that $\Aut({\Gamma'}^{ps}/\ov\Gamma')$ can at most permute chains of unstable vertices terminating in leaves of ${\Gamma'}^{ps}$ (and attached to the same stable vertex); any internal chain of unstable vertices is connected to two stable vertices that cannot be permuted by any such automorphism. In particular, these are exactly the automorphisms that fix the contraction map ${\Gamma'}^{ps}\to \Gamma'_{E'}$ (the choice of which is irrelevant). The same is true for the top vertical map on the left. Thus, we have reduced the claim to determining the intersection of boundary strata in $\cK(\Gamma',\beta)$. If two boundary strata do not agree, their intersection is a stratum of codimension $2$. In the case where they agree, we may conclude by Remark~\ref{rem:additional-exponent}. 
	If ${\Gamma'}^{ps}$ has internal unstable chains, we can reduce as above to the case where it has no external unstable chains.
\end{proof}

\begin{lemma}\label{lem:forgetful-map-fibre-product}
	Suppose $\ov\Gamma$ is a stable graph obtained from $\ov\Gamma'$ by forgetting an incoming marked point. Then, the square 
	\begin{equation}\label{dig:forgetful-map-square}
		\begin{tikzcd}
			\cK(\ov\Gamma',\beta) \arrow[r] \arrow[d] & \cK(\ov{\Gamma},\beta)\arrow[d]\\
			I^{E(\cc \Gamma')}\times \Mbar_{\cc \Gamma'} \arrow[r] &I^{E(\cc \Gamma)}\times\Mbar_{\cc \Gamma} 
		\end{tikzcd}
	\end{equation}
	is a Cartesian square of derived orbifolds up to codimension $1$ strata for any total degree $\beta$.
\end{lemma}

\begin{proof}
	Since we stabilise incoming marked points if and only if the associated vertex has positive degree or negative Euler characteristic, there is no difference between which marked points of $\cK(\Gamma',\beta)$ and of $\cK(\Gamma,\beta)$ are stabilised. Thus, the argument is a straightforward adaption of the proof of Lemma~\ref{lem:gkc-fibre-product}. 
\end{proof}

\appendix
\section{Moduli spaces of regular open stable maps}\label{sec:gluing}
\addtocontents{toc}{\protect\setcounter{tocdepth}{1}}

\subsection{Representability} Suppose $(X,\omega)$ is a symplectic manifold and $L \sub X$ be a Lagrangian submanifold. Fix an $\omega$-tame almost structure $J$ and $\beta\in H_2(X,L;\bZ)$. Note that we do not assume $X$ or $L$ to be compact. Fix a family $\pi\cl \cC\to \cV$ of nodal curves with $h$ (ordered) boundary components and a linear map 
$$P \cl E\to C^\infty_c(\cC\inn\times X,\Omega^{0,1}_{\cC\inn/\cV}\boxtimes_\bC TX)$$
defined on a finite-dimensional vector space, where $\cC\inn \sub \cC$ is the complement of the critical points of $\pi$ and the vertical boundary $\del^v\cC$. Given these data, we define the functor 
$$\fM := \fM^J_P(\pi,\beta)\cl (C^\infty/\cdot)^{\text{op}}\to \Set$$
by letting $\fM(Y/S)$ be the set of diagrams
\begin{equation}\label{eq:functor-de}\begin{tikzcd}
		\cC\arrow[d,"\pi"] & \cC_S\arrow[d]\arrow[l] &\cC_Y\arrow[d]\arrow[l]\arrow[r,"F"] & X\\
		\cV& S\arrow[l,"\varphi"]&Y\arrow[l,]\arrow[r,"w"] & E\end{tikzcd} \end{equation}
where 
\begin{enumerate}[\normalfont i)]
	\item $\varphi$ is continuous,
	\item both squares are cartesian,
	\item $F\cl \cC_Y/\cC_S \to X/*$ is of class rel--$C^\infty$ with $F(\del \cC_Y) \sub L$,
	\item for each $y \in Y$, the restriction $u_y := F|_{C_y}$ is a smooth stable map with boundary on $L$ representing $\beta$, where $C_y$ is the fibre of $\cC_Y\to Y$ over $y$. It satisfies 
	\begin{equation}\label{eq:perturbed-pde-representability}\delbar_J u_y  + P(w(y)) = 0 \end{equation}
	and the (perturbed) linearised Cauchy-Riemann operator at $u_y$ is surjective.
\end{enumerate}

\begin{remark}\label{} Here $C^\infty/\cdot$ is the category of rel--$C^\infty$ manifolds, defined originally in \cite{Swa21}. We briefly recall the relevant notions in \textsection\ref{subsec:rel-smooth-manifolds} and apologise for the non-chronological ordering of the appendices.\end{remark}

\noindent
In order to show this representability result, we adapt the strategy of \cite{Swa21} to our setting. A key ingredient is the following representability criterion.

\begin{proposition}\cite[Proposition 2.16]{Swa21}\label{prop:rep-criterion}  Suppose $X/S$ is an $S$-space and $\fF\cl (C^\infty/\cdot)^{\text{op}}\to \Set$ a functor with the following properties.
	\begin{enumerate}[1),leftmargin=*]
		\item\label{rep:base} $S$ represents the functor $\fF_{\text{base}}(T) = \fF(\emst/T)$.
		\item\label{rep:top} There exists an open cover $\{V_i\}\iI$ of $S$ so that $p\inv(V_i)$ represents the functor $$(\fF|_{V_i})^{\text{top}}: Y\mapsto \fF|_{V_i}(Y/Y).$$ 
		\item\label{rep:locally} $\fF$ is representable near each $x \in X$, that is, for each $x\in X$ there exists a neighbourhood $V \sub S$ of $p(x)$ and $U\sub p\inv(V)$ of $x$ such that $\fF^U|_V$ is representable.
	\end{enumerate}
	In this case, $X/S$ admits a rel--$C^\infty$ structure with which it represents $\fF$.
\end{proposition}

\begin{remark}
	Property \eqref{rep:top} implies that $X$ represents $\fF^{\text{top}}$ by (the non-relative version of) \cite[Lemma 2.15(b)(1)]{Swa21}
\end{remark}

\noindent
Our candidate to represent $\fM$ is given by 
$$\Mbar_{g,h}^J(\pi,\beta)^{\reg} := \set{(v,u,e)\mid v \in V,\; u \cl (C_v,\del C_v) \to  (X,L)\;\delbar_J u + P(e)|_{\cC_v} = 0,\, \eqref{eq:perturbed-pde-representability} \text{ regular}}$$
considered as a $\cV$-space via the obvious forgetful map $\Mbar_{g,h}^J(\pi,\beta)^{\reg}\to \cV$. We endow it with the Gromov--Hausdorff metric on the associated graphs in $\cC\times X$.

\begin{theorem}\label{thm:representable} $\fM$ is representable by $\Mbar_{g,h}^J(\pi,\beta)^{\reg}/\cV$.
\end{theorem}

\noindent
The first two conditions of Proposition \ref{prop:rep-criterion} follow from the same argument as in \cite[Lemma 3.15, Lemma 3.16]{Swa21}. It remains to show local representability. The arguments will be very similar to \cite{Swa21}, the only difference is that we have to deal with boundary nodes. 

\begin{remark}\label{} It might seem strange that, given the existence of codimension $1$ strata, we can represent $\Mbar_{g,h}^J(\pi,\beta)^{\reg}$ as a rel--$C^\infty$ manifold over $\cV$. The reason for this is that the corners come from deformations of the domain and not of the map. Indeed, our thickening will be a rel--$C^\infty$ manifold over a base space that \emph{is} a manifold with corners.\end{remark}

\subsection{Local models and universal deformations}\label{subsec:deformation-of-curve}
In this section we discuss universal deformations of nodal curves with boundary. The classical model for an interior node of a nodal surface is given by 
$$\cN_i = \{z \in \bC^2 \mid z_1z_2 = 0\}$$
with deformation given by 
$$\wt\cN_i = \{(t,z)\in \bC\times\bC^2\mid z_1z_2 = t\}.$$
In the case of  boundary nodes, we have several cases to distinguish, depending on the type of boundary node. Recall that we have the following types:
\begin{itemize}[leftmargin=20pt]
	\item (E) a boundary circle collapses to a point - it is forgotten when passing to the normalisation
	\item (H1) two boundary components intersect - the nodes appears as two points on a boundary circle of the normalised surface
	\item (H2) a boundary circle self-intersects and the resulting node is non-separating - the node corresponds to two points on two \emph{distinct} boundary circles of \emph{one} component of the normalisation.
	\item (H3) a boundary circle self-intersects and the node is separating - it corresponds to two points on two \emph{distinct} boundary circles of two \emph{distinct} components of the normalisation.
\end{itemize}

\noindent
The different symbol for the first singularity follows from the fact that it is an elliptic singularity, while the other singularities are hyperbolic and are locally indistinguishable. We have the local models
$$\cN_E = \{z \in \bC^2 \mid z_1^2 + z_2^2  = 0,\, \Im(z_1) \geq 0,\,\Im(z_2)\geq 0\}.$$
with boundary given by $\cN_E \cap \bR^2$ and with deformation given by 
$$\wt\cN_E = \{(t,z) \in \bR_{\geq 0}\times\bC^2 \mid z_1^2 + z_2^2  = t,\, \Im(z_1) \geq 0,\,\Im(z_2)\geq 0\}$$
mapping to $\bR_{\geq 0}$ by projection onto the first factor. Note how the boundary over the fibre $t > 0$ is simply the circle of radius $\sqrt{t}$. If $t < 0$, a crosscap appears:
\begin{center}
	\includegraphics[scale=0.7]{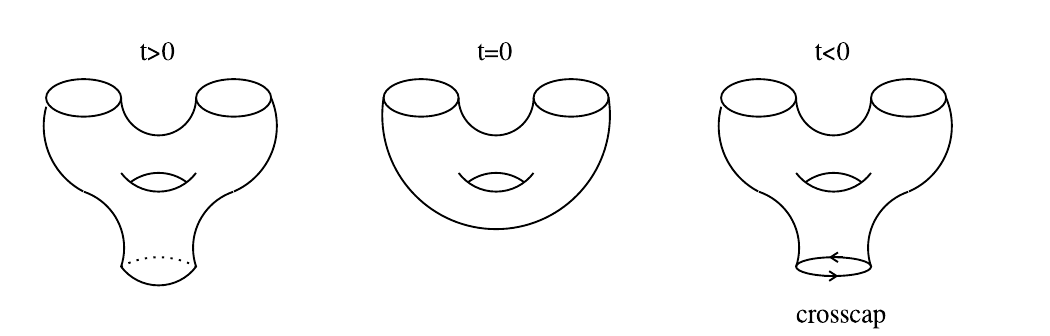}
\end{center}
Meanwhile, the local model for a boundary node of type $(H1)$, $(H2)$, or $(H3)$ is
$$\cN_{H} = \{z \in \bC^2 \mid z_1^2-z_2^2 = 0,\, \Im(z_1) \geq 0,\,\Im(z_2)\geq 0\},$$  
whose deformation is given by 
$$\wt\cN_{H} = \{(t,z) \in \bR_{\geq 0}\times\bC^2 \mid z_1^2 - z_2^2  = t,\, \Im(z_1) \geq 0,\,\Im(z_2)\geq 0\}.$$
In this case, passing to $t< 0$, a M\"obius strip appears:

\medskip

\begin{center}
	\includegraphics[scale=0.7]{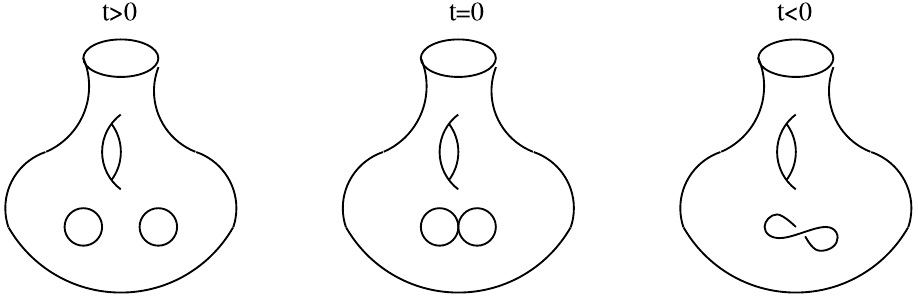}
\end{center}

\medskip

\noindent
Suppose $(C,\del C)$ is a stable nodal Riemann surface of genus $g$ with $h$ ordered boundary components. Denote by $N_i$ and $N_b = N_b^E\sqcup N_b^H$ the set of interior as well as (elliptic and hyperbolic) boundary nodes. For $x \in N_i$ fix two cylindrical ends 
$$\psi^\pm_x \cl [0,\infty)\times S^1 \to C$$
and for $y \in N_b^H$ fix two strip-like ends:
$$\vartheta^\pm_y \cl [0,\infty)\times[0,1]\to C$$
with $\vartheta_y^\pm([0,\infty)\times \{1\})\sub \del C$. Finally, for each elliptic node $y \in N^E_b$ we fix one cylindrical end 
$$\vartheta_y \cl [0,\infty)\times S^1 \to C.$$
We choose these ends so that they have disjoint images. We denote by the same symbol the obvious continuous extension to $[0,\infty]\times S^1$ respectively $[0,\infty]\times [0,1]$.
\subsection{Pre-gluing and anti-gluing}\label{subsec:gluing-set-up}
 Fix thus $(v_0,u_0,e_0) \in \Mbar_{g,h}^J(\pi,\beta)^{\reg}$ with $(C_0,\fj_0)$ the fibre over $v_0$ with its induced complex structure. Let $N_i$ be the set of interior nodes of $C_0$ and $N_i = N_i^E \sqcup N_i^H$ be the set of boundary nodes decomposed further into elliptic and hyperbolic nodes.\par 
 
 Fix cylindrical ends as in \textsection\ref{subsec:deformation-of-curve}. 
 Let $\fj_0 $ denote the complex structure on $C_0$ and let $V$ be a sufficiently large finite-dimensional vector space and $\{\fj_v = \fj(v)\}_{v\in V}$ be a smooth family of complex structures on $C_0$ so that $\fj(0) = j_0$ and $\fj_v$ agrees with $\fj_0$ in a neighbourhood of the images of the cylindrical and strip like ends as well as the boundary.
 
 Given $\alpha \in \bB^{N_i}$ with $\alpha_x = e^{-R_x-\text{i}\theta_x}$ and $\beta \in [0,1)^{N_b}$ with $\beta_i = e^{-T_i}$, we define the curve $C_{\alpha,\beta}$ to be the quotient of the disjoint union 
 $$ \wh C_0 \sqcup \djun{\substack{x \in N_i}{*\in \{\pm\}}}{\psi_x^*([0,R_x]\times S^1)}\sqcup \djun{\substack{y \in N_b^H\\ *\in \{\pm\}}}{\psi_y^*([0, 2 T_y]\times [0,1])}\sqcup \djun{y \in N_b^E}{\psi_y(A_{|\beta_y|})}$$
modulo the relations 
$$\psi_x^+(s,t) \sim \psi_x^-(s',t') \qquad \dimp \qquad \begin{cases}
	s + s' =  _x < \infty, \; t + t' = \theta_x\\
	s = s' = R_x = \infty \end{cases}$$
for $x \in N_i$, while for $y \in N_b^H$ we have 
$$ \psi_y^+(s,t) \sim \psi_y^-(s',t') \qquad \dimp \qquad t = t' \text{ and } \begin{cases} s + s' =  2 T_y < \infty,\\
	s = s' = T_y = \infty\end{cases}.$$
Note that we impose no relations on the cylindrical ends at elliptic boundary nodes.\par
As the cylindrical and strip like ends are holomorphic and by our assumption on the family $\{\fj_v\}_{v\in V}$, we see that each $\fj_v$ descends to $C_{\alpha\beta}$ for any $\alpha$ and $\beta$. We denote the induced complex structure by $\fj_v$ as well.\par 
 Now define $\scV := V\times \bB^{N_i}\times [0,1)^{N_b}$ and let $\uppi\cl \scC\to \scV$ be the unique family of curves with fibre over $(v,\alpha,\beta)$ being $(C_{\alpha\beta},\fj_v)$. It is a straightforward verification that $\scC\to \scV$ is a family of prestable curves with boundary.

 \begin{lemma}\label{lem:reduction-to-versal-deformation} We can choose $V$ above and a map 
 	\begin{equation*}P' \cl  E\to C^\infty_c(\scC\inn\times X,\Omega^{0,1}_{\scC\inn/\scV}\boxtimes_\bC TX) \end{equation*}
 	 so that $\fM^{J}_{P'}(\wt \pi)_\beta$ is representable near $(0,0,u_0,e_0)$ if and only if $\fM^{J}_P(\pi,\beta)$ is representable near the point $(v_0,u_0,\wt e_0)$. 
 \end{lemma}

\begin{proof} The proof of \cite[Lemma~4.1]{Swa21} carries over verbatim.
\end{proof}

\noindent
We assume from now on that $\cV= \scV$ and that $\cC = \scC$ is given by the family of curves constructed above. 
 Fix a Riemannian metric $g$ on $X$ so that $L$ is totally geodesic with respect to $g$ and $g$ is flat near $u_0(N_i \cup N_b)$. The existence of a metric with the first property is asserted by \cite[Lemma~4.3.4]{MS12}. Note that the standard metric in the chosen local coordinates of the lemma satisfies Equation (4.3.4) loc. cit., so we may choose this metric to be flat near the given points. Moreover, fix open subsets $\{U_p\}_{p\in u_0(N_i)}$ and $\{U'_q\}_{q\in u_0(N_b)}$ with pairwise disjoint closure so that 
 \begin{itemize}
 	\item $p\in U_p$ and $q\in U'_q$ and
 	\item $(U_p,g)\cong (B_1(\bR^{2m}),g_{\normalfont\text{std}})$ and $(U'_q,L\cap U'_q,g)\cong (B_1(\bR^{2m}),\bR^n\times\{0\}\cap B_1(\bR^{2m}),g_{\normalfont\text{std}})$,
 	\item the image of $u_0\g \psi_x^\pm$ is contained in $U_{u_0(x)}$ and similarly for $u_0\g \psi^\pm_y$ and $u_0\g \psi_y$,
 \end{itemize}
 for $p \in u_0(N_i)$ and $q\in u_0(N_b)$, respectively, where these isometries map the respective image of the nodal point to $0$. Write 
 \begin{equation*} \phi \cl \djun{p\in u(N_i\sqcup N_b)}U_p \to \bR^{2n}\end{equation*}
 for the induced map. Let $\conn$ be a $J$-linear connection on $TX$ that preserves $TL$. Given $x,x' \in X$ sufficiently close to each other, we denote by 
 \begin{equation*}\label{} \Phi_{x\to x'} \cl T_xX\to T_x X'\end{equation*} the $J$-linear parallel transport along the unique geodesic with respect to $g$ from $x$ to $x'$. By assumption on $g$ and $\conn$, $\Phi_{x\to x'}$ maps $T_xL$ to $T_{x'}L$ if $x,x'\in L$.\par
 
  Fix also two cut-off functions $\chi_1,\chi_2 \cl \bR\to [0,1]$ so that 
 \begin{equation*}\label{}\chi_1(s) = 
 	\begin{cases}
 		1 \quad & s \leq -1\\
 		0 \quad & s \geq 1
 \end{cases}\end{equation*}
we have $-1 < \chi_1' \leq 0$ and $\chi_1^2+ \chi_2^2 \equiv 1$.

\begin{definition}\label{de:glued-spaces}
	Given gluing parameters $\alpha$ and $\beta$, let 
	$$Z_{\alpha_x} := \frac{\bR^+\times S^1 \sqcup \bR^-\times S^1}{(s,t)\sim (2R_x-s,\theta_x-t)},$$\noindent
	with the relation holding for $s \in [0,2R_x]$ if $\alpha_x \neq 0$. Similarly, let 
	$$\Theta_{\beta_y} := \frac{\bR^+\times [0,1] \sqcup \bR^-\times [0,1]}{(s,t)\sim (2T_y-s,t)}$$\noindent
	with the relation applying to $s \in [0,2T_y]$ if $\beta_y \neq 0$ and $y$ is hyperbolic and 
	$$\Theta_{\beta_y} :=  [0,\infty)\times S^1$$\noindent
	when $y$ is elliptic. 
	We define $C_{\alpha\beta}$ as above, and set
	$$\wt Z_{\alpha\beta}:=  \p{\alpha_x \neq 0}{Z_{\alpha_x}}\times\p{\beta_y \neq 0}{\Theta_{\beta_y}},$$\noindent
	if $(\alpha,\beta)\neq 0$ and $ \wt Z_{00} = \emst$.
\end{definition}

\begin{remark}\label{}The coordinates $[s,t]$ on $\bR_+\times S^1$ extend to $Z_{\alpha_x}$ via $s' = s-2R_x$ and $t' = \theta_x-t$ on the other component and similarly for $\Theta_y$.
 \end{remark}

\begin{notation*} Given a continuous vector field $\xi \cl C\to u^*TX$, we denote for $w$ an interior or hyperbolic boundary node and $y$ an elliptic boundary node the vector fields
$$\xi^\pm_w := d\phi(u(\psi_w^\pm))\xi(\psi_w^\pm) \qquad \qquad \xi_y := d\phi(u(\psi_y))\xi(\psi_y).$$\noindent
We define the induced maps $u_w^\pm$ and $u_y$ similarly, while for $\eta \cl \scC\to \Lambda^{0,1}_{\scC,\fj_0}\otimes u^*TX$, we set
$$\eta_w^\pm(s,t) := d\phi(u(w))\Phi_{u^\pm_w\to u(w)}\eta_{\psi_x^+(s,t)}(\del_s)\qquad \qquad \eta_y(s,t) := d\phi(u(y))\Phi_{u^\pm_w\to u(w)}\eta_{\psi_y(s,t)}(\del_s).$$\noindent
\end{notation*}

\begin{definition}[Pre- and anti-gluing of vector fields]\label{}
	Given gluing parameters $\alpha$ and $\beta$ and $\xi \in C^0(C_0,u_0^*(TX,TL))$, define the \emph{pre-glued vector field} $\oplus_{\alpha\beta} \xi$ on $C_{\alpha\beta}$ as follows: it agrees with $\xi$ on away from the necks of $C_{\alpha\beta}$ that correspond to nodes of $C_0$, while over an interior neck associated to $x$ we have in the chosen trivialisation
	\begin{equation*}\label{}\oplus_{\alpha\beta} \xi(s,t) = 
		\chi_1(s-R_x)\xi^+_x(s,t)+\chi_2(s-R_x)\xi^-_x(2R_x-s,\theta_x-t), \end{equation*}
	for $s \in [0,2R_x]$, while over a boundary strip corresponding to the node $y$,
	\begin{equation*}\label{}\oplus_{\alpha\beta} \xi(s,t) = 
		\chi_1(s-T_y)\xi^+_y(s,t)+\chi_2(s-T_y)\xi^-_y(2T_y-s,t), \end{equation*}
	for $s \in [0,2T_y]$, while over an elliptic disc
	\begin{equation*}\label{}\oplus_{\alpha\beta} \xi(s,t) = \chi_1(s-T_y)\xi_y(s,t).\end{equation*}
	The \emph{anti-glued vector field} $\ominus_{\alpha\beta}\xi$ is the product of maps $Z_{\alpha_x}\to TX$ and $\Theta_{\beta_y}\to TX$ given by 
	\begin{equation*}\label{} [s,t]\mapsto 
		-\chi_2(R_x-s)(\xi^+_x(s,t)-[\xi]_{R_x}) + \chi_1(s-R_x)(\xi^-_x(2R_x-s,\theta_x-t)-[\xi]_{R_x}) 
	\end{equation*}
	where 
	$$[\xi]_{R_x} := \frac{1}{2}\lbr{\int_{S^1}\xi^+_x(R_x,\cdot) +\xi^-_x(R_x,\cdot)\,dt},$$\noindent
	respectively, 
	\begin{equation*}\label{} [s,t] \mapsto 
		-\chi_2(s-T_y)(\xi^+_y(s,t)-[\xi]_{T_y}^\bR) + \chi_1(s-T_y)(\xi^-_y(2T_y-s,t)-[\xi]_{T_y}^\bR) 
	\end{equation*}
	where 
	$$[\xi]_{T_y}^\bR := \frac{1}{2}\,\Re\lbr{\int_{0}^1\xi^+_y(T_y,\cdot) +\xi^-_y(T_y,\cdot)\,dt}.$$\noindent
	For an elliptic boundary node $y$ with $\beta_y \neq 0$, the map $\Theta_{\beta_y}\to TX$ is given by
	$$(s,t)\mapsto \xi_y(s,t)-\int_{S^1}\xi_y(T_y,\cdot)\,dt.$$
\end{definition}

\begin{definition}[Pre- and anti-gluing of maps]\label{}
	Given gluing parameters $\alpha$ and $\beta$ and a function $v\cl C\to X$, the \emph{pre-glued map} 
	$$\oplus_{\alpha\beta} v \cl C_{\alpha\beta}\to X$$\noindent 
	agrees with $v$ on away from the necks of $C_{\alpha\beta}$ that correspond to nodes of $C_0$, while over an interior neck corresponding to $x \in N_i$ we have
	\begin{equation*}\label{}\oplus_{\alpha\beta} v(z) = 
			\exp_{v(x)}((\chi_1(s-R_x)+\chi_2(s-2R_x))\exp_{v(x)}\inv(v(z))) \end{equation*}
		for $z = \psi^+_x(s,t)$ and $s \in [0,2R_x]$.
	The definition of $\oplus_{\alpha\beta} v$ over a hyperbolic ``boundary neck" is analogous. Over an elliptic ``boundary disc", we define 
	\begin{equation*}\label{}\oplus_{\alpha\beta} v(z) = \exp_{v(y)}(\chi_1(s-T_y)\exp_{v(y)}\inv(v(z)))\end{equation*}
	for $z = \psi_y(s,t)$ with $s \in [T_y-1,2T_y]$.\par
	The \emph{anti-glued map} is $\ominus_{\alpha\beta}v$ is the product of maps $Z_{\alpha_x}\to X$ and $\Theta_{\beta_y}\to X$ given by 
	\begin{equation}\label{} [s,t] \mapsto 
		\exp_{v(x)}(-\chi_2(s-R_x)\xi_x^+(s,t) + \chi_1(s-R_x)\xi_x^-(2R_x-s,\theta_x-t)) 
	\end{equation}
where $\xi = \exp_{v(x)}\inv(v(\cdot))$. The definition of $\ominus_{\alpha\beta}$ is analogous for a hyperbolic boundary node, while for an elliptic boundary node $y$ the map $\Theta_{\beta_y}\to X$ is simply $v_y$.
\end{definition}

\noindent
Note that $$\oplus_{\alpha\beta} u(\del C_{\alpha\beta})\sub L\qquad \text{and}\qquad\ominus_{\alpha\beta}u(\del Z_{\alpha\beta})\sub \p{\alpha_x,\beta_y\neq 0}{L}$$\noindent 
due to our choice of Riemannian metric and connection. The same is true for pre-glued vector fields.

\begin{remark}\label{}By construction, the exponential map intertwines the two constructions in the sense that 
$$\exp_{\oplus_{\alpha\beta} u_0}(\oplus_{\alpha\beta}\xi) = \oplus_{\alpha\beta} \exp_{u_0}(\xi)$$\noindent
and similarly for $\ominus_{\alpha\beta}$. 
\end{remark}

\begin{notation*}In the coming subsections, we abbreviate $u := u_0$ and $u_{\alpha\beta} := \oplus_{\alpha\beta}u$.\end{notation*}

\subsection{Fredholm set-up} 
We now define the Sobolev spaces we work with and show that the pre-gluing and anti-gluing map together define an isomorphism from sections of $u^*TX$ to those of $u_{\alpha\beta}^*TX$. See \cite{Jem20} for a similar exposition.\par
Given $k \gg 1$ (e.g. $k \geq 6$) and $\delta \in (0,\frac12)$, as well as $\alpha\in \bB^{N_i}$ and $\beta \in [0,1)^{N_b}$ close to $0$, we define the weighted Sobolev spaces 
$$W^{k,2,\delta}_{TL}(C_{\alpha\beta},u_{\alpha\beta}^*TX)\sub \{\xi \in C^0(C_{\alpha\beta},u_{\alpha\beta}^*TX)\mid \xi(\del C_{\alpha\beta})\sub u_{\alpha\beta}^*TL\}$$\noindent
by the usual Sobolev norm on $\wh C_0 \sub C_{\alpha\beta}$ and by 
\begin{equation}\label{eq:norm-interior-end} \norm{\xi}_{k,2,\delta}^2 = \norm{\xi(x)}^2+\int_{\bR_+\times S^1}\lbr{\norm{\xi-\xi(x)}^2 + \sum_{j =1}^{k}\norm{D^k\xi}^2}e^{2\delta s}dsdt  \end{equation}
over each cylindrical end asymptotic to an interior node or an elliptic boundary node $x$, and by 
\begin{equation}\label{eq:norm-boundary-end}  \norm{\xi}_{k,2,\delta}^2 = \norm{\xi(y)}^2+\int_{\bR_+\times [0,1]}\lbr{\norm{\xi-\xi(y)}^2 + \sum_{j =1}^{k}\norm{D^k\xi}^2}e^{2\delta s}dsdt \end{equation}
over a strip-like end asymptotic to a boundary node $y\in N_b$. Over an interior neck, we define 
\begin{equation}\label{eq:norm-interior-neck}  \norm{\xi}_{k,2,\delta}^2 = \norm{[\xi]_R}^2 + \int_{[0,2R_x]\times S^1}\lbr{\norm{\xi-[\xi]_{R_x}}^2 + \sum_{j =1}^{k}\norm{D^k\xi}^2}e^{2\delta \min\{s,2R_x-s\}}dsdt\end{equation}
while over a hyberbolic boundary neck we have 
\begin{equation}\label{eq:norm-boundary-neck}  \norm{\xi}_{k,2,\delta}^2 = \norm{[\xi]_{T_y}^\bR}^2 + \int_{[0,2T_y]\times [0,1]}\lbr{\norm{\xi-[\xi]_{T_y}^\bR}^2 + \sum_{j =1}^{k}\norm{D^k\xi}^2}e^{2\delta \min\{s,2T_y-s\}}dsdt\end{equation}
and over an elliptic boundary neck,
\begin{equation}\label{eq:norm-elliptic-boundary-neck}  \norm{\xi}_{k,2,\delta}^2 = \norm{\int_{S^1}\xi(\psi_y(T_y,\cdot))\,dt}^2 + \int_{[0,2T_y]\times S^1}\lbr{\norm{\xi-[\xi]_{T_y}^\bR}^2 + \sum_{j =1}^{k}\norm{D^k\xi}^2}e^{2\delta \min\{s,2T_y-s\}}dsdt\end{equation}
Here we use the trivialisations chosen in \textsection\ref{subsec:gluing-set-up} to make sense of $\xi-\xi(x)$, respectively, $\xi-[\xi]_R$.

We define the weighted Sobolev norm on 
$$W^{k-1,2,\delta}_{TL}(C_{\alpha\beta},\Omega^{0,1}_{\scC_{\alpha\beta},\fj_v}\boxtimes u_{\alpha\beta}^*TX)\sub \{\eta \in C^0(C_{\alpha\beta},\Omega^{0,1}_{\scC_{\alpha\beta},\fj_v}\boxtimes u_{\alpha\beta}^*TX)\mid \eta(T\del C_{\alpha\beta})\sub u_{\alpha\beta}^*TL\}$$\noindent
by the usual $(k,2)$-norm on $\wh C_0$. Over the ends and necks, we trivialise using the chosen biholomorphisms and charts on $X$ and evaluate at $\del_s$, so that we can consider the $(0,1)$-forms as functions valued in $\bR^{2n}$. We then define the norms by the same expressions as \eqref{eq:norm-interior-end}-\eqref{eq:norm-elliptic-boundary-neck} except that we use $\norm{\eta}$ instead of $\norm{\eta-\eta(w)}$ or $\norm{\eta-[\eta]}$. 

Define for $\alpha_x \neq 0$, the Sobolev spaces
$$W_{\alpha_x} := W^{k,2,\delta}(Z_{\alpha_x},\bC^n) \qquad \wh W_{\alpha_x} := W^{k-1,2,\delta}(Z_{\alpha_x},\Omega^{0,1}\otimes\bC^n)$$\noindent
with the norms
\begin{multline}\label{eq:cylinder-norm} \norm{\xi}_{k,2,\delta}^2 =\int_{\bR_+\times S^1}\lbr{\norm{\xi-\xi(\infty)}^2 + \sum_{j=1}^{k}\norm{D^j\xi}^2}e^{2\delta s}dsdt \\+ \int_{\bR_-\times S^1}\lbr{\norm{\xi-\xi(-\infty)}^2 + \sum_{j=1}^{k}\norm{D^j\xi}}^2e^{2\delta|s'|}ds'dt\end{multline}
and 
\begin{equation}\label{eq:cylinder-norm-01} \norm{\eta}_{k-1,2,\delta}^2 =\int_{\bR\times S^1}\sum_{j=0}^{k-1}\norm{D^j\eta}^2e^{2\delta|s|}dsdt\end{equation}
Meanwhile, for $\beta_y \neq 0$, define
$$W_{\beta_y} := W^{k,2,\delta}_{\bR^n}(\Theta_{\beta_y},\bC^n) \qquad \wh W_{\beta_y} := W^{k-1,2,\delta}_{\bR^n}(\Theta_{\beta_y},\Omega^{0,1}\otimes\bC^n)$$\noindent
with the analogous of the norms \eqref{eq:cylinder-norm} and \eqref{eq:cylinder-norm-01}. Let 
$$\cV_{\alpha\beta} \sub \p{\alpha_x\neq 0}{W_{\alpha_x}}\times \p{\beta_y\neq 0}{W_{\beta_y}}$$\noindent
by the subspace of elements with $\xi_w(\infty) = -\xi_w(-\infty)$ for an interior or a hyberbolic boundary node $w$, and define 
$$\wh\cV_{\alpha\beta} :=  \p{\alpha_x\neq 0}{\wh W_{\alpha_x}}\times \p{\beta_y\neq 0}{\wh W_{\beta_y}}.$$\noindent

\begin{proposition}\label{prop:comparing-smoothings} The complex-linear maps
	\begin{equation}\label{eq:total-gluing}\boxdot_{\alpha\beta}\cl W^{k,2,\delta}_{TL}(C_0,u^*TX)\to W^{k,2,\delta}_{TL}(C_{\alpha\beta},\oplus_{\alpha\beta}u^*TX)\oplus \cV_{\alpha\beta} :\xi \mapsto (\oplus_{\alpha\beta}\xi,\ominus_{\alpha\beta}\xi) \end{equation}
	 and 
	 $$\wh\boxdot_{\alpha\beta}\cl W^{k-1,2,\delta}_{TL}(C_0,\Lambda^1_{\scC,\fj_0}\otimes\, u^*TX)\to W^{k-1,2,\delta}_{TL}(\scC_{\alpha\beta},\Lambda^1_{\scC_{\alpha\beta},\fj_0}\otimes\,\oplus_{\alpha\beta}u^*TX)\oplus \wh\cV_{\alpha\beta}$$\noindent
	 given by $$\wh\boxdot_{\alpha\beta}\eta = (\wh\oplus_{\alpha\beta}\eta,\wh{\ominus}_{\alpha\beta}\eta),$$ \noindent
	are isomorphisms and there exists a constant $c_0 $ independent of $\alpha$ and $\beta$, for $\max\{|\alpha|,|\beta|\}\leq \frac14$, so that 
	\begin{equation}\label{eq:bounded-comparison}\frac1{c_0}\norm{\xi}_{k,2,\delta} \leq \norm{\boxdot_{\alpha\beta}\xi}_{k,2,\delta}\leq c_0\norm{\xi}_{k,2,\delta} \end{equation}
	and similarly for $\wh\boxdot_{\alpha\beta}$.
\end{proposition}	

\begin{proof} This follows from \cite[Theorem~2.23 and Theorem~2.26]{HWZ10} adapted to case with boundary as in \cite[Proposition~12.0.1]{Jem20} (with $m = 0$). In \cite{Jem20}, $E$ and $F$ denote the scale Banach spaces with $$\cX_m  =  W^{k+m,2,\delta_m}_{TL}(C_0,u^*TX)\qquad \qquad \cY_m  =  W^{k-1+m,2,\delta_m}_{TL}(C_0,\Lambda^{0,1}_{C_0,\fj_0}\otimes u^*TX),$$\noindent 
	of which we consider only the case of $m = 0$. Meanwhile $G^a$ loc. cit. corresponds to the right hand side of \eqref{eq:total-gluing}. There are two more relevant differences: over the necks of $C_{\alpha\beta}$, they use a different norms than \eqref{eq:norm-interior-neck}-\eqref{eq:norm-elliptic-boundary-neck}, which, however, are commensurate to the respective norm used here. Moreover, they use a single cut-off function $\beta$ with the same property as $\chi_1$ and take $1-\beta$ instead of requiring $\chi_1^2+ \chi_2^2 \equiv 1$ as we do. The proof of the estimate, \cite[Lemma~2.27]{HWZ10}, used in both statements, goes through verbatim, observing additionally that over $[0,2R]\times S^1$, respectively $[0,2R]\times[0,1]$, the norm $\normd_{k,2,-\delta}$ is commensurate with $\normd_{k,2,\delta}$.
\end{proof}

\noindent
Given $v \in V$ sufficiently close to the origin, the map 
$$I(v) := \frac12(\Ide-\fj_v\fj_0) \cl (T\wt C,\fj_0)\to (T\wt C,\fj_v)$$\noindent
is an isomorphism which is the identity near the preimages of nodes of $C$. Given $\alpha$ and $\beta$ close to $0$, $I(v)$ thus descends to an isomorphism $I_{\alpha\beta}(v)\cl(T\wt C_{\alpha\beta},\fj_0)\to (T\wt C_{\alpha\beta},\fj_v)$.

\subsection{Cauchy-Riemann equation}\label{subsec:gluing-estimates}
Set $$\cX := W^{k,2,\delta}_{TL}(C_0,u^*TX)\qquad \qquad \cY:= W^{k-1,2,\delta}_{TL}(\wt C,\Omega^{0,1}_{\wt C}\otimes \wt{u}^*TX).$$\noindent 
Define 
$$\cF\cl V\times \bB^{N_x}\times [0,1)^{N_b}\times \cX\times E\to \cY$$\noindent
by 
\begin{multline}\label{eq:cr-map} \cF_{v,\alpha\beta}(\xi,e) :=\\\, \wh\boxdot_{\alpha\beta}\inv\left[(\Phi_{\exp_{u_{\alpha\beta}}(\oplus_{\alpha\beta}\xi)\to u_{\alpha\beta}}\otimes I_{\alpha\beta}(v))\lbr{ d(\exp_{u_{\alpha\beta}}(\oplus_{\alpha\beta}\xi))^{0,1}_{J,\fJ_v}+ P(e_0+e)(\alpha,\beta,\cdot,\exp_{u_{\alpha\beta}}(\oplus_{\alpha\beta}\xi))},\delbar\ominus_{\alpha\beta}\xi\right].\end{multline}
 $\cF$ is only defined near the origin and we will implicitly always take $(v,\alpha,\beta,\xi,e)$ close to zero. The following result is an adaptation of the gluing analysis in \cite{HWZ17}, see also \cite{Jem20}. However, we do not use notions from scale calculus beyond the next statement.\par
  Recall that an \emph{$sc$-Banach space} $\cZ$ is a sequence $(\cZ_m)_{m\geq 0}$ of Banach spaces together with compact embeddings $\cZ_{m+1}\hkra \cZ_m$ so that $\cZ_\infty := \inter{m\geq 0}{\cZ_m}$ is dense in $\cZ_0$. If $U \sub \cZ$ is an open subset, we call a map $f \cl U\to \cZ'$ to another $sc$-Banach space \emph{of class $sc^0$} if $f(U\cap \cZ_m)\sub\cZ_m$ for any $m\geq 0$ and $f$ is continuous. We say $f$ is \emph{$sc^1$} if there exists a fibrewise-linear ${sc}^0$ map 
  $$Tf \cl TU = U\times \cZ_1\to T\cZ' : (x,h)\mapsto (f(x),df(x)h)$$\noindent
  so that 
  $$\limes{\norm{h}_1}{0}{\frac{1}{\norm{h}_1}\norm{f(x+h)-f(x)-Tf(x,h)}} =0$$\noindent
  for any $x \in U$ and $h \in \cZ_1$. We call $f$ \emph{of class $\text{sc}^n$} for $n \geq 1$ if $Tf$ is of class $sc^{n-1}$ and we call it $sc$-smooth if it is of class $sc^n$ for any $n$.\par 
  If $U = B\times U' \sub W\times \cZ$, where $ W$ is a finite-dimensional manifold, possibly with boundary, and $\cZ$ is an $sc$-Banach space, we say $f \cl U \to \cZ'$ is \emph{$sc$-smooth relative to $B$} if the map $f_b := f(b,\cdot)$ is $sc$-smooth for any $b\in B$ and $b \mapsto T^kf_b$ is continuous for any $k \geq 0$.

\begin{theorem}\label{prop:cr-sc-smooth} $\cF$ is 
	$sc$-smooth relative to $V\times \bB^{N_x}\times[0,1)^{N_i}$ with respect to the filtrations 
	$$\cX_m := W^{k+m,2,\delta_m}_{TL}(C_0,u^*TX) \qquad \qquad \cY_m := W^{k-1,2,\delta}_{TL}(\wt C,\Omega^{0,1}_{\wt C}\otimes \wt{u}^*TX)$$\noindent
	where $(\delta_m)_{m\geq 0}$ is any strictly increasing sequence with $\delta = \delta_0 < \delta_m  < \frac12$.
\end{theorem}

\begin{proof} This is a generalisation of \cite[Proposition~4.21]{HWZ17} to the case with boundary and the same proof goes through, with certain ingredients replaced by the corresponding statement for bordered Riemann surfaces. We will summarise the proof and point out the changes.\par
	Away from the cylindrical or strip-like ends, the proof of (relatives) $sc$-smoothness follows from standard elliptic theory, as well as the fact that the maps involved are the differential operator $d$ and pre- or post-composition with maps defined on smooth closed manifolds, respectively their tangent bundles. See also the discussion below \cite[Equation~(4.16)]{HWZ17}.\par 
	Near interior or elliptic boundary nodes, the proof of \cite[Proposition~4.21]{HWZ17} goes through verbatim since we defined the gluing and anti-gluing over these nodes in the same way save for the domain. In particular, the proof of \cite[Proposition~4.11]{HWZ17} goes through. We note here that we use slightly different cut-off functions, which does not affect the relevant estimates. We use different norms over the necks, which, however, are commensurate with those of \cite{HWZ17}. Finally, the anti-gluing defined here differs from the anti-gluing of \cite{HWZ17} in that we precompose with the parallel transport $\Phi_{u \to u(w)}$ for a node $w$. This invertible operator is independent of the element of the domain of $\cF$ and hence does not affect any estimate apart from increasing certain constants. Finally, the proof of \cite[Proposition~4.11]{HWZ17} for gluing and anti-gluing at an elliptic boundary node is in the same vein, in fact, even easier.
\end{proof}

\subsection{Gluing map}\label{subsec:gluing-map} 
We need a relative version of (a special case of) \cite[Theorem~4.6]{HWZ09}.

\begin{proposition}[Relative $sc$-implicit function theorem]\label{prop:sc-implicit-function} Let $\cZ$ and $\cZ'$ be $sc$-Banach spaces and $T$ a finite-dimensional manifold, possibly with corners. Suppose $f \cl T\times \cZ\to \cZ'$ is an $sc$-smooth Fredholm map relative to $T$, so that for some $t_0 \in T$ and $z_0 \in f_{t_0}\inv(\{0\})$, the linearisation $df_{t_0}(z_0)\cl T_{z_0}\cZ\to T_{f_{t_0}(z_0)} \cZ'$ is surjective and there exists a continuous family $M$ of $sc$-complements to $\ker(df_t(z_0))$ for $t$ near $t_0$. Then $S := f\inv(\{0\})$ admits a \emph{good parametrisation} near $(t_0,z_0)$.\par That is, there exists neighbourhoods $W\sub T\times \ker(df_{t_0}(z_0))$ and $U\sub T\times E$ near $\{(t_0,0)\}$, respectively $(t_0,z_0)$, as well as an $sc$-smooth map
	$$\Psi \cl W\to M,$$\noindent 
	compatible with the projection to $T$, so that
	\begin{enumerate}[1),leftmargin=20pt,ref=\arabic*]
		\item\label{i:origin} $\Psi(t,0) = 0$ and $d\Psi_t(0) = 0$;
		\item\label{i:graph} the map $$\Gamma\cl W\to \cZ : w \mapsto (t,z_0 + w + \Psi_t(w))$$\noindent 
		maps onto $S\cap U$ and is compatible with the projection to $T$,
		\item\label{i:kernels} if $(t,z) = \Gamma(t,w)\in S\cap U$, then the map $\ker(df_{t_0}(z_0))\to \ker(df_t(z))$ given by 
		$$v\mapsto v + d\Psi_t(w)v$$\noindent
		is a linear isomorphism.
	\end{enumerate}
\end{proposition}

\begin{proof} The claim in the case of $T = \pt$ given by \cite[Theorem~4.6]{HWZ09}. The general case follows from a straightforward adaptation to the relative setting; refer e.g. to the proof of the relative inverse function theorem, \cite[Lemma~5.10]{Swa21}. Denote by $\Psi(\cdot)_\cX$, respectively $\Psi(\cdot)_\cX$ the composition of $\Psi$ with the projection onto the respective vector space.
\end{proof}

\begin{lemma}\label{} There exists a linear subspace $E'\sub E$ and finitely many points $q_1,\dots,q_m \in \wh C_0$ and subspaces $V_i\sub T_{u(q_i)}X$ so that the map 
	
	\begin{align*}\label{} L \cl \cX\oplus E&\to E/E'\oplus \bigoplus\limits_{i = 1}^m \modulo{T_{u(q_i)}X}{V_i} \\ (\xi,e) &\mapsto (e + E')\oplus (\xi(q_i)+ V_i)_{i = 1,\dots,m}\end{align*}
	restricts to an isomorphism on $K_{v,\alpha\beta}$ for any $v,\alpha,\beta$ close to $0$. 
\end{lemma}

\noindent
Refer to \cite[Lemma~4.10]{Swa21} for a proof, which carries over verbatim to our setting.\par

Define 
\begin{equation}\label{eq:right-inverse} Q_{v,\alpha\beta}\cl \cY\to \cX : \eta \mapsto (D_{v,\alpha\beta}|_{\ker(L)})\inv(\eta). \end{equation}
By Theorem \ref{prop:cr-sc-smooth}, $Q_{v,\alpha\beta}$ is strongly continuous in $v, \alpha$ and $\beta$. Taking $(t_0,z_0) = (0,0,e_0)$, Proposition \ref{prop:sc-implicit-function} asserts the existence of maps $\Psi$ and $\Gamma$ with the properties \eqref{i:origin}-\eqref{i:kernels}.

\begin{definition}\label{de:gluing-amp} 
	We define the \emph{gluing map} 
	$$\varphi = \varphi_{u_0} \cl U\sub V\times \bB_{\epsilon}^{N_i}\times [0,\epsilon)^{N_b} \times K_{0,0}\to \Mbar^J(\pi,X,L)$$\noindent
	by 
	$$\varphi(v,\alpha,\beta,\xi,e) = \lbr{v,\alpha,\beta,\oplus_{\alpha\beta}\exp_{ u}(\xi+\Psi_{v,\alpha\beta}(\xi,e)_\cX),\Psi_{v,\alpha\beta}(\xi,e)_E},$$\noindent
	where $U$ is a sufficiently small neighbourhood of the origin. 
\end{definition}

\noindent
Since $K_{0,0}$ is finite-dimensional, the strong operator topology on bounded operators $K_{0,0} \to \cX$ agrees with the topology induced by the operator norm, so $\Psi$ defines a relatively smooth map with respect to the operator norm. 

\begin{lemma}\label{lem:gluing-injective} $\varphi$ is injective near the origin.
\end{lemma}

\begin{proof} Suppose there exists sequences $(v_j,\alpha_j,\beta_j,\xi_j,e_j)\jN$ and $(v'_j,\alpha'_j,\beta'_j,\xi'_j,e_j')\jN$ converging to $0$ so that 
	$$(v_j,\alpha_j,\beta_j,\xi_j,e_j)\neq (v'_j,\alpha'_j,\beta'_j,\xi'_j,e_j')\qquad \qquad \varphi(v_j,\alpha_j,\beta_j,\xi_j) = \varphi(v'_j,\alpha'_j,\beta'_j,\xi'_j,e_j')$$\noindent 
	for each $j\geq 1$. This implies $(v_j,\alpha_j,\beta_j) = (v'_j,\alpha'_j,\beta'_j)$ as well as  
	$$\Psi_{v_j,\alpha_j,\beta_j}(\xi_j,e_j)_E = \Psi_{v_j,\alpha_j,\beta_j}(\xi_j,e_j)_E$$\noindent
	 and
	$$\exp_{u_{\alpha_j\beta_j}}(\oplus_{\alpha\beta}(\xi_j+\Psi_{v_j,\alpha_j,\beta_j}(\xi_j,e_j)_\cX)) = \exp_{u_{\alpha_j\beta_j}}(\oplus_{\alpha\beta}(\xi_j'+\Psi_{v_j,\alpha_j,\beta_j}(\xi'_j,e_j')_\cX)).$$\noindent
	 By the Sobolev embedding theorem, we have $$\norm{\oplus_{\alpha\beta}\xi_j+\Psi_{v_j,\alpha_j,\beta_j}(\xi_j,e_j)_\cX}_{L^\infty} \leq\normalfont\text{injr}(X,g)$$ for $j \gg 1$ and similarly for $\xi'_j$, whence
	 $$\oplus_{\alpha\beta}\xi_j+\Psi_{v_j,\alpha_j,\beta_j}(\xi_j,e_j) =\oplus_{\alpha\beta}\xi'_j+ \Psi_{v_j,\alpha_j,\beta_j}(\xi'_j,e_j').$$\noindent
	 As $\im(\Psi_{v,\alpha,\beta})\sub \ker(L)$ and $L$ is injective on $K_{0,0}$, the claim follows by applying $L$.
\end{proof}

\begin{lemma}\label{lem:gluing-surjective} $\varphi$ surjects onto a neighbourhood of $(0,0,0,u_0,e_0)$. 
\end{lemma}

\begin{proof} This follows from the same arguments as in \cite[Proposition~B.11.5]{P16}. He only has the interior gluing parameters $\alpha$, which does not affect the arguments, and $\kappa$ instead of $\xi$. The expression $\oplus_{\alpha\beta}(\xi +\Psi_{v,\alpha\beta}(\xi,e)_\cX)$ corresponds to his $\kappa^\infty_{v,\alpha}(\xi,e)$.
\end{proof}

\noindent
By \cite[Lemma~B.12.1]{P16} the map $\varphi$, possibly restricted to an even smaller neighbourhood of the origin, is a homeomorphism onto a neighbourhood of $(0,0,0,u,e_0)$. 
\subsection{Rel--$C^\infty$ structure}

This part of the proof is exactly as in \cite[\textsection4.4.]{Swa21} and only reproduced here for the sake of completeness. Shrinking the domain of $\varphi_u$ further, we may assume  
$$U_u = V_u \times \bB_r^{N_i}\times [0,r)^{N_b}\times K_u$$
where $V_u \sub V$ and $K_u \sub K$ are small neighbourhoods of the respective origin. Let 
$$\pi_u \cl U_u \to \cS_u := V_u \times \bB_r^{N_i}\times [0,r)^{N_b}$$ 
be the projection and let $\cC_u \to \cS_u$ be the restriction of the universal curve. Define 
$$\eva_u \cl \cC_u \times K_u \to X : ((v,\alpha,\beta,z),(\xi,e))\mapsto \oplus_{\alpha\beta}\exp_u(\xi + \Psi_{v,\alpha\beta}(\xi,e)_\cX)(z).$$
By construction, $\eva_u(\delbar \cC_u\times K_u)\sub L$.\par
Let $W_u \sub \Mbar_{P}^J(\pi,\beta)^{\reg}$ be the open subset given by the image of $\varphi_u$. Recall that the topological space $\Mbar_{P}^J(\pi,\beta)^{\reg}$ represents the functor 
$$\fM^{\normalfont\text{top}} : Y\mapsto \fM(Y/Y)$$ 
(by the arguments of \cite[\textsection3.2.2]{Swa21}). We define the functor $\fM^{W_u}$ by
$$\fM^{W_u}(Y/T) = \set{\chi \in \fM(Y/T)\mid \chi^{\normalfont\text{top}}\in \fM^{\normalfont\text{top}}(Y) \text{ has image contained in }W_u}.$$ 

\begin{proposition}[Local representability]\label{prop:local-representability} The commutative diagram
		\begin{equation}\label{} \begin{tikzcd}
			\cC  \arrow[d,"\pi"]&\cC\arrow[l,"="] \arrow[d,""] & \cC_u\times K_u\arrow[r,"\eva_u"]\arrow[l,""]\arrow[d,"\pi_u\times\ide"]  & X\\ 
			\cS & \cS \arrow[l,"="] & U_u\arrow[l,"\pi_u"]\arrow[r,"w"] & E \end{tikzcd} \end{equation} 
		defines an element $\wt\varphi_u \in \fM(U_u/\cS_u)$ lifting $\varphi_u \in \fM^{\normalfont\text{top}}(U_u)$. The associated natural transformation $U_u/\cS_u \to \fM$ induces an isomorphism $U_u/\cS_u \xra{\sim}\fM^{W_u}$. 
\end{proposition}

\begin{proof} The first claim follows as soon as we have shown that $\eva_u\cl \cC_u\times K_u/\cC$ is relatively smooth. This can be seen from the same arguments as in the proof of \cite[Theorem~4.14]{Swa21}, whose notation we follow closely. The second assertion is our analogue of \cite[Theorem~4.16]{Swa21}, whose proof carries over verbatim. Note that in both cases, we do not have the additional perturbation space $E$ since we consider honest pseudo-holomorphic curves, while \cite{Swa21} allows for perturbations of the Cauchy-Riemann equation.\end{proof}

This result relieves us of having to show that the transition functions between the local charts given by $\varphi_u\cl U_u \to W_u$ are relatively smooth with respect to $\cS$. It completes thus the proof of Theorem \ref{thm:representable}.
\section{Orientations}\label{sec:orientations}
We discuss orientations of global Kuranishi charts and summarise how to orient the index bundles of Cauchy--Riemann operators over a bordered surface with relatively spin totally real boundary condition.

\subsection{Orientations of global Kuranishi charts}\label{subesc:gkc-conventions}
We use the following conventions for the orientation of quotients and vector bundles.

\begin{definition}\label{de:or-quotient} Suppose a Lie group $G$ acts almost freely, properly, and locally linearly on an oriented topological manifold $M$. We orient the quotient $M/G$ (over $\bQ$) by pulling back the orientation $(-1)^{(\dim(M)-\dim(G))\dim(G)}\fo_M$ along the homeomorphism \eqref{eq:slice-homeo}.
\end{definition}

\begin{definition}\label{de:or-vector-bundle} Given an oriented orbi-bundle $\pi \cl W\to Y$ over an oriented orbifold, we orient the total space $W$ via the Thom isomorphism $H^{\dim(W)}(W;\bQ)\cong H^{\dim(Y)}(Y;\bQ)\otimes \bQ\lspan{\tau_W}$. 
\end{definition}

\begin{remark}\label{rem:or-smooth-case} If $M$ and $Y$ are smooth, then this can be said more concisely. That is, letting $q$ be the quotient map, we orient $T(M/G)$ and $TW$ via the isomorphisms 
	$$TM\cong q^*T(M/G)\oplus \fg\qquad \text{and}\qquad TW\cong \pi^*TY\oplus \pi^*W.$$
\end{remark}

\begin{definition}\label{de:orientation-gkc} For a rel--$C^\infty$ global Kuranishi chart $\cK = (G,\cT/\cB,\cE,\obs)$ for $Z$, we define its \emph{determinant} 
	\begin{equation}\label{de:determinant-gck}\det(\cK) := \det(T\cT)\otimes \det(\fg)\dul \otimes\det(\cE)\dul,\end{equation}
	which we consider as a vector bundle germ near $\obs\inv(0)$. An \emph{orientation} of $\cK$ is a section of its \emph{orientation line} $(\det(\cK)\sm \{0\})/\bR_{> 0}$.
\end{definition}

\noindent
As was already observed in \cite{AMS21}, a global Kuranishi chart might be orientable even though its thickening and obstruction bundle are not. However, in this case one can always stabilise by the obstruction bundle to obtain a global Kuranishi chart whose thickening and obstruction bundle are orientable. 

\begin{definition}\label{de:product-orientation} If $\cK_i$ is a global Kuranishi chart for $i = \{0,1\}$, equipped with an orientation $\fo_{\cK_i}$, then we orient the \emph{product global Kuranishi chart} 
	$$\cK_0 \times\cK_1 = (G_0\times G_1,\cT_0\times\cT_1,\cE_0\boxplus \cE_1,\obs_0\boxtimes\obs_1)$$\noindent
	via the \emph{product orientation} $\fo_{\cK_0}\times\fo_{\cK_1} := \fo_{\cK_0}\wedge \fo_{\cK_1}$. 
\end{definition}

\begin{remark}\label{} Given orientations of thickening, covering group and obstruction bundle, another natural choice of orientation of $\cK_0\times\cK_1$ would have been 
	$$(\fo_{\cK_0}\times\fo_{\cK_1})' := \fo_{\cT_0}\wedge \fo_{\cT_1}\wedge \fo_{\fg_1}\dul \wedge \fo_{\fg_0}\dul \wedge \fo_{\cE_1}\dul \wedge \fo_{\cE_0}\dul.$$\noindent 
	The advantage of Definition \ref{de:product-orientation} is that it does not require a choice of the components of the global Kuranishi charts and the property $\fo_{\cK_1}\times \fo_{\cK_0} = (-1)^{\vdim(\cK_0)\vdim(\cK_1)}\fo_{\cK_0}\wedge \fo_{\cK_1}$. \end{remark}

\noindent
More generally, we define the fibre product of global Kuranishi charts and its orientation as follows, recovering  Definition \ref{de:orientation-fibre-product} in case both obstruction bundles are trivial.

\begin{definition}\label{de:fibre-product-gkc} Suppose $\cK_0$ and $\cK_1$ are oriented global Kuranishi charts, each equipped with a $G_i$-invariant rel--$C^\infty$ submersion $f_i \cl \cK_i \to Y$ to a smooth oriented manifold $Y$. We equip the \emph{fibre product global Kuranishi chart} 
	$$\cK_0 \times_Y\cK_1 := (G_0\times G_1,\cT_0\times_Y\cT_1,\cE_0\times \cE_1|_{\cT_0\times_Y\cT_1},\fs_0\times\fs_1|_{\cT_0\times_Y\cT_1})$$\noindent
	with the unique orientation $\fo_{\cK_0\times_L\cK_1}$ so that the canonical isomorphism 
	$$\det(\cK_0\times_Y\cK_1)\otimes \det(TY) \xra{\cong} \det(\cK_0)\otimes \det(\cK_1)$$\noindent
	maps it to $(-1)^{\vdim(\cK_1)\dim(Y)}\fo_{\cK_0}\times\fo_{\cK_1}$.
\end{definition}

\noindent
For the next lemma, we think of the thickenings as the orbifolds $\cT/G$ as this is how we will use it in the main application, and thus just write $\fo_{\cT}$ instead of $\fo_\cT\wedge \fo_\fg\dul$.

\begin{lemma}\label{lem:fibre-product-orientations} Let $\cK_0$ and $\cK_1$ be as in Definition \ref{de:fibre-product-gkc} with submersion to $Y$. Then the obstruction bundle of $\cK_0 \times_Y\cK_1$ is given by $\cE_0 \boxplus\cE_1$. Suppose $\cE_0$ and $\cE_1$ are orientable and endow their direct sum with the direct sum orientation. Then, write $\wt\fo_{\cT_0\times_Y\cT_1}$ for the induced orientation of the thickening and $\fo_{\cT_0\times_Y\cT_1}$ for the standard fibre product orientation, we obtain that 
	$$\wt\fo_{\cT_0\times_Y\cT_1} =(-1)^{\rank(\cE_0)(\dim(Y) +\vdim(\cK_1))} \fo_{\cT_0\times_Y\cT_1}. $$
\end{lemma}

\begin{proof} Let $n = \dim(Y)$ and $m_i = \rank(\cE_i)$. We have 
	$$\wt\fo_{\cT_0\times_Y\cT_1}\wedge( \fo_{\cE_0}\wedge \fo_{\cE_1})\dul\wedge\fo_Y= (-1)^{\vdim(\cK_1)n}\fo_{\cT_0}\wedge \fo_{\cE_0}\dul\wedge \fo_{\cT_1}\wedge \fo_{\cE_1}\dul$$\noindent
	whence 
	
	\begin{align*}(-1)^{n(m_0+m_1) }\wt\fo_{\cT_0\times_Y\cT_1}\wedge\fo_Y &= (-1)^{\vdim(\cK_1)n+ m_0\vdim(\cK_1)} \fo_{\cT_0}\wedge \fo_{\cT_1} \\&=(-1)^{\vdim(\cK_1)n+ m_0\vdim(\cK_1)+n\dim(\cT_1)} \fo_{\cT_0\times_Y\cT_1}\wedge \fo_Y. 
	\end{align*}
\end{proof}

\begin{lemma}\label{lem:permuting-fibre-products} Suppose we have transverse fibre products $X_i\times_{Z_i}Y_i$ for $i = \{0,1\}$ and $Y_1 \times_WY_2$ of orbifolds so that $(X_1\times_{Z_1}Y_1)\times_W(X_2\times_{Z_2}Y_2)$ is transverse as well. If all orbifolds are oriented, then 
	\begin{gather*}\label{} (X_1\times_{Z_1}Y_1)\times_W(X_2\times_{Z_2}Y_2) = (-1)^{\epsilon}(X_1\times_W X_2)\times_{Z_1\times Z_2} (Y_1\times_W Y_2)\\
		\text{with }\quad \epsilon =(\dim Y_1 -\dim Z_1)(\dim X_2-\dim Z_2 -\dim W)\end{gather*}
\end{lemma}

\noindent
The same result holds for fibre products of global Kuranishi charts, replacing dimension by virtual dimension.

\begin{proof} As the question is pointwise, we may assume that all orbifolds are manifolds (in fact, vector spaces) and $X_i = X'_i \times Z_i$ and $X'_1 = X''_1\times W$ and $X'_2 = W\times X''_2$. Then 
	$$X_i \times_{Z_i} Y_i \cong X'_i \times Y_i\qquad\qquad (X''_1\times W)\times_W (W\times X''_2) = X''_1\times W\times X''_2.$$\noindent
	Thus, we have 
	
	\begin{align*}(X_1\times_{Z_1}Y_1)&\times_W(X_2\times_{Z_2}Y_2) =  (X''_1\times W\times Y_1)\times_W(W\times X''_2\times Y_2)
		\\&= (-1)^{\dim Y_1 \dim X'_1} Y_1 \times (X''_1\times W\times X''_2)\times Y_2 
		\\&=(-1)^{\dim Y_1\dim X''_2} (X''_1\times W\times X''_2)\times(Y_1\times Y_2)
		\\&=(-1)^{\dim Y_1\dim X''_2} (X''_1\times W\times X''_2)\times Z_1\times Z_2\times_{Z_1\times Z_2}(Y_1\times Y_2)
		\\&=(-1)^{\dim Y_1\dim X''_2+\dim Z_1\dim X'_2} ((X''_1\times Z_1\times W\times X''_2\times Z_2)\times_{Z_0\times Z_1}(Y_1\times Y_2)
		\\&  =(-1)^{\dim Y_1\dim X''_2+\dim Z_1\dim X'_2} ((X''_1\times Z_1\times W) \times_W X_2)\times_{Z_1\times Z_2} (Y_1\times Y_2)
		\\&=(-1)^{\dim Y_1\dim X''_2+\dim Z_1\dim X''_2}(X_1\times_W X_2)\times_{Z_1\times Z_2} (Y_1\times Y_2)
	\end{align*}
	so that 
	\begin{equation*}\epsilon \equiv(\dim Y_1 +\dim Z_1)(\dim X_2-\dim Z_2 -\dim W).\end{equation*}
\end{proof}

\subsection{Cauchy--Riemann operators on bordered surfaces}\label{subsec:orientation-construction} This subsection is a brief summary of the definition of \cite{CZ24}, describing how a relative spin structure defines an orientation on Cauchy-Riemann problem on a surface with boundary. Proposition \ref{prop:orientation-for-H1-node} is the analogue of \cite[Theorem~4.3.3(c)]{WW17} using the constructions and conventions of \cite{CZ24}. Concretely, it describes how orientations compare as a surface develops a boundary node of type (H1). To maximise compatibility with \cite{CZ24} we identify a surface of boundary with the symmetric surface given by taking its complex double with the induced real structure.\footnote{As we consider any surface with boundary to be endowed with an ordering of its boundary components, this induces a decoration on its complex double in the sense of \cite{CZ24}.}\\

\noindent\textbf{Spin structures: }Recall that $\text{Spin}(n)$ is the universal cover of $\text{SO}(n)$ and that a spin structure on an oriented vector bundle $E\to B$ is a lift of the classifying map $B\to B\text{SO}(n)$ to $B\to B\text{Spin}(n)$. 
Since $\pi_2(B\text{Spin}(n))  =0$ and a nullhomotopy of $B\to B\text{SO}(n)$ determines a (homotopy class) of trivialisation, this definition can be rephrased if $B$ is a CW complex and $n \ge 3$. In this case, such a lift is equivalent to a choice of homotopy class of trivialisations of $E$ restricted to the $2$-skeleton of $B$, cf. \cite[Definition~1.2]{CZ24}. In particular, it yields a choice of homotopy class of trivialisation of $\gamma^*E$ for any map $\gamma\cl S^1\to B$. It is exactly this property of yielding `coherent \emph{choices} of trivialisations' that will be crucial for the construction of orientations of Cauchy-Riemann operators below.

\begin{definition}\label{de:ospin} An \emph{OSpin structure} $\fo\fs$ on a real vector bundle pair $(V,\varphi)$ over a symmetric surface $(C,\sigma)$ is an orientation $\fo$ on $V^\varphi$ together with a spin structure $\fs$ on $V^\varphi \to C^\sigma$ lifting $\fo$.
\end{definition}

\noindent\textbf{On a smooth surface:} Suppose $(V,\varphi)$ is a real vector bundle pair of rank $n$ over a decorated symmetric surface $(C,\sigma)$ equipped with an OSpin structure $(\fo,\fs)$. Suppose $D$ is a real Cauchy--Riemann operator on $(V,\varphi)$. Let $\{U_i\}_{i\leq r}$ be a set of $\sigma$-invariant tubular neighbourhoods of the boundary components so that $\cc{U_i}\cap \cc{{U_j}} = \emst$ if $i \neq j$ and $L|_{\cc{U_i}}$ is trivial. Let $\gamma_i^\pm := \del^\pm U_i$ and let $C_0$ be the surface obtained from $C$ by collapsing all loops $\gamma_i^\pm$. Let $\cC\to Q:= \bB^{2r} \sub \bC^{2r}$ be the associated family of deformations so that $C = \cC|_{\{\mathbf{t}_1\}}
$ and $C_0 =\cC|_{\{\mathbf{0}\}}$  and let $(\cV,\wt\varphi)$ be an extension of $(V,\varphi)$ over all of $\cC$. As the base of the deformation is connected, $(\fo,\fs)$ determines an OSpin structure on $(V_{\mathbf{t}},\varphi_{\mathbf{t}}) = (\cV,\wt\varphi)|_{\cC_{\mathbf{t}}}$ for each $\mathbf{t}\in Q$. It remains to associate to $\fo_{x_{\mathbf{0}}}$ an orientation on $C_{0}$. Pulling $(V_0,\wt\varphi_0)$ back to $\wt C_0$ and using the exact triple 
$$0 \to D_0 \to \wt D_0 \to \oplus \bigoplus\limits_{z \in \cN_{C_0}} (V_{0})_z\to 0$$\noindent
of Fredholm operators, it remains to fix an orientation of $\det(\wt D_0)$. The normalisation $\wt C_0$ is given by a disjoint union of a closed surface $\Sigma$ and several copies of the disc labelled by $\{1,\dots,r\}$ according to the given ordering of the boundary circles, each of which is attached to $\Sigma$ by an interior node $z_i$. This yields an isomorphism 
\begin{equation}\label{eq:orientation-line-normalisation}
	\det(D_0)\otimes \bigotimes\limits_{1 \leq i \leq r} (V_{0})_{z_i}\cong \bigotimes\limits_{1 \leq i \leq r}\det(\wt D_0|_{\bD_i})\otimes \det(\wt D_0|_\Sigma) 
\end{equation}
 We endow $\det(\wt D_0|_\Sigma)$ with the complex orientation if and only if $\lspan{w_{\fs_0}(L_0),[\Sigma]} = 0$. Meanwhile, we use the trivialisation of the real bundle pair over each disc induced by the spin structure to identify $\det(\wt D_0|_{\bD_i})$ with $\bR^n$, e.g., given by evaluating at a point on the boundary and equipping it with the canonical induced orientation. Transporting the induced orientation of $\det(D_0)$ back to $\cC_1$, we obtain the \emph{intrinsic orientation} of $\det(D_{(L,\phi)})$.
\bigskip

\noindent\textbf{On a nodal surface:} Let $\mathbf{C}$ be a decorated symmetric nodal Riemann surface and suppose $(V,\varphi)$ is a real vector bundle pair of rank $n$ over the underlying symmetric surface $(C,\sigma)$ equipped with an OSpin structure. Suppose $D_V$ is a real Cauchy-Riemann operator on $V$. Let $(\wt C,\wt\sigma)$ be its normalisation and let $(\wt V,\wt\varphi)$ and $D_{\wt V}$ be the pullback of $(V,\varphi)$ and $D_V$ to the normalisation. Then we obtain a short exact sequence 
$$0 \to D_V \to D_{\wt V}\to \bigoplus\limits_{x\in \cN^i_C}(V_\bC)_x\oplus \bigoplus\limits_{x\in \cN^b_C}V_x\to 0$$\noindent
of Fredholm operators, where $\cN^i_C$ and $\cN^b_C$ are the sets of interior and boundary nodes of $C$. The last map is induced by the evaluation map $\xi \mapsto (\xi(x_+)-\xi(x_-))_{x\in \cN_C}$. By \cite{Zi16}, we obtain a unique isomorphism 
$$\det(D_V) \otimes \bigotimes\limits_{x\in \cN^i_C}(V_\bC)_x\otimes \bigotimes\limits_{x\in \cN^b_C}V_x\cong \det(D_{\wt V})$$\noindent
so an orientation of $\det(D_{\wt V})$ and an orientation of $V_x$ for each $x \in \cN_x$ induces an orientation of $\det(D)$. If the orientation of $V$ and $\det(D_{\wt V})$ is induced by the pullback of an OSpin structure $\fo\fs$ on $(V,\varphi)$, we call the induced orientation $\fo_{\normalfont\text{int}}(V,\varphi;\fo\fs)$ of $\det(D)$ the \emph{intrinsic orientation} of $(V,\varphi)$ induced by $\fo\fs$.\par 
On the other hand, letting $\pi \cl \cC\to B$ be a flat deformation of $(C,\sigma)$ over a connected base $B$ and endowing the fibres of $\pi$ with the induced decoration, we can extend $(V,\varphi)$ to a real bundle pair $\cV$ over $(\cC,\wh\sigma)$. Let $\mathbf{C_1}$ be the fibre of $\pi$ over a regular value $b_1$ and let $(V_1,\varphi_1)$ be the restriction of $\cV$ to $C_1$. Any OSpin structure $\fo\fs$ on $(V,\varphi)$ extends to an OSpin structure on $\cV$, which in turn restricts to an OSpin structure $\fo\fs_1$ on $(V_1,\varphi_1)$. Thus, it induces an orientation on $\det(D_{V_1})$ via the construction for smooth surfaces. Using any path from $b_1$ to the image of $C$ under $\pi$, we obtain a homotopy of Fredholm operators and, in particular, an induced orientation on $\det(D_V)$. Following \cite{CZ24}, we call this orientation the \emph{limiting orientation} $\fo_{\lim}(V,\varphi;\fo\fs)$ of $D_V$.

\begin{proposition}\label{prop:orientation-for-H1-node} Suppose $\mathbf{C}$ is a decorated symmetric surface with a unique node $\normalfont\text{nd}$ of type (H1). Let $(V,\varphi)$ be a real bundle pair over $\mathbf{C}$ of rank $n$ equipped with a relative OSpin structure $\fo\fs$. Then, the intrinsic and the limiting orientation on $D_{(V,\varphi)}$ agree if and only if 
	\begin{equation}\label{} n(|\pi_0(C^\sigma)|-i) + \frac{n(n-1)}{2} \in 2\bZ,\end{equation} 
	where $i$ is the position of the singular component of $C^\sigma$.
\end{proposition}

\begin{proof} Changing the order of the boundary components, we may assume by \cite[CROrient~$1\fo\fs$(2)]{CZ24} that the singular boundary component comes last. Collapsing the boundary of a tubular neighbourhood of the singular component of $C$ and applying \cite[CROrient~7C(a)]{CZ24}, we may assume that $C$ is a cylinder with an (H1) node and that $(V,\varphi)$ is the trivial real bundle pair, with the trivialisation obtained from the OSpin structure $(\fo,\fs).$
	
There are two ways of proving the desired statement, either by adapting the construction of \cite[\textsection8.1]{CZ24} to the case of the cylinder using \cite[Theorem~C.1.10]{MS12} or by leveraging the results of \cite[\textsection 8]{CZ24}. We will pursue the latter approach. 

Let $(\cC\to[0,2],\wt \sigma)$ be a flat deformation of $C= C_0$ to a smooth cylinder $C_1 = \cC|_{\{1\}}$ and further to a cylinder $C_2$ with one conjugate (i.e., interior) node.
Extend the decoration of $(C,\sigma)$ and the real vector bundle pair $(V,\varphi)$ to a fibrewise decoration and a real vector bundle pair $(\cV,\wh\varphi)$ on $\cC$.
By \cite[CROrient~7C(a)]{CZ24}, the intrinsic orientation of $D_{V_2}$ agrees with the limiting orientation coming from the intrinsic orientation of $D_{V_1}$. Thus we may compare the intrinsic orientation of $D_{V_0}$ with the orientation induced by this homotopy and the intrinsic orientation of $D_{V_2}$. The normalisations $\wt C_0$ and $\wt C_2$ are a disc and a disjoint pair of discs $\bD_1\sqcup \bD_2$, respectively. Recall that the orientation on $\det(D_0)$ and $\det(D_2)$ is induced by the exact sequences of Fredholm operators 
	\begin{equation}\label{fr-ses-1}0 \to D_0 \to \wt D_0 \to V^\varphi\to 0,\end{equation}
	respectively, 
	\begin{equation}\label{fr-ses-2}0 \to D_2 \to \wt D_2 \to V\to 0,\end{equation}
	where we identify the finite-dimensional vector spaces with the Fredholm operator given by the map to the trivial vector space. Since the trivial Cauchy--Riemann operator on the disc is unobstructed, we have that $H^0(\wt D_0) \cong V^\varphi$ and $H^0(\wt D_2 ) = (V^\varphi)^{\oplus 2}$. It follows that~\eqref{fr-ses-1} yields the long exact sequence
	\begin{equation}\label{fr-ses-3}
		0 \to \ker(D_0)\to V^\varphi \xra{v\mapsto v-v} V^\varphi\to \coker(D_0)\to 0,
	\end{equation}	
	whence $\orl(D_0) = \orl(V^\varphi) \orl(V^\varphi)\dul \cong \orl(0)$. The isomorphism is orientation-preserving since the zero map $V^\varphi\to V^\varphi$ is homotopic to the identity. Meanwhile~\eqref{fr-ses-2} is isomorphic to 
	\begin{equation}\label{fr-ses-4}
		0 \to \ker(D_2)\to V^\varphi\oplus V^\varphi \xra{(v,w)\mapsto w-v} V \to V/V^\varphi\to 0,
	\end{equation}
	where we identify the vector bundles with the trivial vector space (using the spin structure). We can homotope the map $V^\varphi\oplus V^\varphi \xra{(v,w)\mapsto w-v} V$ to the map $(v,w)\mapsto w$. Then, using that the canonical orientation of $V$ as a complex vector space is given by $v_1,\ii v_1,v_2,\dots,\ii v_n$ for any basis $v_1,\dots,v_n$ of $V^\varphi$, we obtain that the induced isomorphism 
	\[\det(D_2)\oplus V \cong V^\varphi\oplus V^\varphi\g V/V^\varphi\]
	has orientation sign $(-1)^{\frac{n(n-1)}{2}}$ with respect to the canonical orientation of $V^\varphi=\det(D_2)$.
\end{proof}
	
\noindent
For later use, we record the following statement about perturbed Cauchy--Riemann operators. 

\begin{lemma}\cite[Lemma~3.2]{Bao23}\label{lem:orientation-of-sum} Suppose $D\cl V\to W$ is a Fredholm operator and $\nu \cl Z\to W$ is a linear map from a finite-dimensional vector space so that  $D+\nu$ is surjective. Then there exists a canonical isomorphism 
	$$\det(D+\nu) \cong \det(D)\otimes \topw Z$$\noindent
	given by 
	$$v\wedge (v_{k+1},z_1)\wedge \dots \wedge (v_n,z_{n-k})\mapsto v\otimes ([\nu(z_{n-k})]\wedge \dots \wedge[\nu(z_m)])\dul \otimes z_1\wedge \dots \wedge z_m$$\noindent
	for any $v \in\topw\ker(D)$ and any basis $z_1,\dots,z_m$ of $Z$, where we identify $v  =v_1 \wedge \dots v_k$ with the vector $(v_1,0)\wedge \dots \wedge (v_k,0)\in \Lambda^k\ker(D+\nu)$.\qed
\end{lemma}

\section{Rel--$C^\infty$ structures and relative smoothings}\label{sec:rel-smooth-forms}

\subsection{The category of rel--$C^\infty$ manifolds}\label{subsec:rel-smooth-manifolds} We recall the definition of a rel--$C^\infty$ manifold from \cite{Swa21}. We write $Y/S$ for the datum of a continuous map $p\cl Y\to S$, whenever $p$ is clear from the context.

\begin{definition}[Rel--$C^\infty$ maps]
	Given spaces $S$ and $T$ and integers $m,n\ge 0$, suppose $U\subset\bR^n\times S$ and $V\subset\bR^m\times T$ are open subsets. A \emph{rel--$C^\infty$ map} $\varphi = (\tilde\varphi,\bar{\varphi})\cl U/S\to V/T$ consists of a continuous map $\varphi\cl S\to T$ and a continuous map $\tilde\varphi\cl U\to V$ having the form
	\begin{align*}
		\tilde\varphi(x,s) = (F(x,s),\varphi(s)),
	\end{align*}
	where $F\cl U\to\bR^m$ is such that all its partial derivatives $D^\alpha_xF\cl U\to\bR^m$ with respect to $x\in\bR^n$ are defined and continuous as functions on $U$.
\end{definition}

\begin{definition}[Rel--$C^\infty$ manifolds]
	For a continuous map of spaces $p\cl Y\to S$ and an integer $n\ge 0$, a \emph{chart} of relative dimension $n$ is an open subset $U\subset Y$ together with an open embedding $\varphi\cl U\to \bR^n\times S$ which is compatible with the projection maps to $S$. Two charts $(U,\bar{\varphi})$ and $(V,\psi)$ of relative dimension $n$ for $Y/S$ are said to be \emph{rel--$C^\infty$ compatible} if
	\begin{align*}
		(\psi\circ\varphi^{-1},\text{id}_S)\cl \varphi(U\cap V)/S\to\psi(U\cap V)/S
	\end{align*}
	is a \emph{rel--$C^\infty$ diffeomorphism}, i.e., a rel--$C^\infty$ map with a rel--$C^\infty$ inverse.
	
	A \emph{rel--$C^\infty$ structure} on $Y/S$ is a maximal atlas of rel--$C^\infty$ compatible charts for $Y/S$. In this situation, we refer to $Y/S$ with this atlas as a \emph{rel--$C^\infty$ manifold} or we say that $Y$ has the structure of a rel--$C^\infty$ manifold over $S$.
\end{definition}

\begin{definition}	Given rel--$C^\infty$ manifolds $Y/S$ and $Z/T$, a \emph{rel--$C^\infty$ map $\varphi=(\tilde\varphi,\bar{\varphi})\cl Y/S\to Z/T$} is a pair of continuous maps intertwining the structure maps so that, in any local chart belonging to the maximal atlases of $Y/S$ and $Z/T$, the pair $(\tilde\varphi,\bar{\varphi})$ induces a rel--$C^\infty$ map between open subsets of $\bR^n\times S$ and $\bR^m\times T$ for some integers $m,n\ge 0$.
	
	The \emph{category of rel--$C^\infty$ manifolds} $(C^\infty/\cdot)$ has as objects rel--$C^\infty$ manifolds and as morphisms rel--$C^\infty$ maps.
\end{definition}

\begin{definition} A \emph{rel--$C^\infty$ vector bundle} on $Y/S$ is a topological vector bundle $E$ on $Y$ along with a maximal atlas of local trivialisations whose transition functions are rel--$C^\infty$ maps to $\GL_\bK(\rank)$ for $\bK\in \{\bR,\bC\}$. 
\end{definition}

\noindent
As, $E/S$ is naturally a rel--$C^\infty$ manifold, there is a well-defined notion of rel--$C^\infty$ sections of a rel--$C^\infty$ vector bundle.

\begin{example} The pullback of a continuous vector bundle $E\to S$ is a rel--$C^\infty$ vector bundle on $Y/S$.
\end{example}

\begin{example}
	Any rel--$C^\infty$ manifold $Y/S$ has a well-defined \emph{vertical tangent bundle} denoted by $T_{Y/S}$. This is a rel--$C^\infty$ real vector bundle on $Y/S$.
\end{example}

\begin{definition}
	A morphism $f\cl Y'/S\to Y/S$ in $(C^\infty/\cdot)$, where the underlying map $S\to S$ is the identity, is called a \emph{(rel--$C^\infty$) vertical submersion} if the differential $df\cl T_{Y'/S}\to f^*T_{Y/S}$ is surjective.
\end{definition}

\noindent
Rel--$C^\infty$ submersions are given by coordinate projections in suitable local coordinates just like $C^\infty$ submersions. This follows from the inverse function theorem with parameters, \cite[Lemma 5.10]{Swa21}. Similarly, we have a well-behaved notion of \emph{vertical transversality} for rel--$C^\infty$ maps, which is the analogue of usual transversality formulated using the vertical tangent bundle.

\begin{definition}\label{} Let $(\wt{C}^\infty/\cdot)$ be the category whose objects are rel--$C^\infty$-manifolds $Y\xra{p} S$, together with the structure of a smooth manifold with corners on $S$, and whose morphisms are rel--$C^\infty$ morphism $(\wt\varphi,\bar{\varphi}$, where $\bar{\varphi}$ is required to be smooth.
\end{definition}

\noindent
There is a canonical forgetful functor $(\wt{C}^\infty/\cdot)\to (C^\infty/\cdot)$. By abuse of notation, we call objects of $(\wt{C}^\infty/\cdot)$ also rel--$C^\infty$ manifolds with the understanding that the base will from now on always admits a smooth structure.

\begin{remark}\label{} There also exists a forgetful functor from $(\wt{C}^\infty/\cdot)$ to the category of topological manifolds.\end{remark}

\begin{definition}\label{de:tangent-bundle} The \emph{tangent bundle} of $Y\xra{p} S$ in $(\wt{C}^\infty/\cdot)$ is the rel--$C^\infty$ vector bundle $$TY := T_{Y/S}\oplus p^*TS.$$
\end{definition}

\begin{definition}\label{} An \emph{orientation} of $Y/S$ is an orientation of the tangent bundle $TY$.\end{definition}

\noindent
Clearly, any morphism $\varphi = (\wt\varphi,\bar{\varphi}\cl Y/S\to Y'/S'$ has a well-defined differential 
\begin{equation}\label{eq:differential}d\varphi \cl TY\to TY' : (y,\hat{y},v)\mapsto (\wt\varphi(y), d\wt\varphi(y)\hat{y},d\bar{\varphi}(p(y))v) \end{equation}
which defines a rel--$C^\infty$ map $TY/S\to TY'/S'$. If $\varphi$ is an isomorphism, its \emph{orientation sign} is $\pm 1$, depending on whether it is orientation-preserving or orientation-reversing. For the next definition, we have to be careful in the case where the base has corners. Following \cite{ST16}, we say a smooth map $S\to S'$ between manifolds with corneres is a submersion if it is locally diffeomorphism to a projection $\bR_{\geq 0}^{k}\times\bR_{\geq 0}^{n-k}\to \bR_{\ge 0}^k$.

\begin{definition}\label{de:submersion} We call $\varphi \cl Y/S\to Y'/S'$ a \emph{rel--$C^\infty$ submersion} if the map $S\to S'$ is a submersion and $d^v\varphi\cl T^vY\to T^vY'$ is fibrewise surjective.
\end{definition}

\noindent
This recovers the notion of a vertical submersion whenever the base spaces agree, as well as the usual notion of a submersion when the base spaces are just points.

\begin{lemma}\label{lem:fibre-product} If $\varphi \cl Y/S\to Y'/S'$ is a rel--$C^\infty$ submersion and $\psi \cl Z/T\to Y'/S'$ is any rel--$C^\infty$ map, then the fibre product 
	$$\modulo{Y\times_{Y'}Z}{S\times_{S'}T}$$
	is an object of $(\wt{C}^\infty/\cdot)$. 
\end{lemma}

\begin{proof}
	Since $S \rightarrow S'$ is a submersion, we may form the fibre product $S \times_{S'} T$. Then, $Y\times_S(S\times_{S'}T)$ admits a canonical rel--$C^\infty$ structure over $S\times_{S'} T$. The result now follows from the canonical homeomorphism
	\[
	(Y \times_{S} (S \times_{S'} T)) \times_{(Y' \times_{S'} (S \times_{S'} T))} (Z \times_T (S \times_{S'} T)) \ \cong\ Y \times_{Y'} Z
	\]
	and \cite[Lemma~5.7]{Swa21}.
\end{proof}

\begin{cor}\label{cor:fibre-of-submersion} If $\varphi \cl Y/S\to Y'/S'$ is a rel--$C^\infty$ submersion, then $$\varphi\inv(\{y'\}) := (\wt\varphi\inv(\{y'\}),\bar\varphi\inv(\{s'\}))$$ 
	is a rel--$C^\infty$ submanifold of $Y/S$ for any $y'\in Y'$ with image $s'\in S'$. Furthermore, $$T\varphi\inv(\{y'\}) = \ker(d\varphi)|_{\wt\varphi\inv(\{y'\})}.$$
\end{cor}

\begin{definition}\label{de:orientation-fibre-product} If $\varphi \cl Y/S\to Y'/S'$ and $\psi \cl Z'/T'\to Y'/S'$ are transverse, we equip the fibre product $Z/T:= Y/S\times_{Y'/S'} Z'/T'$ with the unique orientation so that the canonical isomorphism 
	$$T_y Y\oplus T_{z'}Z' \to T_{(y,z')}Z \oplus T_{\varphi(y)}Y'$$
	has orientation sign $(-1)^{\dim(Z')\dim(Y')}$ for any $(y,z')\in Z$. We call it the \emph{fibre product orientation}.
\end{definition}

\begin{remark}\label{}This agrees with the convention of \cite{ST23b}. \end{remark}

\noindent
We call a rel--$C^\infty$ manifold $Y/S$ \emph{Hausdorff, respectively paracompact} if $Y$ and $S$ are Hausdorff, respectively paracompact.

\begin{lemma}\label{lem:partition-of-unity} Suppose $Y\xra{p}S$ is a paracompact Hausdorff rel--$C^\infty$ manifold. Given any countable locally finite cover $\{U_i\}\iI$ there exists a rel--$C^\infty$ partition of unity $\{\rho_i\}\iI$ subordinate to $\{U_i\}\iI$, i.e., for each $i \in I$, $\rho_i\cl Y/S\to [0,1]/*$ is a rel--$C^\infty$ map.
\end{lemma}

\begin{proof} Refining $\{U_i\}\iI$, we may assume without loss of generality that each $U_i$ is of the form $h_i : U_i \cong V_i \times B$, where $V_i$ is the image of $U_i$ in $S$ and $B\sub \bR^n$ is the open unit ball, so that $\{h_i\inv(V_i\times B_{\delta_i})\}\iI$ is still an open cover of $Y$ for some $0 < \delta_i < 1$. Using the paracompactness of $S$, we may refine further and assume that $\{V_i\}\iI$ is a locally finite cover of (an open subset of) $S$ as well. Let $\{\tau_i\}\iI$ be a smooth partition of unit subordinate to $\{V_i\}\iI$ and fix for each $i \in I$ a smooth compactly supported bump function $\chi_i \cl B\to [0,1]$, which is identically $1$ on $B_{\delta_i}$. Then 
	$$\rho'_i(y) := \begin{cases}
		\tau_i(p(y))\chi_i(\pr_2 h_i(y))\qquad & y \in U_i\\
		0 \quad & \text{otherwise}
	\end{cases}$$ 
	is a relatively smooth map $Y/S\to [0,1]$ supported in $U_i$. By the assumptions on $\{U_i\}\iI$, the function $a := \s{i\in I}{\rho_i'}$ is rel--$C^\infty$ and positive everywhere. Thus, $\{\rho_i := \frac{1}{a}\rho_i'\}\iI$ is a suitable partition of unity.
\end{proof}

\subsection{Rel--$C^\infty$ orbifolds} In this subsection, we adapt the definitions of \cite{ST16} to the relative setting. The main (technical) result is Lemma~\ref{lem:equivalent-to-ep-groupoid}, allowing us to go from a global Kuranishi chart with corners to a derived orbifold chart with corners.

\begin{definition}\label{de:rel-smooth-groupoid} A \emph{rel--$C^\infty$ groupoid} $\cY = [Y_1/S_1 \rightrightarrows Y_0/S_0]$ consists of a groupoid, where the object and morphism space are rel--$C^\infty$ manifolds, the source and target map are rel--$C^\infty$ submersions, and all other structure maps are relatively smooth.\par Its \emph{orbit space} is $$|\cY| := Y_0/\sim,$$
	where $y \sim y'$ if there exists $\phi \in Y_1$ with $s(\phi) = y$ and $t(\phi) = y'$.
\end{definition}

\begin{definition}\label{de:rel-smooth-groupoid-map} A \emph{morphism of rel--$C^\infty$ groupoids} $\varphi \cl \cY\to \cY'$ consists of two rel--$C^\infty$ maps $\varphi_i \cl Y_i/S_i \to Y_i'/S_i'$ for $i = 0,1$ that intertwine the structure maps of the groupoids.
\end{definition}

\noindent
We will call a rel--$C^\infty$ groupoid just groupoid from now on. The usual definitions carry over.

\begin{definition}\label{de:ep-groupoid} A groupoid $\cY$ is \emph{\'etale} if source and target map are local diffeomorphisms. It is \emph{proper} if $(s,t) \cl Y_1/S_1 \to  Y_0/S_0\times Y_0/S_0$ is proper.
\end{definition}

\begin{definition}\label{de:equivalent-groupoids} We call $\varphi\cl \cY\to \cY'$ an \emph{equivalence of rel--$C^\infty$ groupoids} if $\varphi$ is an equivalence of categories and $\varphi_0$ is a submersion. We call $\cY$ and $\cY'$ \emph{equivalent} if there exists a sequence  $\cY\xleftarrow{}\cY_0\to\dots\leftarrow\cY_n \to\cY'$ of equivalences. 
\end{definition}

\noindent
Following \cite{ST23b}, we call an equivalence $\varphi$ a \emph{refinement} if $\varphi_0$ is a local diffeomorphism. 

\begin{remark}\label{} We need somewhat more definitions as our orbifolds are naturally given as transformation groupoids, a perspective which also easier for some constructions, and we have to pass between the two models for orbifolds back and forth.\end{remark}

\begin{example}\label{ex:finite-group-groupoid} Suppose $G$ is a compact Lie group acting by rel--$C^\infty$ diffeomorphisms on $Y/S$, i.e., acting smoothly on $S$ and continuously on $Y$ so that $Y\to S$ is equivariant and the action map $G\times Y/G\times S\to Y/S$ is relatively smooth. Then 
	$$[(Y/S)/G] := [G\times Y/S\rightrightarrows Y/S]$$
	is a groupoid. If $G$ is finite, it is \'etale and proper.
\end{example}

\begin{example}\label{ex:disjoint-union-groupoids} A disjoint union of étale proper Lie groupoids where the object spaces have the same dimension is an étale proper Lie groupoid.\end{example}

\noindent
The following lemma is standard and written here for completeness. It is used to combine the output of Theorem \ref{thm:open-gw-global-chart} with the results of \cite{ST16}.

\begin{lemma}\label{lem:equivalent-to-ep-groupoid} Suppose $G$ is a compact Lie group acting by rel--$C^\infty$ diffeomorphisms on $Y/S$ so that $G_y$ is finite for any $y \in Y$ and $Y/G$ is paracompact. Then $[(Y/S)/G]$ is equivalent to an \'etale proper groupoid.
\end{lemma}

\noindent
Note that we do not require the action of $G$ on $S$ to be almost free. This will not be the case in our main application. The proof is a generalisation of \cite[Example~2.6]{RS06} to the rel--$C^\infty$ setting.

\begin{proof} Let $y \in Y$ with image $s \in S$ be arbitrary. By \cite[Lemma~5.9]{HS22} applied to a neighbourhood of $y\in Y$ and $\Gamma = G_s$\footnote{The cited result assumes that $\Gamma$ is a finite group. However, the same argument works with any compact Lie group, replacing the sum by integration over the Haar measure.} we can find a smooth submanifold $W_s \sub S$ and $T_y \sub Y$ so that $T_y/W_s$ is a rel--$C^\infty$ manifold, $T_y$ is $G_y$-invariant, and the action map 
	\begin{equation}\label{eq:slice-homeo} G\times_{G_y} T_y \to Y\end{equation}
	is a homeomorphism onto an open neighbourhood $U_y$ of $y$. The inclusion $T_y \hkra U_y$ extends to an equivalence 
	$$[(T_y/W_s)/G_y]\xra{\simeq} [(U_y/S)/G].$$
	Let $Z \sub Y$ be any countable subset so that $\{U_y/G\}_{y\in Z}$ is an open locally finite cover of $Y/G$. Define $$\wt Y_0 := \djun{y\in Z}T_y \qquad \wt Y_1 := \djun{(y,y')\in Z^2}\{(g,x,x')\in G\times T_y\times T_{y'}\mid g\cdot x = x'\}$$
	and define $\wt S_0$ and $\wt S_1$ similarly with $W_s$ instead of $T_y$. The structural maps are given by
	$$s(y,y',g,x,x') = (y,x) \in T_y \qquad t(y,y',g,x,x') = (y',x')\in T_{y'}$$
	and analogously on the level of base spaces. Multiplication, unit element and inversion are defined in the obvious way. This defines an \'etale proper groupoid $\wt\cY$. Using equivariant tubular neighbourhoods, one shows that $\wt \cY$ is equivalent to the groupoid $\cU$ defined in the same way with the collection $\{U_y/S\}$ instead of $\{T_y/W_{p(y)}\}$. Finally, the inclusions $U_y \hkra Y$ define an equivalence $\Psi\cl \cU \to [(Y/S)/G]$. 
\end{proof}

\begin{remark}\label{} This lemma shows in particular that equivalence does not preserve the property of being \'etale. \end{remark}

\begin{definition}\label{} A \emph{orbifold} $(Y,\cY)$ consists of a topological space $Y$ together with an étale proper Lie groupoid $\cY$ so that $|\cY|$ is homeomorphic to $Y$.
\end{definition}

\subsection{Going from a rel--$C^\infty$ to a smooth structure}\label{subsec:smoothing-theory}
In this subsection we explain how to use the smoothing theory of Lashof, \cite{Las79}, to equip the `total space' of a rel--$C^\infty$ manifold $X/B$ over a smooth manifold $B$ with a smooth structure so that the structural map $X\to B$ becomes smooth.

\begin{definition} Let $G$ be a compact Lie group acting continuously by rel--$C^\infty$ diffeomorphisms on the rel--$C^\infty$ manifold $X/B$. A \emph{relative $G$-smoothing} of $X\xra{\pi} B$ is a $G$-equivariant rel--$C^\infty$ diffeomorphism $\psi \cl X\to V$, where $V$ is a smooth $G$-manifold with a smooth submersion $p \cl V\to B$ and $p\g\psi = \pi$.
\end{definition}

\begin{theorem}\label{thm:to-smooth} Suppose $X/B$ is a rel-$C^\infty$ manifold equipped with a rel--$C^\infty$ action by a compact Lie group $G$ that acts smoothly on $B$. Then, given any compact $G$-invariant subset $K \sub X$,  there exists an orthogonal $G$-representation $W$ and a relative $G$-smoothing of $X'\times W$, where $X'$ is an open invariant neighbourhood of $K$.\par
	Moreover, if $f\cl X/B\to Y/*$ is a rel--$C^\infty$ submersion, we can find a relative $G$-smoothing with respect to which $f$ is smooth.
\end{theorem}

\begin{proof} The proof strategy is similar to the one of \cite[Theorem~1.3]{Las79}. However, instead of embedding $X$ into an orthogonal $G$-representation, we will embed it into $B\times V$ for some orthogonal $G$-representation $V$. For this, we will use that a rel--$C^\infty$ structure (in fact, rel--$C^1$ suffices) yields a vertical tangent bundle $T^vX$ and any choice of rel--$C^\infty$ $G$-invariant Riemannian metric on $T^vX$ determines a $G$-equivariant rel--$C^\infty$ embedding 
	\begin{equation}\label{} \varphi \cl \Omega\sub T^vX \to X\times_B X
	\end{equation}
	where $\Omega$ is an open subset of the zero section, $\pr_1\varphi(x,v) = x$, $\varphi(x,0) = (x,x)$ and $\varphi$ is relatively smooth over $B$.\par 
	Suppose first that there exists a vector bundle $E \to B$ so that $T^vX \cong \pi^*E$. Let $W$ be the orthogonal $G$-representation and $j \cl X'/B\to B\times W/B$ be the relatively smooth $G$-equivariant embedding of Proposition~\ref{prop:embedding-in-representation}. For the rest of the proof we will replace $X$ by $X'$ without further mention. Since $\im(j)$ is a relatively smooth submanifold of $B\times W/B$, we can use a relative version of the tubular neighbourhood theorem and find an equivariant rel--$C^\infty$ retraction $r \cl U\to X$, where $U$ is an invariant open neighbourhood of $j(X)$. Define 
	\begin{equation}\label{} 
		\theta \cl X\times W\to B\times W : (x,w)\mapsto (\pi(x),j_2(x)+w),
	\end{equation}
where $j_2$ is the composition of $j$ with the projection to $W$. This map is clearly equivariant and relatively smooth. Since $\theta(x,0) \in \im(j)$, we can find an invariant rel--$C^\infty$ map $\delta \cl X\to \bR_{> 0}$ so that $\theta(x,w)\in U$ if $\norm{w} < \delta(x)$. Let $W_1\sub W$ be the open unit ball and define $\theta'\cl X\times W_1\to U$ by $\theta'(x,w) = \theta(x,\delta(x)w)$. Choosing $\delta$ sufficiently small, we can ensure that $(r\theta'(x,y),x)\in \varphi(\Omega)$. Now, define 
\begin{equation}\label{}
	\psi \cl X\times W\to E\oplus W:  (x,w)\mapsto \lbr{\theta'(x,w),\pr_2\varphi\inv((r\theta'(x,y),x))}.
\end{equation}
This is rel--$C^\infty$ and equivariant. It follows from the same argument as in \cite[Theorem~1.3]{Las79} that $\psi$ is an open embedding. Since we can equip $E$ with a smooth structure that makes the $G$-action smooth, we can pull back the smooth structure on $E\oplus W$ along $\psi$. For the general case one can use the following relative version of \cite{Las79} as loc. cit: If $\rho \cl E\to X$ is a rel--$C^\infty$ vector bundle, then a relative rel--$C^\infty$ vector bundle reduction $\varphi\cl \Omega\sub T^vX\to X\times_B X$ induces a relative rel--$C^\infty$ vector bundle reduction $\wh\varphi\cl \rho^*(E\oplus T^vX)\to E\times_B E$. The proof of this claim is verbatim the same as the one of \cite[Lemma~1.6]{Las79}.\par 
The last assertion follows immediately from the former by noting that in this case $X/B\times Y$ is a rel--$C^\infty$ manifold due to Lemma~\ref{lem:fibre-product}.
\end{proof}

\begin{proposition}\label{prop:embedding-in-representation} Suppose $X/B$ is a rel--$C^\infty$ manifold with a relatively smooth $G$-action. Given any compact $G$-invariant subset $K \sub X$, there exists a $G$-representation $W$ and an invariant neighbourhood $X'$ of $K \subset X$ with a relatively smooth equivariant embedding $j \cl X'\to B\times W$ over $B$.
\end{proposition}

\begin{proof}
	  Let $d$ be the rel--$C^\infty$ invariant fibrewise metric associated to a chosen Riemannian metric on $T^vX$. To construct $j$, suppose we have constructed for each orbit a $G$-equivariant rel--$C^\infty$ embedding $f'_x \cl  B_\epsilon(G\cdot x)\to B\times W'_x$ for some orthogonal $G$-representation $W'_x$. Choosing $\epsilon  > 0$ sufficiently small, we may assume $f'_x$ has bounded image in $W'_x$ (with respect to some arbitrary $G$-invariant norm) since $G$ is compact. Equip $\bR$ with the trivial $G$-representation and let $W_x = W'_x \oplus \bR$. Then, define 
	\begin{equation}\label{} 
		f_x(y) = \frac{1}{1+\norm{f'_x(y)}}(f'_x(y),1).
	\end{equation} 
	This is clearly still a $G$-equivariant rel--$C^\infty$ embedding. Note that it takes values in the unit sphere. Now, choose a finite set $Z\sub X$ such that $\{B_\epsilon(G\cdot x)\mid x \in Z\}$ covers $K$ and let $\{\rho_x\}_{x\in Z}$ be a $G$-invariant partition of unity subordinate to it. Given these data, we can define 
	\begin{equation}\label{} 
		j \cl X' \to B\times\p{x\in Z}{W_x}
	\end{equation}
	by 
	\begin{equation}\label{}
		j(y) = (\pi(y),(\rho_x(y)f_x(y))_{x\in Z}),
	\end{equation}
	where $\pi \cl X\to B$ is the structural map and we extend $\rho_x f_x$ by zero to all of $X'$. Then, $j$ is manifestly relatively smooth and equivariant. Let us check that it is injective. If $j(y) = j(z)$, then, as $f_x$ has image in the unit sphere, $\rho_x(y) = \rho_x(z)$. Thus, $f_x(y) = f_x(z)$ for any $x \in Z$ so that $\rho_x(y) > 0$. As some such $x$ must exist, we can use that $f_x$ is an embedding to obtain that $y = z$. Closedness follows from the fact that each $f_x$ is closed.\par 
	In order to construct $f := f'_x$, we use an inductive argument as in the proof of \cite[Theorem~6.1]{Mos57}. The first step yields a global embedding; for the second step, where we `filter' the fibre, we restrict to a neighbourhood of an orbit. Fix, thus, $b \in B$ and define the $G$-invariant closed subsets 
	\[B_{\ge b} := \{b'\in B\mid [G_b] \sub [G_{b'}]\}\qquad B_{> b} := \{b'\in B\mid [G_b] \sub [G_{b'}]\}\]
	and set $B_b := B_{\ge b}\sm B_{> b}$. Note that $B_{\ge b}$ and $B_{> b}$ are disjoint unions of smooth submanifolds. Then, $X_{* b } = X\times_B B_{* b}$ is canonically a rel--$C^\infty$ manifold over $B_{* b}$ for $*\in \{=,\ge ,>\}$. Suppose we are given $G$-equivariant rel--$C^\infty$ embeddings 
	\begin{equation}\label{eq:building-blocks} 
		f_= \cl X_{ = b}/B_{= b} \to B_{= b}\times V_=/B_{= b}  \qquad f_{>} \cl X_{ > b}/B_{> b} \to B_{> b}\times V_>/B_{> b}.
	\end{equation}
	Set $V_{\ge} := V_= \oplus V_{>}$. Let $\rho\cl B_{\ge b}\to [0,1]$ be a smooth $G$-invariant function with $\rho\inv(\{0\}) = B_{\ge b}$. By the equivariant Tietze extension theorem, we can extend $f_>$ to a $G$-equivariant continuous function 
	\[\wt f_{>} \cl X_{\ge b}\to B_{\ge b}\times V_{>}.\] 
	Then, define $f_{\ge}\cl X_{\ge b}/B_{\ge b} \to B_{\ge b}\times V_\ge/B_{\ge b}$ by 
	\begin{equation}\label{eq:patching} 
		f_{\ge}(x) = \begin{cases}
			(\pi(x),\rho(\pi(x))f_{=}(x),(1+\rho(\pi(x)))\wt f_{>}(x))\qquad & x\in X_{=b}\\
			(\pi(x),0,f_{>}(x))\quad & x \in X_{> b}.
		\end{cases}
	\end{equation}
	Using induction on the longest sequence of stabilisers $[G_{b_1}]\sub [G_{b_2}]\dots \sub G$ (in an invariant neighbourhood of the compact subset $K$), this shows that we can find a $G$-equivariant embedding $X/B\hkra B\times W/B$ for some orthogonal $G$-representation $W$ once we have constructed $f_=$ as in~\eqref{eq:building-blocks} for each $b$. Fix $x_0 \in X/B$ with $b_0  =\pi(x_0)$ and let $B_0 := B_{= b_0}$ and $X_0 = B_\epsilon(G\cdot x)\cap \pi\inv(B_0)$. Then $X_0/B_0$ is canonically relatively smooth. To construct $f := f_=\cl X_0/B_0 \to B_0 \times V/B_0$ we use the same inductive construction as before, now filtering $X_0$ by the (conjugacy class of the) stabiliser. Replacing $X_0$ by a neighbourhood of $G\cdot x_0$, we may assume $G_y \sub G_{x_0}$ for any $y \in X_0$ up to conjugacy. Given functions
	\begin{equation}\label{eq:building-blocks-vertical} 
		f_= \cl X_{ = x_0}/B_{0} \to B_{0}\times V_=/B_{0}  \qquad f_{>} \cl X_{ > x_0}/B_0 \to B_0\times V_>/B_0,
	\end{equation}
we can use Lemma~\ref{lem:extending-rel-smoothly} and the formula~\eqref{eq:patching} to construct $f_{\ge} \cl X_{\ge x_0}\to B_0\times V_{=}\oplus V_{>}$ using a relatively smooth function $\rho \cl X_{\ge x_0}\to [0,1]$ instead with $\rho\inv(0) = X_{> x_0}$ and so that all (relative) derivatives of $\rho$ tend to zero locally uniformly near $X_{> x_0}$. Since $X_{> x_0}$ is a (disjoint union of) smooth submanifold(s), this can be achieved by using a tubular neighbourhood as in Lemma~\ref{lem:extending-rel-smoothly}. It remains thus to construct an equivariant embedding
\begin{equation*}\label{}
	f \cl X_{= x_0}/B_0\to B_0 \times V.
 \end{equation*}
Recall for this, that we assumed $X_0$ to be a neighbourhood of the orbit $G\cdot x_0$. Fix a $G_{x_0}$-equivariant chart $\varphi \cl U_b \times U^v \to X_{=x_0}$. Let $B' \sub U_b$ be a $G$-slice through $b_0$ and $S \sub U^v$ be a $G_{b_0}$-slice through $x_0$. The key observation is that 
\begin{equation}\label{eq:iso} 
	G\times_{G_{x_0}} (B'\times S)\,\cong\, G\times_{G_{b_0}}(B'\times (G_{b_0}\times_{G_{x_0}}S))
\end{equation}
equivariantly, whence $B'\times S$ is a $G$-slice through $x \in X_{=x_0}$, this space is relatively smooth over $B$, and the map 
\begin{equation*}\label{} 
G\times_{G_{x_0}} (B'\times S)\to  X_= : [g,(b,s)]\mapsto g\cdot \varphi(b,s)
\end{equation*}
is a relatively smooth diffeomorphism (over $B$). Choose a \emph{smooth} embedding $j \cl S\hkra \bR^n$. By \cite[Lemma~5.2]{Mos57}, we can pick an orthogonal $G$-representation $V$ admitting $v \in V$ with $G_v = G_{x_0}$. Then, define 
\begin{equation}\label{} 
	\psi' \cl G\times_{G_{x_0}} (B'\times S) \to B\times V^{\oplus n} : [g,b,s]\mapsto (g\cdot b,(j_k(s)g\cdot v)_{1 \leq k \le n}).
\end{equation}
Precomposing $\psi'$ with the rel--$C^\infty$ diffeomorphism in~\eqref{eq:iso} yields the desired embedding.
\end{proof}

\begin{lemma}\label{lem:extending-rel-smoothly} Let $X/B$ be a rel--$C^\infty$ closed manifold and $A\sub X$ a closed rel--$C^\infty$ submanifold over $B$. Given a rel--$C^\infty$ function $f \cl A/B\to B\times V/B$, where $V$ is some Euclidean space, there exists a rel--$C^\infty$ extension $\wt f\cl X/B\to B\times V/B$. The same is true in the equivariant setting.
\end{lemma}

\begin{proof}
 Fix a rel--$C^\infty$ Riemannian metric on $X$ and let $\epsilon \cl A\to [0,1]$ be a sufficiently small cut-off function so that $U = \union{x\in A}{B_{\epsilon(x)}(x)}$ is a tubular neighbourhood of $A$. Let $r \cl U\to A$ be the canonical retraction and fix a cut-off function $\rho \cl \bR\to [0,1]$ with $\rho(0) = 1$ and $\rho(t) = 0$ for $t \ge \frac12$. Then, define $d(y) := \frac{d(y,r(y))}{\epsilon(r(y))}$ for $y \in U$ and set 
 $$\wt f(y) =\begin{cases} (\pi(y),\rho(d(y))f_V(y)) & y\in U\\
 	(\pi(y),0)\quad & \text{otherwise}.\end{cases}$$
 \noindent This is clearly rel--$C^\infty$ and we can choose the extension to be $G$-equivariant by choosing the metric to be $G$-invariant and picking invariant functions $\epsilon$ and $\rho$.
\end{proof}

\begin{definition} Two relative $G$-smoothings $\psi_0,\psi_1$ of a rel--$C^\infty$ $G$-manifold $X/B$ are \emph{concordant} if there exists a relative $G$-smoothing $\wt\psi$ of $X\times[0,1]/B\times[0,1]$ with $\wt\psi|_{B\times\{i\}} = \psi_i$.
\end{definition}

%

In \cite[Stable $G$-smoothing theory II]{Las79}, Lashof proves that isotopy classes of vector bundle lifts of the tangent microbundle correspond bijectively to sliced concordance classes of equivariant smoothings. In particular, two smoothings associated to the same vector bundle lift (as is the case for us) are shown to be sliced concordant. For our main application, Theorem~\ref{thm:cubical-cobordisms-enhanced}, we need the following stronger statement.

\begin{lemma}\label{lem:invariance-of-smoothing}
	Suppose $X/B$ is a rel--$C^\infty$ $G$-manifold. Given any relative $G$-equivariant diffeomorphism 
	\[\psi \cl \del[0,1]^k\times X\to \del[0,1]^k\times V\]
	over $\del[0,1]^k\times B$ where $V$ is a smooth $G$-manifold with a smooth equivariant submersion $V\to B$ and the diffeomorphism is constructed as in Theorem~\ref{thm:to-smooth}, there exists an orthogonal $G$-representation $W$ and a `filling' 
	\[\wt\psi \cl [0,1]^k\times X\times W\to [0,1]^k\times V\times W\]
	of $\psi\times\ide_W$, that is, a rel--$C^\infty$ diffeomorphism over $[0,1]^k\times B$, which extends $\psi$.
\end{lemma}

\begin{proof} Reparametrise the interval to $[-1,1]$. We first observe the following. Suppose $j \cl \del[-1,1]^k\times X\to \del[-1,1]^k\times B\times W$ is an equivariant rel--$C^\infty$ embedding. Fix another equivariant embedding $i \cl X\to B\times W'$ and let $r \cl [-1,1]^k\sm\{0\}\to \del[-1,1]^k$ be a smooth retraction. Pick a cut-off function $\chi$ on $[-1,1]^k$, which is zero near the origin and $1$ near $\del[-1,1]^k$. Then, define 
\begin{equation*}\label{} 
	\notag\wt j \cl [-1,1]^k\times X\to [-1,1]^k\times B\times W\times W'
\end{equation*}
by 
\begin{equation}\label{} 
	\wt j(t,x) = (t,\pi(x),\chi(t)j_{r(t)}(x),(1-\chi(t))i(x)).
\end{equation}
This is clearly an equivariant rel--$C^\infty$ embedding that extends $j\oplus 0$. 
Going through the proof of Theorem~\ref{thm:to-smooth} with the embedding $j\oplus 0$ instead of the embedding $j$, we obtain the diffeomorphism $\psi\times \ide_{W'}$ instead of $\psi$. This completes the proof.
\end{proof}

Finally, we make the following observation, which has a similar conclusion to \cite[Lemma~4.5]{AMS21}. For completeness, we include a proof.

\begin{lemma}
	\label{lem:smoothing-maps}
	Suppose $\cT$ and $X$ are smooth manifolds and $f\cl \cT \rightarrow X$ is a continuous map. Let $\exp$ be the exponential map of a Riemannian metric on $X$ and let $\wt f = \exp_{f}\cl  f^*TX \rightarrow X$ be the stabilisation of $f$. If $X$ is compact, there exists a smooth structure on a neighbourhood of the zero section of $f^*TX$ such that 
	\begin{enumerate}[label=\normalfont \arabic*),leftmargin=20pt,ref=\arabic*]
		\item $\wt f$ and the projection $f^*TX \rightarrow \cT$ are smooth;
		\item if $\cT/\cB$ is a rel--$C^\infty$ manifold and $f$ is rel--$C^\infty$, then the rel--$C^\infty$ structure induced by the smooth structure on $f^*TX$ agrees with the canonical one;
		\item if $f$ is already smooth near a set $C \subset \cT$, then the smooth structure on $f^*TX$ is the given one near $(f|_C)^*TX$.
	\end{enumerate}
	If $f$ is invariant with respect to a smooth action by a compact Lie group $G$ on $\cT$, then we can choose the $G$-action to be smooth with respect to the new smooth structure and assume $\wt f$ to be invariant as well.
\end{lemma}
\begin{proof}
	Let $\epsilon$ be the injectivity radius of the metric and let $g$ be a smooth deformation of $f$ with $||f-g||_{C^0} < \epsilon/2$. Then $g^*TX$ has a canonical smooth structure, and there exists an isomorphism of topological (or rel--$C^\infty$) vector bundles \[
	\Phi\cl f^*TX \cong g^*TX.
	\]
	For $u \in \cT$, the images of the maps $\wt f, \wt g \circ \Phi\cl  B_\epsilon(0) \subset T_{f(u)}X \rightarrow X$ intersect in an open neighbourhood of $f(u)$. Let $h := \wt f^{-1} \circ \wt g \circ \Phi$ wherever it is defined. Then, we have that $\wt f = \wt g \circ \Phi \circ h^{-1}$ on a neighbourhood of the zero section. Define the smooth structure on $f^*TX$ to be the one pulled back via the map $\Phi \circ h^{-1}$. The same procedure works in the $G$-equivariant setting.
\end{proof}

\begin{remark}\label{} 
	Note that the zero section $\cT \hookrightarrow f^*TX$ might not be smooth with respect to the smooth structure on $f^*TX$ given by this lemma, but it is rel--$C^\infty$ whenever $f$ is.
\end{remark}
\section{Extensions of Thom forms}\label{sec:extend-Thom-forms}
The main result of this section says that we can extend certain differential forms given on a collection of submanifolds to a differential form on the whole manifold, if the restrictions to intersections agree. 

\begin{proposition}\label{prop:extend-thom-form} Let $M$ be a  $G$-manifold with corners, where $G$ is a compact Lie group acting almost freely, and suppose $E\xra{\pi} M$ is an orientable smooth $G$-vector bundle. Let $M_i \subset M$, for $i \in \cI$, be a finite collection of $G$-invariant submanifolds of $M$. Suppose that the intersection $\bigcap_{i \in I} M_i$ is clean for all $I\sub \cI$, and that we are given $G$-invariant Thom forms $\tau_i$ for $E|_{M_i}$ such that $\tau_i|_{M_{ij}} = \tau_j|_{M_{ij}}$ for all $i,j$. Then there exists a $G$-invariant Thom form $\tau$ of $E$ so that $\tau|_{M_i} = \tau_i$ for all $i\in \cI$.
\end{proposition}

In order to prove this result, we introduce some definitions. Throughout, we will be working in the context and with the notation of Proposition \ref{prop:extend-thom-form}.

\begin{definition}
	Define the double complex \begin{equation}
		C^{p,q}(M_i, E) := \p{\substack{I\sub \cI,\\|I| = q+1}}{\Omega_{cv}^p(E|_{M_I}) }
	\end{equation}
	Its two differentials are given by the deRham differential $d\cl C^{p,q}(M_i, E) \rightarrow C^{p+1,q}(M_i, E)$ and the \v{C}ech differential $\delta\cl C^{p,q}(M_i, E) \rightarrow C^{p,q+1}(M_i, E)$ given by 
	\begin{equation*}
		\delta(f)_{I} = \sum^{q}_{k = 0} (-1)^k i_{I}f_{\hat{I}_k}
	\end{equation*}
	where $\hat{I}_k$ is obtained from $I$ by removing the $k$'th element, and $i_{I}\cl \Omega^*_{cv}(E|_{M_{\hat{I}_k}}) \rightarrow \Omega^*_{cv}(E|_{M_I})$ is induced by the inclusion map $M_I \hookrightarrow M_{\hat{I}_k}$.
\end{definition}

\begin{lemma}\label{lem:d-exact}
	The differential $d$ is exact in degrees $p \leq \rank(E)$.
\end{lemma}
\begin{proof}
	This follows directly from the Thom isomorphism theorem applied to $E|_{M_I}$.
\end{proof}
Unlike in the case of \v{C}ech cohomology (where one uses open covers), some more work is needed to show the following lemma.
\begin{lemma}
	\label{lem:global exactness of delta}
	The differential $\delta$ is exact in all degrees, including $q = -1$, where $\delta$ is just given by restricting from $M$ to the $M_i$.
\end{lemma}
In order to show this, we first prove the result locally, generalising \cite[Lemma~2]{Kott}.

\begin{lemma}
	\label{lem:local exactness of delta}
	Suppose $M = \bR^n \times Y$ for some smooth manifold $Y$ and that $H_i \subset \bR^n$ is a linear subspace (not necessarily of codimension 1) for each $i\in\cI$. Assume $M_i = H_i \times Y$ and let $E \rightarrow Y$ be any vector bundle. Then the differential $\delta$ is exact on $C^{*,*}(\{H_i \times Y\}, \pi^*E)$, where $\pi\cl  \bR^n \times Y \rightarrow Y$ is the projection.
\end{lemma}

\begin{proof}
	Let $f \in C^{*,q}(\{H_i \times Y\}, \pi^*E)$ be $\delta$-closed. Then for $J = (j_1, \dots, j_q)$, define
	\begin{equation*}
		g_{J} = \sum_{I<k<j_1} (-1)^{|I|}\pi_{J}^{IkJ} \circ i_{IkJ} f_{kJ},
	\end{equation*}
	where $\pi_{J}^{IkJ}\cl \Omega^*_{cv}(E|_{M_{IkJ}}) \rightarrow \Omega^*_{cv}(E|_{M_{J}})$ is induced by the projection map $M_{J} \rightarrow M_{IkJ}$, and the sum is over all $k \notin J$ and $I = (i_1, \dots, i_r) \subset \{1, \dots n\} \setminus (J \cup \{k\})$ such that the tuple $(i_1, \dots, i_r, k, j_1, \dots, j_q)$ is ordered. We allow for $I =  \emptyset$. Then, for any $J = (j_0, \dots, j_q)$, we have 
	\begin{align}	\label{eq:terms with p>0}
		\notag\delta(g)_{J} &:= \sum_p (-1)^p i_{J} g_{\hat{J}_p}\\
		&= \sum_{\substack{p \neq 0\\I < k < j_0}} (-1)^{p+|I|} \pi_{J}^{IkJ} \circ i_{IkJ} f_{k\hat{J}_p}+ \sum_{I < k < j_1} (-1)^{|I|} \pi_{J}^{IkJ} \circ i_{IkJ} f_{k\hat{J}_0},
	\end{align}
	where we have separated the contributions from $p=0$ and $p >0$, and used the fact that 
	$$i_{J} \circ \pi_{\hat{J}_p}^{Ik \hat{J}_p} \circ i_{Ik \hat{J}_p} = \pi_{J}^{IkJ} \circ i_{Ik J}.$$ 
	We next separate the contributions with $p=0$ further by taking terms 
	\begin{itemize}
		\item with $(k,I) = (j_0, \emptyset)$,
		\item with $k = j_0$ and $|I| >1$
		\item with $k < j_0$,
		\item and with $k>j_0$,
	\end{itemize} 
	obtaining
	\begin{align}
		\sum_{I < k < j_1} (-1)^{|I|} \pi_{J}^{IkJ} \circ i_{IkJ} f_{j \hat{J}_0} = f_J
		&+\sum_{\substack{I\neq \emptyset\\I < j_0}} (-1)^{|I|} \pi_{J}^{IJ} \circ i_{IJ} f_{J} 
		\label{eq:terms with I nonempty and k = j0}\\
		&+\sum_{I < k < j_0} (-1)^{|I|} \pi_{J}^{IkJ} \circ i_{IkJ} f_{j \hat{J}_0}
		\label{eq:terms with k < j0}\\
		&+\sum_{\substack{k > j_0, I\\(I,k,j_1)}} (-1)^{|I|} \pi_{J}^{IkJ} \circ i_{IkJ}f_{k\hat{J}_0}.
		\label{eq:terms with k > j0}
	\end{align}
	First we show that the final contribution in line \eqref{eq:terms with k > j0} vanishes, computing
	\begin{align*}
		\sum_{\substack{k > j_0, I\\I < k < j_1}} (-1)^{|I|} \pi_{J}^{IkJ} \circ i_{IkJ} f_{k\hat{J}_0} &= \sum_{\substack{k > j_0, I\\I < k < j_1 \\ 
				j_0 \in I}} (-1)^{|I|} \pi_{J}^{IkJ} \circ i_{IkJ} f_{k\hat{J}_0} + \sum_{\substack{k > j_0, I\\I < k < j_1 \\ j_0 \notin I}} (-1)^{|I|} \pi_{J}^{IkJ} \circ i_{IkJ}f_{k\hat{J}_0}\\
		&= \sum_{\substack{k > j_0, I\\I < k < j_0\\ j_0 \notin I}} (-1)^{|I| + 1} \pi_{J}^{IkJ} \circ i_{IkJ} f_{k\hat{J}_0} + \sum_{\substack{k > j_0, I\\I < k< j_1\\ j_0 \notin I}} (-1)^{|I|} \pi_{J}^{IkJ} \circ i_{IkJ}f_{k\hat{J}_0}\\& = 0.
	\end{align*}
	The second equality holds as $\pi_{J}^{IkJ} \circ i_{IkJ}$ does not depend on whether $j_0 \in I$ since in either case we set the variable $x_{j_0} = 0$. Next, we note that 
	\begin{align*}
		0 = \delta(f)_{kJ} = i_{kJ} f_J + \sum_{p = 0}^k (-1)^{p+1} i_{kJ} f_{k\hat{J}_p}.
	\end{align*}
	Applying $\pi_{J}^{IkJ} i_{IkJ}$ to this equation and separating $p = 0$ and $p > 0$ yields
	\begin{equation*}
		0 = \pi_{J}^{IkJ} i_{IkJ} f_J - \pi_{J}^{IkJ} i_{IkJ} f_{k\hat{J}_0} + \sum_{p=1}^{k} (-1)^{p+1} \pi_{J}^{IkJ} i_{IkJ} f_{k\hat{J}_p}.
	\end{equation*}
	Next we rewrite the terms with $I \neq \emptyset$ and $k = j_0$ in line \eqref{eq:terms with I nonempty and k = j0}, that is,
	\begin{equation*}
		\sum_{\substack{I\neq \emptyset\\I < j_0}} (-1)^{|I|} \pi_{J}^{IJ} \circ i_{IJ} f_{J} = \sum_{I < k < j_0} (-1)^{|I|+1} \pi_{J}^{IkJ} \circ i_{IkJ} f_{J}.
	\end{equation*}
	We thus find that the terms with $p > 0$ in line \eqref{eq:terms with p>0}, those with $p = 0, I \neq \emptyset, k = j_0$ in line \eqref{eq:terms with I nonempty and k = j0}, and those with $p =0, k < j_0$ in line \eqref{eq:terms with k < j0} are exactly covered by the relation $\delta(f) = 0$.
\end{proof}

\begin{remark}
	We note that the same result holds if we work instead in a corner stratum $\bR_{\geq 0}^{n_1} \times \bR^{n_2}$ and take appropriate coordinate subspaces.
\end{remark}

\begin{proof}[Proof of Lemma \ref{lem:global exactness of delta}]
	For each nonempty subset $I\sub \cI$, pick a G-invariant tubular neighbourhood $\pi_I\cl U_I \rightarrow M^\circ_I$, where $M_I^\circ := M_I \setminus \cup_{J \ssuper I} M_J$. We may assume they are sufficiently small so that $\cc{U_I} \cap M_J = \emptyset$ unless $J \subset I$. Moreover, as the intersections are clean, we can choose the $\pi_I$ so that under the isomorphism $U_I \cong \bR^{\text{codim} M_I} \times M_I$, the intersection $U_I \cap M_J$ is mapped to $H_J \times M_I$ for some coordinate subspace $H_J$. Here we have covered our $\cc{U_I}$ by open sets so that the normal bundle to $M_I$ is trivial over $U_I$. Let $\{\chi_I\}_{\emst\neq I\sub \cI}$ be a partition of unity subordinate to $\{U_I\}_{\emst\neq I\sub \cI}$.\par Suppose now $f\in C^{p,q}(\{M_i\},E)$ is $\delta$-closed. Then, the image of $f$ in $C^{p,q}(\{M_i \cap U_I\},E|_{U_I})$ is still $\delta$-closed for each $I$. We thus obtain an element $f^I \in C^{p,q}(H_i \times M_I, E|_{U_I})$. By Lemma \ref{lem:local exactness of delta}, we obtain an element $g^I \in C^{p,q}(\{H_i \times M_I\}, E|_{U_I})$ such that $\delta(g^I)= f^I$. Setting $g = \sum_I \chi_I g_I$, we obtain that $\delta (g) = f$.
\end{proof}

\begin{lemma}
	\label{lem:can pick primitives compatibly}
	In the situation of Proposition \ref{prop:extend-thom-form}, let $\tau' \in A_{cv}^*(E)$ be any $G$-invariant Thom form over $M$. Then, there exist G-invariant forms $\alpha_i \in A_{cv}^*(E|_{M_i})$ such that $\tau_i = \tau'|_{M_i} + d\alpha_i$ and $\alpha_i|_{M_{ij}} = \alpha_j|_{M_{ij}}$.
\end{lemma}

\begin{proof} Let $k:= \rank(E)$.
	Both $\tau'|_{M_{ij}}$ and $\tau_i|_{M_{ij}}$ are Thom forms for $E|_{M_{ij}}$ for any $i,j$, so there exists $\alpha_0 \in C^{k-1,0}(\{M_i\},E)$ with $d\alpha_{0,i} = \tau_i - \tau'|_{M_i}$. We want to modify $\al_0$ to obtain $\alpha$ with the same property satisfying additionally that $\delta(\alpha) = 0$. As $\delta(\tau-\tau') = 0$, we find that $d\delta(\al_0) = 0$. By Lemma~\ref{lem:d-exact}, we can find $\sigma^{k-2,1}$ with $d(\sigma^{k-2,1}) = \delta(\al_0)$. 
	Repeating this diagonally down the double complex, we obtain $\sigma^{k-2-i,1+i}\in C^{k-2-i,1+i}(\{M_i\},E)$ with 
	\begin{equation}\label{eq:using-exactness}d(\sigma^{k-2-i,1+i}) = \delta(\sigma^{k-1-i,i}). \end{equation}
	 We now observe that 
	$$d\delta( \sigma^{0,k-1}) = \delta d (\sigma^{0,k-1}) = \delta^2( \sigma^{1,k-2}) = 0.$$ 
	Thus, the elements of $\delta \sigma^{0,k-1}$ are locally constant functions on $E_{M_I}$ with compact vertical support. Hence, they must vanish. Thus, we can find $\beta^{0,k-2}\in C^{0,k-2}(\{M_i\},E)$ with $\delta(\beta^{0,k-2}) = \sigma^{0,k-1}$. Setting $\sigma^{1,k-2}_1 :=\sigma^{1,k-2}-d\beta^{0,k-2}$ we obtain a new form satisfying \eqref{eq:using-exactness} as well as $\delta( \sigma^{1,k-2}_1) = 0$. Repeating this diagonally up the double complex shows that we can modify $\al_0$ to obtain a cochain $\alpha\in C^{k-1,0}(\{M_i\},E)$ with $\delta(\al) = 0$ and $\tau_i -\tau'|_{M_i} = d\alpha_i$.
\end{proof}

\begin{proof}[Proof of Proposition \ref{prop:extend-thom-form}]
	Let $\tau'\in A_{cv}^*(E)$ be any $G$-invariant Thom form over $M$ and choose forms $\alpha_i \in A_{cv}^*(E|_{M_i})$ as in Lemma \ref{lem:can pick primitives compatibly}. Then, Lemma \ref{lem:global exactness of delta} yields $\alpha \in A_{cv}^*(E)$ extending all $\alpha_i$. Setting $\tau = \tau' + d\alpha$ proves the result.
\end{proof}

\bibliographystyle{amsalpha}
\bibliography{bib}

\smallskip
\Addresses

\end{document}